\documentclass[11pt,twoside]{article}

\let\svthefootnote\thefootnote
\newcommand\blfootnote[1]{%
	\let\thefootnote\relax\footnotetext{#1}%
	\let\thefootnote\svthefootnote%
}

\usepackage{geometry}
\usepackage{amssymb,amsthm}
\usepackage{amsmath}
\numberwithin{equation}{section}

\newcommand{\norm}[1] {\left \| #1 \right \|}
\newcommand{\inclu}[0] {\ar@{^{(}->}}

\newcommand{\cA}{\mathcal{A}}

\newcommand{\EE}{\mathbb{E}}
\newcommand{\trace}[1]{\mathrm{trace}\left(#1\right)}

\newcommand{\RR}{\mathbb{R}}

\newcommand{\cX}{\mathcal{X}}

\newcommand{\cT}{\mathcal{T}}

\newcommand{\cN}{\mathcal{N}}

\newcommand{\abs}[1]{\left| #1 \right|}

\newcommand{\ceil}[1]{\left \lceil #1 \right \rceil }
\newcommand{\floor}[1]{\left \lfloor #1 \right \rfloor }

\newcommand{\argmin}{\operatornamewithlimits{argmin}}

\newcommand{\argmax}{\text{argmax}}

\newcommand{\dotp}[1]{\left\langle #1\right\rangle}

\newtheorem{theorem}{Theorem}[section]

\newtheorem{definition}[theorem]{Definition}
\newtheorem{proposition}[theorem]{Proposition}
\newtheorem{lemma}[theorem]{Lemma}
\newtheorem{corollary}[theorem]{Corollary}

\newtheorem{assumption}{Assumption}

\newcommand{\cF}{\mathcal{F}}

\newtheorem{remark}{Remark}

\newcommand{\paren}[1]{ \left( #1 \right) }
\theoremstyle{remark}

\usepackage{mathtools}

\usepackage{hyperref}
\usepackage{color}
\usepackage{enumitem} 
\usepackage{algorithm}
\usepackage{algpseudocode}
\usepackage{caption}
\usepackage{subcaption}
\usepackage{adjustbox}  %

\theoremstyle{plain}
\newtheorem*{assumptionBprime}{\textbf{Assumption B\textquotesingle}}

\newcommand{\supp}{\mathrm{supp}}
\newcommand{\cov}{\mathrm{cov}}
\newcommand{\liph}{L_H}

\newcommand{\algname}{\mathrm{VISOR}}

\newcommand{\cZ}{\mathcal{Z}}
\newcommand{\ch}{C_H}

\newcommand{\nclass}{\cN}

\newcommand{\xSAA}{\widehat x^{(\texttt{SAA})}}
\newcommand{\cp}{\widehat x_n^{(\texttt{cp})}}
\newcommand{\rpj}{\widehat x_n^{(\texttt{RPJ})}}
\newcommand{\saa}{\widehat x_n^{(\texttt{SAA})}}
\newcommand{\tmu}{\tilde \mu}
\newcommand{\lmin}{\omega}
\begin{document}
	
	\begin{center}
		
		{\bf{\Large{Instance-optimal stochastic convex optimization: Can we improve upon sample-average and robust stochastic approximation?}}}
		
		\vspace*{.2in}
		
		{\large{
				\begin{tabular}{ccc}
					Liwei Jiang$^\star$ and Ashwin Pananjady$^\dagger$
				\end{tabular}
		}}
		\vspace*{.2in}
		
		\begin{tabular}{c}
			$^\star$Edwardson School of Industrial Engineering, Purdue University \\
			$^\dagger$H. Milton Stewart School of Industrial and Systems Engineering \& \\ School of Electrical and Computer Engineering, 
			Georgia Institute of Technology
		\end{tabular}
		
		\vspace*{.2in}

		\today
		
		\vspace*{.1in}

		\abstract{\blfootnote{Part of this work was performed when the first author was at Georgia Tech.} We study the unconstrained minimization of a smooth and strongly convex population loss function under a stochastic oracle that introduces both additive and multiplicative noise; this is a canonical and widely-studied setting that arises across operations research, signal processing, and machine learning. 
			We begin by showing that standard approaches such as sample average approximation and robust (or averaged) stochastic approximation can lead to suboptimal --- and in some cases arbitrarily poor --- performance with realistic finite sample sizes. In contrast, we demonstrate that a carefully designed variance reduction strategy, which we term $\algname$ for short, can significantly outperform these approaches while using the same sample size. Our upper bounds are complemented by finite-sample, information-theoretic local minimax lower bounds, which highlight fundamental, instance-dependent factors that govern the performance of any estimator. Taken together, these results demonstrate that an accelerated variant of $\algname$ is instance-optimal, achieving the best possible sample complexity up to logarithmic factors while also attaining optimal oracle complexity. We apply our theory to generalized linear models and improve upon classical results. In particular, we obtain the best-known non-asymptotic, instance-dependent generalization error bounds for stochastic methods, even in linear regression. 
            }
	\end{center}
\section{Introduction} \label{sec:intro}
Consider the canonical stochastic optimization problem of minimizing a smooth and strongly convex population objective function $F_{f,P}: \RR^d \to \RR$, given by
	\begin{align}\label{eqn:objective}
	F_{f,P}(x) := \EE_{z \sim P}[f(x,z)].
	\end{align}
	Here we have an underlying but unknown distribution $P$, and $f\colon \RR^d \times \cZ \rightarrow \RR$ is a ``sample-wise" objective function.
	Numerous problems in optimization, statistics, and machine learning can be modeled as minimizing a population objective as in Eq.~\eqref{eqn:objective}, but with access only to $n$ noisy functions $\{f(\cdot; z_1), \ldots, f(\cdot; z_n)\}_{i = 1}^n$, where $z_1, z_2,\ldots, z_n \in \cZ$ denote i.i.d. observations drawn from $P$. A long line of literature has focused on understanding the fundamental limits of this problem, and on analyzing how the behavior of various canonical algorithms compares with these limits. 
    Classically, there have been two approaches to this family of questions.
    
    The first approach uses the worst-case risk over a class of problem instances as the  measure by which to compare algorithms and establish lower bounds~\cite{complexity,agarwal2009information}. 
    A canonical class of such problems (see~\cite{ghadimi2012optimal}) is those in which the population objective is smooth, $\mu$-strongly convex, and the stochastic gradients have bounded variance at any point:\footnote{One can add additional constraints such as Lipschitz continuity of gradients and function values, but the same information-theoretic minimax lower bound still holds.} 
    \begin{align*}
        \mathcal{P}(\mu, \sigma) := \left\{ (f,P) \left| 
	\begin{aligned}
	&\text{ $ F(y) - F(x) - \dotp{\nabla F(x), y-x}\ge \frac{\mu}{2}\|y-x\|_2^2$,  \ for any $x,y \in \RR^d$} \\
	& \text{ and \ $\EE[\|\nabla f(x,\xi) - \nabla F(x)\|_2^2] \le \sigma^2$, \ for all $x \in \RR^d$}
	\end{aligned}
	\right.\right\}.
        \end{align*}
        Then a standard application of Fano's method shows that for any estimator $\widehat{x}_n$ based on i.i.d. samples $\{z_i\}_{i=1}^{n}$, we must have
        the worst-case risk lower bounded as 
        \begin{align} \label{eq:worst-case-lb}
         \sup_{(f,p) \in \mathcal{P}(\mu,\sigma)} \EE[\|\widehat x_n - x^\star(F_{f,P})\|_2^2] \ge \frac{\sigma^2}{n\mu^2}.
        \end{align}
    One can ask if the lower bound~\eqref{eq:worst-case-lb} is achieved by any estimator $\widehat{x}_n$, and indeed,
    over this class of stochastic optimization problems, it is known that minimax rate-optimal estimators can be constructed with stochastic first-order information~\cite{ghadimi2012optimal}. These estimators achieve the worst-case lower bound~\eqref{eq:worst-case-lb} up to universal constant factors and only require access to $\nabla f(x_t, z_t)$ at carefully chosen query points $\{x_t\}_{t = 1}^n$. 
    
    On the one hand, the above results provide a complete picture of the worst-case risk in this problem. Furthermore, complexity characterizations based on the minimax risk are well-defined for every finite $n$ and possess many appealing properties from the perspective of statistical decision theory~\cite{le2012asymptotic, tsybakov2008nonparametric}. On the other hand, assessing the sample complexity of algorithms purely in terms of their worst-case risk is pessimistic. Indeed, the worst-case optimality of an algorithm over the global problem class $\mathcal{P}(\mu, \sigma)$ does not necessarily imply that this algorithm is able to leverage geometric properties of problem instances that are \emph{close to} $(f, P)$. In particular, the worst-case optimality of a stochastic first-order method does not imply anything about its \emph{adaptivity} to any structure present in the instance $(f, P)$.
    
    As a remedy to the minimax approach, the second approach to understanding fundamental limits of stochastic optimization better captures the \emph{local} desideratum alluded to above.
    A typical result of this form, which is a specialization of the one proved in~\cite{duchi2021asymptotic}, takes the following form: For a population objective $F$ that is smooth and strongly convex with minimizer $x^\star$, define the Gaussian random vector 
	\begin{align} \label{eq:Gaussian-covariance}
Z\sim N\left(0\;,\; \underbrace{\nabla^2 F(x^\star)^{-1} \cdot \cov_{z_i\sim P}(\nabla f(x^\star, z_i)) \cdot  \nabla^2 F(x^\star)^{-1}}_{\Lambda}\right).
\end{align}
	Then any estimator $\widehat x_n$ based on $n$ i.i.d. samples satisfies
	\begin{align}\label{eqn:asymp_lb}
	\liminf_{c \rightarrow \infty}\;\; \liminf_{n\rightarrow \infty} \sup_{\tilde P \colon \mathrm{D_{KL}}(\tilde P \| P) \le \frac{c}{n}} \EE_{z_i\overset{\mathrm{iid}}{\sim} \tilde P}[ \| \sqrt{n}(\widehat x_n - x^\star) \|^2_2 ] \ge \EE[\| Z \|^2_2] = \trace{\Lambda},
	\end{align}
	where $\mathrm{D_{KL}}(\cdot \| \cdot)$ denotes the Kullback–Leibler (KL) divergence between two probability distributions and the limit over $c$ is taken for technical reasons. 
	The lower bound~\eqref{eqn:asymp_lb} implies that in the asymptotic regime when $n \rightarrow \infty$, the estimation error $\|\sqrt{n} (\widehat x_n - x^\star)\|^2$ is bounded below by $\trace{\Lambda}$, where $\Lambda$  is the so-called \emph{inverse Fisher information matrix}, determined by the interaction between the problem's geometry (captured by the population Hessian at $x^\star$) and the noise characteristics (captured by the covariance of the sample gradient at $x^\star$). In that sense, this characterization is \emph{instance-dependent}, since the error has explicit dependence on $(f, P)$ and may be small when the instance has favorable geometry with a small value of $\trace{\Lambda}$. %
	
	As before, one can again ask if the lower bound~\eqref{eqn:asymp_lb} is achieved by some estimator $\widehat{x}_n$. The first candidate for such an estimator is sample average approximation (SAA) or empirical risk minimization, which selects an estimator
	\begin{align*}
	\saa \in \argmin_{x \in \RR^d} \frac{1}{n} \sum_{i=1}^{n} f(x, z_i).
	\end{align*} 
	Classical results (eg~\cite[Theorem 3.3]{shapiro1989asymptotic}) show that under mild regularity conditions, SAA exhibits asymptotic normality; recalling the matrix $\Lambda$ from before, we have the weak convergence property
\begin{align}\label{eqn:saa_normality}
	\sqrt{n} (\saa - x^\star) \overset{w}{\longrightarrow} Z \sim N(0,\Lambda).
	\end{align}
	However, since the $\ell_2$-norm is unbounded, this does not imply that the lower bound~\eqref{eqn:asymp_lb} is attained, and SAA can incur infinite $\ell_2^2$ risk (see Section~\ref{sec:SAA_fail}).  
    Having said that, the instance-dependent lower bound~\eqref{eqn:asymp_lb} is indeed achievable by a different but also classical estimator. Consider the average 
	$$\rpj = \frac{1}{n} \sum_{i=1}^{n}x_k$$ 
	of the iterates $\{ x_k \}_{k \geq 1}$ of the stochastic gradient method. Ruppert~\cite{ruppert1988efficient} and Polyak and Juditsky~\cite{polyak1992acceleration} show that with suitable stepsize choices and under mild conditions, this estimator also exhibits asymptotic normality, with
\begin{align}\label{eqn:rpj_normality}
\sqrt{n} (\rpj - x^\star) \overset{w}{\longrightarrow} Z \sim N(0,\Lambda).
\end{align}
	 Furthermore, a direct application of~\cite[Theorem 3]{moulines2011non} by Bach and Moulines implies that
	\begin{align}\label{eqn:upper_bound_RPJ}
	\lim_{n\rightarrow \infty}\EE_{z_i\overset{\mathrm{iid}}{\sim}  P}[\|\sqrt{n}(\rpj - x^\star)\|_2^2] \rightarrow \trace{\Lambda},
	\end{align}
	which implies that $\rpj$ is an asymptotically instance-optimal algorithm for $\ell_2^2$ risk. Since the average iterate is more numerically stable to hyperparameter choices than the last iterate, this method is often referred to as \emph{robust stochastic approximation (SA)}~\cite{nemirovski2009robust}.

    \subsection{An illustrative experiment: Is optimality achieved in practice? }

    While the state of affairs described above suggests that we have succeeded in developing instance-optimal estimators for stochastic optimization, the situation is significantly more nuanced in practice. Consider for instance the second notion mentioned above, of local asymptotic optimality. 
    Both the instance-dependent lower bound~\eqref{eqn:asymp_lb} and the upper bound~\eqref{eqn:upper_bound_RPJ} are valid only as the sample size tends to infinity, and may only be meaningful for very large (and impractical) $n$. In practice, we are in the non-asymptotic or finite-sample regime, in which questions surrounding instance-dependent optimality ought to take a different flavor. 
    Concretely, we might ask if and when it is possible to achieve an instance-dependent risk of the order $\trace{\Lambda}$ when $n$ is finite. In particular, how large must $n$ be for any algorithm to exhibit such behavior, and does the asymptotically optimal estimator $\rpj$ perform well in the finite-sample regime?

\begin{figure}[h!]
	\centering
	\begin{minipage}{\textwidth}
		\centering
		\includegraphics[width=0.8\linewidth]{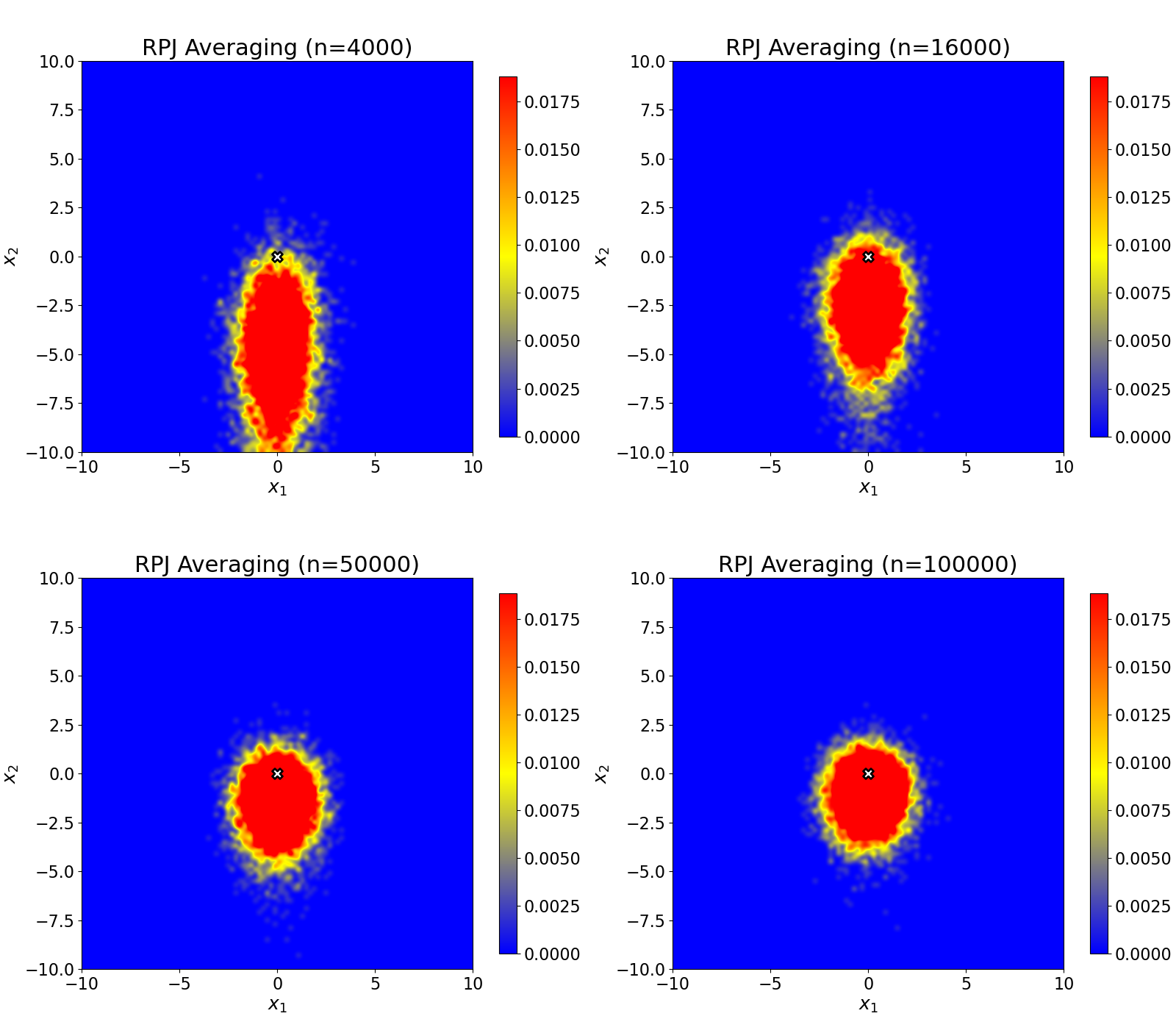}
		\caption{\small Heat maps of $\sqrt{n}(\rpj - x^\star)$ for different $n$. We always initialize the algorithm at the origin (initial distance to minimizer is $\sqrt{2}$) and each heatmap is generated over $10{,}000$ trials. \vspace{-3mm}}
		\label{fig:polyak_averaging_asymp}
	\end{minipage}
\end{figure}

    To obtain answers to these questions in a concrete example (to be more extensively examined in Section~\ref{sec:suboptimality_of_SA}) consider the following quadratic optimization problem parameterized by $\zeta \ge 1$.
	Define the matrix-vector pair
	\begin{align} \label{eqn:example_break_robust_SA}
	A = \begin{bmatrix}
	\zeta^2 & 0 \\
	0 & 1
	\end{bmatrix} \quad \text{and}  \qquad b = \begin{bmatrix}
	-\zeta^2 \\
	-1
	\end{bmatrix}.
	\end{align}
	For this family of $(A, b)$ pairs, suppose our goal is to minimize the function $F(x) = \frac{1}{2} x^\top A x + b^\top x$  using only i.i.d. samples $A_i$ of $A$ and $b_i$ of $b$, where 
	\begin{itemize}
		\item $A_i = A + \begin{bmatrix}
		z_i &  -z_i \\
		-z_i & z_i
		\end{bmatrix}$ and $z_i$ are i.i.d. RVs taking values $\zeta$ and $-\zeta$ with probability $1/2$ each;
		\item $b_i = b + \eta_i$ where $\eta_i \sim N\left(\begin{bmatrix}
		0\\
		0
		\end{bmatrix}, \begin{bmatrix}
		\zeta^4 & 0\\
		0 & 1
		\end{bmatrix}\right)$.
	\end{itemize}
A straightforward calculation shows that for any $\zeta \ge 1$, we have $ x^\star = \begin{bmatrix}
	1\\
	1
	\end{bmatrix}$ uniformly. Moreover, the limiting covariance (see Eq.~\eqref{eq:Gaussian-covariance}) is given by $\Lambda = \begin{bmatrix}
	1 & 0\\
	0&  1
	\end{bmatrix}$ uniformly. 
Therefore, the results discussed above yield that asymptotically, we have 
\begin{align*}
\sqrt{n}(\rpj - x^\star) \overset{w}{\longrightarrow} N\left(\begin{bmatrix}
0\\
0
\end{bmatrix}, \begin{bmatrix}
1 & 0 \\
0 & 1
\end{bmatrix}\right),
\end{align*}
i.e., the scaled error vector of averaged stochastic approximation converges to a standard Gaussian.

In Figure~\ref{fig:polyak_averaging_asymp}, we run simulations of the stochastic approximation algorithm in this problem for various values of $n$. Even for the moderate choice $\zeta^2 = 20$, we see that the rescaled error $\sqrt{n}(\rpj - x^\star)$ only begins to resemble a standard Gaussian after approximately $n=10^6$ samples, which is a very large sample size for a $2$-dimensional, moderately-conditioned problem. Before this, the error distribution is skewed, suggesting a much larger $\ell_2^2$ error than is predicted asymptotically.

To further probe this phenomenon, we run our simulation for a sequence of $\zeta$ values, choosing a problem-dependent sample size $n(\zeta) = 200\zeta^2$ for each such simulation. Assuming for the moment that this large a sample size is sufficient for \emph{some} estimator to attain the instance-optimal $\ell_2^2$ error, we should hope that the rescaled error has distribution resembling a standard Gaussian.
However, we see from Figure~\ref{fig:polyak_averaging_different_zeta} that this is not borne out in practice -- the error gets worse as $\zeta$ gets larger.

\begin{figure}[h!]
	\centering
	\begin{minipage}{\textwidth}
		\centering
	\includegraphics[width=\textwidth]{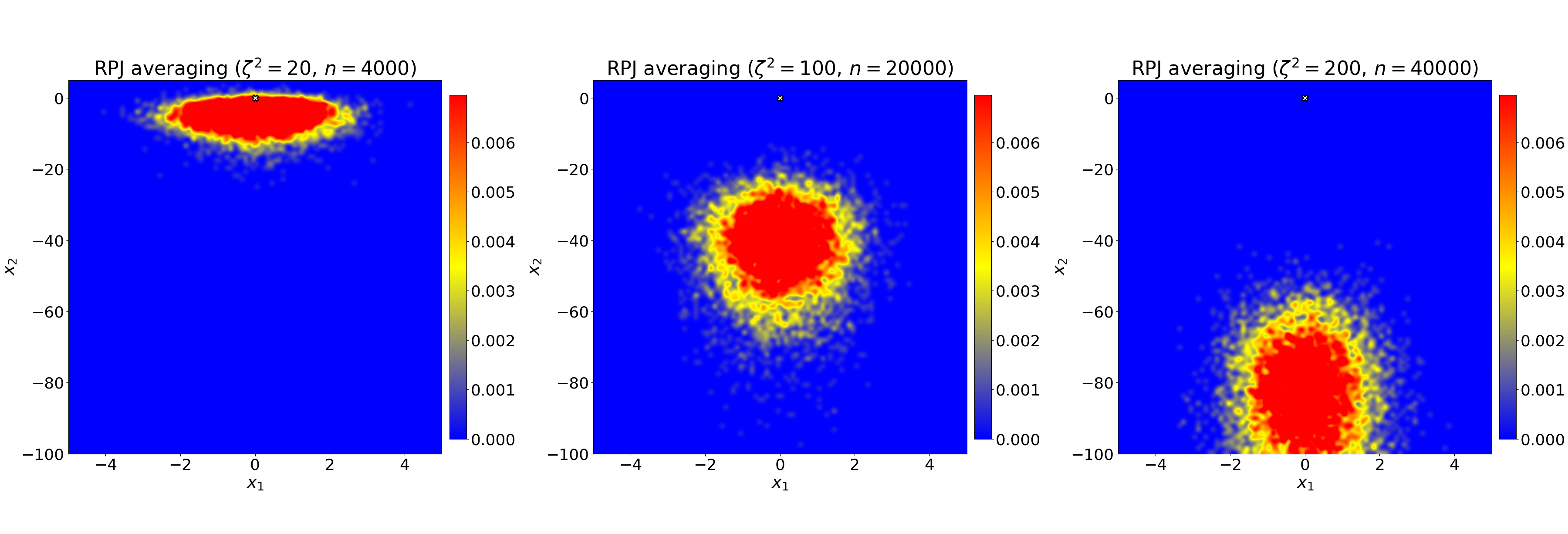}
	\caption{\small Heat maps of $\sqrt{n}(\rpj - x^\star)$ for different $\zeta^2$ and sample size $n = 200\zeta^2$. We always initialize the algorithm at the origin (initial distance to minimizer is $\sqrt{2}$) and each heatmap is generated over $10,000$ trials. (Note that $x_1$ and $x_2$ have different scales in the above plots.)}
	\label{fig:polyak_averaging_different_zeta}
\end{minipage}
\end{figure}

This poor finite-sample performance of the asymptotically optimal estimator $\rpj$ raises two important questions: Is performance poor because the given problem is information-theoretically challenging for our illustrated values of $n$ (meaning that no algorithm can significantly improve upon the performance of $\rpj$)? Or is a better performance attainable in this finite-sample regime, but by a different estimator?
Motivated by these observations, we pose the following questions for smooth and strongly convex stochastic optimization:
\begin{center}
	\begin{enumerate}
		\item[Q1.] How can we characterize the instance-dependent hardness of problem~\eqref{eqn:objective} for a finite, fixed sample size $n$? 
        \item[Q2.] Which algorithms can achieve optimality in this finite-sample regime?
	\end{enumerate}
    \end{center}

\subsection{Contributions and organization}

	Our main contribution is to answer the above questions. We prove the following results:
	\begin{enumerate}
		\item  \textbf{Non-asymptotic lower bound and suboptimality of $\rpj$ and $\saa$.} We  establish a lower bound that characterizes the instance-dependent hardness of smooth and strongly convex stochastic optimization problems for any given sample size $n$. We show that, up to a universal constant factor, the key geometric quantity that governs the local minimax lower bound in $\ell_2^2$ risk is still $\trace{\Lambda}$, as long as the sample size exceeds some explicit, problem-dependent threshold $n_0$. %
        
        When applied to the family of quadratic problems above, our Theorem~\ref{thm: mainlowerboundtheorem_intro} implies that there exist two universal constants $c_1$ and $c_2$ such that 
		for any sample size $n \le c_1\zeta^2$, one should not expect any reasonable algorithm to attain finite error.
		On the other hand, for $n \ge c_1\zeta^2$, the expected rescaled $\ell^2_2$ error $\EE[\|\sqrt{n}(\widehat x_n - x^\star)\|_2^2]$ is at least 
			\begin{align*}
			c_2 \cdot \trace{\Lambda} = 2 c_2,
			\end{align*}
            where the last equality follows because in this family of problems, we have $\Lambda = I$.
		Note that once we have $\Omega(\zeta^2)$ samples, the lower bound is a universal constant independent of $\zeta$. Figure~\ref{fig:polyak_averaging_different_zeta}, which examines this regime, shows that $\rpj$ fails to achieve our non-asymptotic lower bound since the error is sensitive to $\zeta$. 
        Similar issues plague the estimator $\saa$. 
		
    \begin{figure}[b!]
	\centering
	\begin{minipage}{\textwidth}
		\centering
	\includegraphics[width=\textwidth]{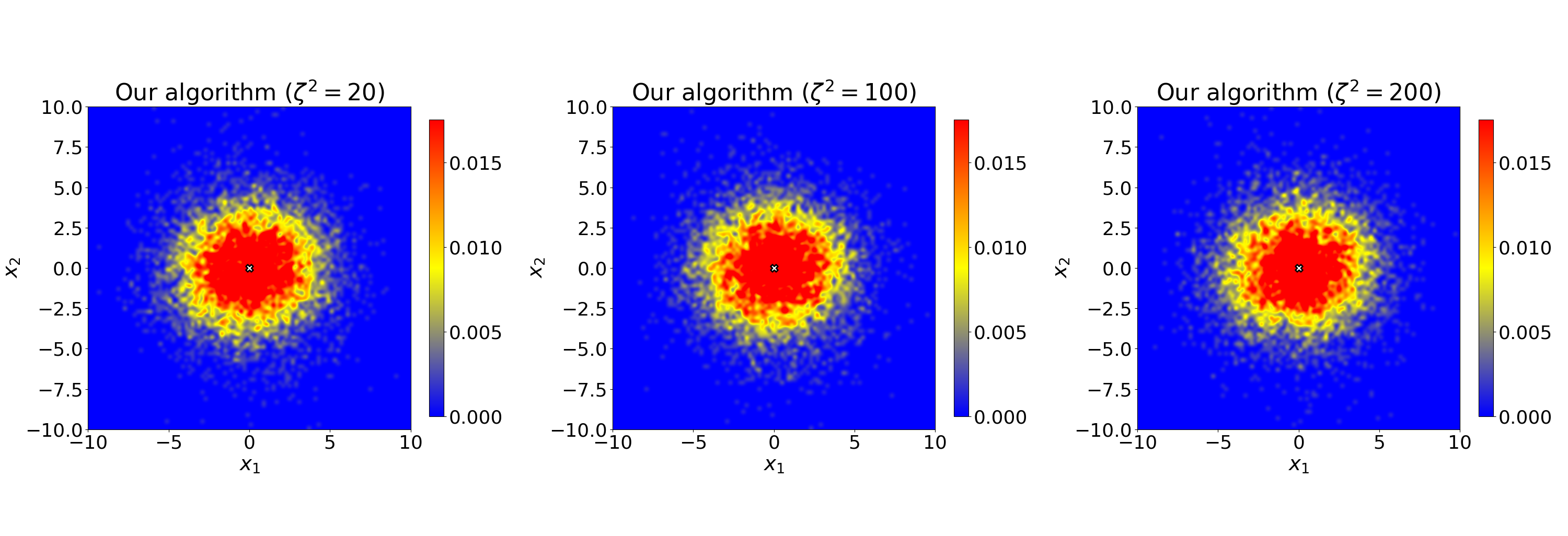}
	\caption{Heat maps of $\sqrt{n}(\widehat x_n - x^\star)$ for our algorithm for different $\zeta^2$ and sample size $n = 200\zeta^2$. For each $\zeta^2$, we always initialize the algorithm at the origin (initial distance to minimizer is $\sqrt{2}$) and perform $10,000$ trials to generate the heat map.}
	\label{fig:our_alg_averaging_different_zeta}
		\end{minipage}
\end{figure}

    		\item \textbf{Instance-optimal (and accelerated) stochastic optimization algorithms.} We propose a simple first-order online algorithm, $\algname$, that incorporates variance reduction techniques wrapped around a (possibly accelerated) stochastic approximation inner loop. For quadratic optimization problems, $\algname$ matches our non-asymptotic, instance-dependent lower bound up to a logarithmic factor. For general non-quadratic problems, it nearly attains the aforementioned local lower bound under an additional assumption on the noise in the problem. Since our method is in general accelerated, it also achieves optimal first-order oracle complexity (which is particularly desirable when the noise level is small). Notably, our convergence guarantees hold for any norm induced by an inner product, not only the standard $\ell_2$ norm --- this feature of our results is not only of general interest but also allows to obtain novel guarantees on the generalization error of our algorithm for least-squares regression (see the point below).
    
    To illustrate, let us again consider the family of quadratic problems above, parameterized by $\zeta$. Applying Theorem~\ref{thm:vr_quadratic}, there exists a universal constant $C > 0$ such that when the sample size satisfies $n \gtrsim_{\log} \zeta^2$, the output $\widehat{x}_n$ of $\algname$ satisfies
    \begin{align*}
    	\EE[\|\sqrt{n}(\widehat x_n - x^\star)\|_2^2] \le C \trace{\Lambda} = 2C,
    \end{align*}
    which matches the lower bound up to a logarithmic factor. Moreover, as shown in Figure~\ref{fig:our_alg_averaging_different_zeta}, for the same sequence of problems as before with sample sizes $n = 200 \zeta^2$, the rescaled error exhibits approximately Gaussian behavior with the correct covariance structure up to a universal constant factor. 
    Contrast this with the behavior of the $\rpj$ estimator in Figure~\ref{fig:polyak_averaging_different_zeta}.

    \item \textbf{Applications to generalized linear models.} We apply our convergence guarantees to generalized linear models~\cite{nelder1972generalized} and obtain nearly instance-optimal and non-asymptotic risk bounds. In particular, we show in Section~\ref{sec:improve_least_squares} that our algorithm improves the best known non-asymptotic guarantees for stochastic methods in least-squares regression~\cite{jain2018accelerating} by a factor of the condition number. 

	\end{enumerate}

	The rest of this paper is organized as follows. Section~\ref{sec:related_work} contains a detailed discussion of related work. In Section~\ref{sec:background}, we set the stage, state key assumptions, and provide concrete examples of problems covered by our theory. Section~\ref{sec:suboptimality} provides a non-asymptotic local minimax lower bound for the class of quadratic optimization problems as well as simple examples where both SAA and averaged SGD fail to match this lower bound.   We present our new algorithm in Section~\ref{sec:better_algorithm} and its convergence guarantees in Section~\ref{sec:overall_complexity}. Our general non-asymptotic lower bounds are presented in Section~\ref{sec:lower_bound}. Conceptually simple and short proofs are presented just after the corresponding statements of results, while the more technical proofs are deferred to the appendix.

\subsection{Related work}\label{sec:related_work}

The literature on statistical analysis in stochastic optimization is vast, and we cannot hope to do justice to it here. We refer the reader to the books~\cite{benveniste2012adaptive,shapiro2021lectures} for classical (and largely asymptotic) results and the book~\cite{lan2020first} for a more modern non-asymptotic treatment. Below, we discuss the results that are most closely related to the focus of our paper, organized under two subheadings.

\paragraph{Non-asymptotic instance-dependent analysis.}
Non-asymptotic and instance-dependent analysis has been a challenging but fruitful program in high-dimensional statistics, and was first carried out under the so-called ``two-point'' framework by~\cite{cai2015framework} for estimation of one-dimensional convex functions.  Unlike classical minimax analysis that considers the worst-case over all functions in a function class, this framework obtains lower bounds for any specific instance by only considering the worst-case risk over that instance and its hardest alternative. While the two-point framework has been applied to many different contexts since, it is insufficient to characterize local complexity beyond the one-dimensional case.
More recently, non-asymptotic and instance-dependent guarantees have been established in multiple dimensions with the goal of matching the asymptotic risk. Settings considered include Markov decision processes~\cite{pananjady2020instance,khamaru2021temporal,khamaru2021instance,li2023accelerated, li2020breaking, li2024settling} and stochastic approximation~\cite{mou2022optimal,haque2023tight}, but these works do not justify whether the asymptotic minimax risk remains the appropriate complexity measure in non-asymptotic settings.
Non-asymptotic lower bounds have also been derived in~\cite{mou2022optimal1, mou2023optimal} for estimation in projected fixed-point equations and Markovian linear stochastic approximation. However, both settings are linear and do not address nonlinear scenarios that form the focus of our work.

\paragraph{Local guarantees for stochastic optimization.} As mentioned in Section~\ref{sec:intro}, the paper~\cite{duchi2021asymptotic} applied H\'ajek and Le Cam's local minimax theory to develop asymptotic local minimax lower bounds for stochastic optimization problems. This asymptotic complexity measure is matched by the Ruppert--Polyak--Juditsky averaging procedure exactly in smooth~\cite{ruppert1988efficient,polyak1992acceleration}, nonsmooth and constrained settings~\cite{davis2024asymptotic}.
Recent work~\cite{khamaru2024stochastic} studies general stochastic constrained convex optimization problems and proposes a non-asymptotic instance-dependent lower bound that is extracted from the proof of the asymptotic local minimax theory. However, it is unclear whether the hardest instances asymptotically are also the hardest instances for each fixed sample size. 
Also inspired by non-asymptotic bounds, the paper~\cite{duchi2016local} applies the two-point lower bound to stochastic convex optimization, characterizing problem difficulty in one-dimensional problems. A recent effort to capture local geometry in optimization problems was also made in the paper~\cite{cutler2024radius}.

More broadly, non-asymptotic guarantees for stochastic optimization problems have been studied extensively in the literature, and we provide a brief overview of results under smooth and strongly convex settings. The optimal algorithm in the classical minimax sense has been studied in~\cite{ghadimi2012optimal,ghadimi2013optimal}, where Ghadimi and Lan propose the AC-SA algorithm and show that its restarted version is worst-case optimal under the assumption that gradient noise has uniformly bounded variance. Moving beyond worst-case characterizations, instance-dependent analysis of stochastic optimization algorithms has also been carried out in several works. The work~\cite{nguyen2019new} proves an instance-dependent rate for iterate convergence of stochastic gradient descent (SGD), but the rate itself is not instance-optimal. Applying variance-reduction techniques, the paper~\cite{nguyen2021inexact} proves an instance-optimal rate on the gradient norm, but it is unclear how the analysis can be extended to iterate convergence without losing instance-optimality. The works~\cite{moulines2011non}, ~\cite{xu2011towards}, and~\cite{gadat2017optimal} provide non-asymptotic analysis for Ruppert--Polyak--Juditsky averaging and the convergence rate of has an instance-dependent leading term that matches the asymptotically optimal rate, and with higher-order term (which controls non-asymptotic performance) taking the form $O(n^{-7/6})$ in the first work and $O(n^{-5/4})$ in the last two works.  The work~\cite{li2022root} proposes the ROOT-SGD algorithm and provides convergence guarantees in terms of gradient norm, distance to solution, and function gap. All the convergence rates have leading terms that match the asymptotically optimal rate, and higher-order terms scale as $O(n^{-3/2})$, improving upon early works. The paper~\cite{frostig2015competing} proposes streaming SVRG and shows that under a self-concordance assumption, its convergence rate on the function gap has a leading term that can be made arbitrarily close to the asymptotic rate of empirical risk minimization. However, in all the works above, the convergence rates have higher-order terms that can be potentially much larger than the leading term, and the number of samples they require for the higher-order terms to be dominated by the optimal leading term can be large (as illustrated in Figure~\ref{fig:polyak_averaging_asymp}).

Besides these papers, an extensive body of work~\cite{defossez2015averaged,dieuleveut2016nonparametric, jain2018parallelizing,jain2018accelerating} studies and obtains instance-dependent guarantees for least-squares regression, which is an important special case of our setting. In particular,~\cite{jain2018accelerating} obtains the best-known sample complexity for stochastic methods.
As alluded to before, many of these algorithms rely on careful forms of variance reduction.
Our algorithm also draws inspiration from SVRG~\cite{johnson2013accelerating} but makes crucial modifications to it. For a review of related works on variance-reduced gradient methods, we refer the reader to a recent survey paper~\cite{gower2020variance}.

\subsection{Notation}\label{sec:notation}
Let $\dotp{\cdot,\cdot}$ denote the dot product in Euclidean space, which induces norm $\|x\|_2= \sqrt{\dotp{x,x}}$.  For $r > 0$ and $x \in \mathbb{R}^d$, we denote by $B_r(x)$ and $\overline{B}_r(x)$ the open and closed Euclidean balls of radius $r$ centered at $x$, respectively. We denote the unit sphere in \(\RR^d\) under the standard Euclidean norm by \(\mathbb{S}^{d-1}\). As previously mentioned, we will work not just with the $\ell_2$ norm but with general Hilbert norms\footnote{Even when we work with general norms, the inner product will always denote the Euclidean dot product.} $\|\cdot\|$. Since we operate in finite-dimensional spaces, any such norm can be written as $\| x \| = \|x\|_Q := \sqrt{\dotp{x,Q x}}$ for some positive definite matrix $Q$. We denote the dual of the norm $\|\cdot\|$ by $\|\cdot \|_*$, i.e.  $\|y\|_* = \sup_{\|x\| \le 1} \dotp{x,y}$.  For any matrix $M \in \RR^{d\times d}$, we use $\|M\|$ to denote the induced operator norm, i.e., $\|M\| = \sup_{\|x\| = 1} \|Mx\|$. In particular, $\|A\|_2$ will denote the spectral norm of a matrix $A$. The notation $\|A\|_{nuc}$ denotes the nuclear norm of $A$, namely, the sum of all its singular values. For any matrix $A\in\RR^{d\times d}$, we let $\det(A)$ denote the determinant of $A$. For a symmetric matrix $A \in \RR^{d\times d}$, we use $\lambda_{\max}(A)$ and $\lambda_{\min}(A)$ to denote its largest and smallest eigenvalue, respectively.  For symmetric matrices \(A,B \in \RR^{d\times d}\), we write \(A \preceq B\) if \(B-A\) is positive semidefinite, and \(A \succeq B\) if \(B \preceq A\). For a random vector $\xi$ with bounded second moment, we denote its covariance matrix by 
$\cov(\xi) = \EE[(\xi - \EE[\xi]) (\xi - \EE[\xi])^\top]$. 
A random variable $X$ is sub-exponential with parameters $(\nu^2, \alpha)$ if for any $t$ such that $|t| \le \frac{1}{\alpha}$, we have
\begin{align*}
	\EE[e^{t (X-\EE[X])}] \le e^{\frac{t^2 \nu^2}{2}}.
\end{align*}
  Define the Orlicz norms
	\begin{align*}
		\|X\|_{\psi_1} := \inf\left\{ t>0 \;\middle|\; \EE\!\left[\exp\!\left(\frac{|X|}{t}\right)\right] \le 2 \right\} \text{ and }
		\|X\|_{\psi_2} := \inf\left\{ t>0 \;\middle|\; \EE\!\left[\exp\!\left(\frac{X^2}{t^2}\right)\right] \le 2 \right\}.
	\end{align*}
	It is well-known that $X$ is sub-exponential iff $\|X\|_{\psi_1}<\infty$. We say $X$ is sub-Gaussian if $\|X\|_{\psi_2}<\infty$.
    
For any continuously differentiable function $h:\RR^d \rightarrow \RR$, we let $\nabla h(x) \in \mathbb{R}^d$ denote the gradient of $h$ evaluated at $x$. When $h$ is twice differentiable, we denote its Hessian at $x \in \RR^d$ by $\nabla^2 h(x) \in \mathbb{R}^{d \times d}$. For a general smooth map $G:\RR^d \rightarrow \RR^m$, we denote its Jacobian at $x \in \RR^d$ by $\nabla G(x)  \in \RR^{m\times d}$, which can also be viewed as a linear map from $\RR^d$ to $\RR^m$. For two maps $F$ and $G$, we write $F\equiv G$ if they are identical. For two sequences of nonnegative reals $\{f_n\}_{n\ge 1}$ and $\{g_n\}_{n\ge 1}$, we use $f_n \lesssim g_n$ to indicate that there is a universal positive constant $C$ such that $f_n \le C g_n$ for all $n \ge 1$. We use $f_n \lesssim_{\log} g_n$ to indicate that there is a universal positive constant $c$ such that $f_n \lesssim  g_n \log^c (en) $. The relation $f_n \gtrsim g_n$ (resp. $f_n \gtrsim_{\log} g_n$) indicates that $g_n \lesssim f_n$ (resp. $g_n \lesssim_{\log} f_n$). We also use the standard order notation $f_n = O(g_n)$ to indicate that $f_n \lesssim g_n$ and $f_n = \tilde O(g_n)$ to indicate that $f_n \lesssim_{\log} g_n $. We say that $f_n = \Omega(g_n)$ (resp. $f_n = \tilde \Omega(g_n)$) if $g_n = O(f_n)$ (resp. $g_n = \tilde O(f_n)$).	For any \(x \in \RR\), \(\ceil{x}\) and \(\floor{x}\) denote the smallest integer greater than or equal to \(x\) and the largest integer less than or equal to \(x\), respectively.

\section{Formal setup and examples}\label{sec:background}

In referring to the population objective $F_{f,P}$ in Eq.~\eqref{eqn:objective}, we drop the subscripts $(f,P)$ when these are clear from context. We denote the minimizer of $F$ by $x^\star(F)$ if it is unique, and drop the parentheses when $F$ is clear from context. When it is not clear if there is a unique minimizer, we use $\argmin F$ to denote the set of minimizers. For the sample-wise functions, we use $\nabla f(\cdot,z)$ to denote the gradient of $f$ with respect to its first argument  whenever this is well-defined.

Recall that a function $h$ is $\mu$-strongly convex and $L$-smooth with respect to a norm $\| \cdot \|$ if $h$ is differentiable and
\begin{align*}
	\frac{\mu}{2}\|y-x\|^2 \le h(y) - h(x) - \dotp{\nabla h(x), y-x} \le \frac{L}{2} \|y-x\|^2
\end{align*}
for all $x, y$ in its domain.
We focus on the smooth and strongly convex setting and begin with the following assumption on the population objective function.
\begin{assumption}\label{assum: non-quadratic}
	The population objective function $F$ is $\mu$-strongly convex and $L$-smooth on $\mathbb{R}^d$ with respect to the norm $\| \cdot \|$. We denote its minimizer by $x^\star$. Additionally, $F$ has $\liph$-Lipschitz Hessian in an instance-specific norm, meaning that for any $x, x' \in \RR^d$,
	\begin{align}\label{eqn:lip_hessian}
		\|\nabla^2 F(x^\star) ^{-1}(\nabla^2 F(x) - \nabla^2 F(x'))\| \le \liph \|x - x'\|.
	\end{align}
	We further define $\lmin := \inf_{\|v\| =1} \| \nabla^2 F(x^\star) v \|_*$, so  that $\|\nabla^2 F(x^\star) v\|_* \ge \lmin \|v\|$ for any $v \in \RR^d$.
\end{assumption}
Note that the Lipschitz Hessian condition~\eqref{eqn:lip_hessian} is stated with respect to a norm scaled by $\nabla^2 F(x^\star)^{-1}$. This instance-dependent condition naturally reflects the local geometry at the optimal solution: directions associated with larger curvature at $x^\star$ allow for greater variations in the Hessian. Note also that $\lmin \ge \mu$ (see Lemma~\ref{lem:general_lambda_min}). Next, we state our regularity condition on the stochastic noise. 
\begin{assumption}\label{assum:vr_assum}
	For every fixed sample $z$, the function $f(\cdot, z)$ is differentiable.  
	The stochastic gradient $\nabla f(x, z)$  has a finite second moment for all $x \in \RR^d$. Moreover, there exists a constant $\zeta \ge 0$ such that, for all $x, x' \in \RR^d$,
	\begin{align}\label{eqn:vr_assum}
		\EE_{z \sim P}\!\left[\big\|(\nabla f(x, z) - \nabla F(x)) - (\nabla f(x', z) - \nabla F(x'))\big\|_*^2\right]
		\le \zeta^2 \|x - x'\|^2,
	\end{align}
	where $F(x) := \EE_{z \sim P}[f(x, z)]$.  
	We denote the covariance matrix of the stochastic gradient at the optimum by
	\[
	\Sigma := \cov_{z \sim P}(\nabla f(x^\star, z)).
	\]
\end{assumption}
Assumption~\ref{assum:vr_assum} has appeared in~\cite{li2022root} in the optimization setting and in~\cite{mou2023optimal,li2023accelerated} in other related settings. We point out that our assumption \emph{does not} require that each sample objective is $L$-smooth and it is strictly weaker than the following popular assumption in stochastic optimization (e.g. the papers~\cite{moulines2011non, needell2014stochastic, frostig2015competing, nguyen2019new,khamaru2024stochastic}, which use this stronger assumption with $\| \cdot \| = \| \cdot \|_2$):
\begin{assumptionBprime}\label{assum:vr_assum_standard} 
	For almost every $z$, the function $f(\cdot, z)$ is differentiable.   The noisy gradient $\nabla f(x,z)$ has finite second moment for any $x \in \RR^d$.  In addition, there exists $\zeta' \ge 0$ such that 
	for any $x, x' \in \RR^d$:
	\begin{align}\label{eqn:vr_assum_as}
		\|\nabla f(x,z)  -\nabla f(x',z)\|_* \le \zeta' \|x - x'\| \quad \text{almost surely.}
	\end{align} 
\end{assumptionBprime}
It is straightforward to verify that the almost sure Lipschitz gradient assumption~\eqref{eqn:vr_assum_as} implies that Assumption~\ref{assum:vr_assum} holds with $\zeta = 2\zeta'$. 
We now present some examples in which Assumptions~\ref{assum: non-quadratic} and~\ref{assum:vr_assum} hold.

\paragraph{Example 1: Quadratic optimization.}
Suppose that we want to optimize the function 
\begin{align*}
	F(x) = \frac{1}{2} x^\top A x + b^\top x,
\end{align*}
where $A \in \RR^{d\times d}$ is symmetric positive definite and $b \in \RR^d$ is a constant vector. Instead of accessing $A$ and $b$ directly, we only observe i.i.d. samples $(A_i, b_i) \sim P$ such that 
\begin{align*}
	\EE[A_i] = A \qquad \text{and} \qquad \EE[b_i] = b.
\end{align*}
We can define the sample objective function $f(x,\tilde A,\tilde b) = \frac{1}{2} x^\top \tilde A x + \tilde b^\top x$ and model the task as a stochastic optimization problem. When $\|\cdot\|$ is the $\ell_2$ norm, 
it is straightforward to verify that Assumption~\ref{assum: non-quadratic} holds with $\mu = \lambda_{\min}(A)$, $L = \lambda_{\max}(A)$, and $\liph = 0$. On the other hand, when $\|\cdot\| = \| \cdot \|_A$ is induced by the Hessian matrix $A$, we see that Assumption~\ref{assum: non-quadratic} holds with $\mu =L = 1$ and $\liph = 0$.

In general, Assumption~\ref{assum:vr_assum} holds with parameter $\zeta$ if $A_i$ is symmetric, $A_i x + b_i$ has finite second moment for any $x \in \RR^d$ and 
\begin{align*}
	\sup_{\|v\| = 1} \EE[\|(A_i - A) v\|_*^2] \le \zeta^2.
\end{align*}
In Section~\ref{sec:suboptimality}, we will pay particular attention to this family of problems since popular algorithms already exhibit suboptimality on such simple instances. \hfill $\clubsuit$

\paragraph{Example 2: Least-squares regression.}
Consider the following problem
\begin{align*}
	\min_{x\in \RR^d} F(x), \quad \text{where} \quad F(x) =  \frac{1}{2}  \cdot \EE_{(\xi,y) \sim P}[(y- \dotp{x,\xi})^2].
\end{align*}
Here, the sample objective function for a given $(\xi,y)$ is $f(x,\xi,y) = \frac{1}{2} (y - \dotp{x,\xi})^2$.
Let $H$ denote the second moment matrix of $\xi$, which is also the Hessian of $F$, i.e., 	$H = \EE_{(\xi,y) \sim P}[ \xi \xi^\top] = \nabla^2 F(x)$.	Suppose that $H \succ 0$ and that each coordinate of $\xi$ has finite second and fourth moments. 
These assumptions are standard and also appear in the literature~\cite[Section 2.1]{jain2018accelerating}. 

Since $H$ is positive definite, $F$ is a strongly convex quadratic function, and we denote its unique minimizer by $x^\star$. For a sample $(\xi,y)\sim P$, we denote the noise as $\epsilon = y - \dotp{\xi, x^\star}$. By the optimality conditions at $x^\star$, we have $\EE[\epsilon \xi] = 0$ and 
\begin{align*}
\Sigma = \cov_{(\xi,y)\sim P} (\nabla f(x^\star, \xi,y)) = \EE[\epsilon^2 \xi \xi^\top ].
\end{align*} Define the statistical condition number $\tilde \kappa$ as the smallest non-negative number such that
\begin{align}\label{eqn:stochastic_cond_number}
\EE[\|\xi\|_{H^{-1}}^2 \xi \xi^\top ]  \preceq  \tilde \kappa H.
\end{align}
Now consider the case that $\|\cdot\|$ is induced by the Hessian matrix $H$. Then it can be verified that Assumption~\ref{assum: non-quadratic} holds with $\mu =L = 1$ and $\liph = 0$. Moreover, for any $x,x'\in \RR^d$, we have
\begin{align*}
\EE[\|\nabla f(x,\xi,y) - \nabla f(x', \xi,y)\|_*^2] &= \EE[\|\nabla f(x,\xi,y) - \nabla f(x', \xi,y)\|_{H^{-1}}^2] \\
&= \EE[\|\dotp{x-x', \xi} \xi\|_{H^{-1}}^2]\\
&= (x-x')^\top \EE[\|\xi\|_{H^{-1}}^2 \xi \xi^\top ] (x-x')\\
&\le \tilde \kappa \|x-x'\|_H^2.
\end{align*} 
Therefore, Assumption~\ref{assum:vr_assum} holds with $\zeta^2 = \tilde \kappa$.
\hfill $\clubsuit$

\paragraph{Example 3: Regularized GLM.}
Let $(\xi_i,y_i)\in\RR^d\times\mathcal Y$ be i.i.d. samples from a joint distribution $P$ given by a generalized linear model (GLM).
In particular, assume $y_i\mid \xi_i$ follows an exponential-family distribution with natural (canonical) parameter $\theta_i$ and dispersion $r(\phi)>0$:
$$
p(y_i\mid \theta_i)
=\exp\left\{\frac{y_i\theta_i-u(\theta_i)}{r(\phi)}+s(y_i,\phi)\right\}.
$$
With the \emph{canonical link}, the linear predictor equals the natural parameter, i.e.,  $\theta_i=\langle x,\xi_i\rangle.$ Note that 
\begin{align*}
- \log p(y_i \mid \theta_i, \phi) = \frac{u(\theta_i) - y_i\theta_i}{r(\phi)} - s(y_i,\phi).
\end{align*} 
We consider the following regularized population risk
\begin{align}\label{eqn:generalized_linear_model}
\min_{x\in\RR^d}\; F(x)
:= \EE_{(\xi,y)\sim P}[\ell(x,\xi,y)] +  \frac{\lambda}{2}\|x\|_2^2 = \EE_{(\xi,y)\sim P}\left[u(\dotp{x,\xi}) - y \dotp{x,\xi}\right] + \frac{\lambda}{2}\|x\|_2^2,
\end{align}
where $\ell(x,\xi, y)$ is the per-sample negative log-likelihood function after rescaling. 
We consider the case that $\|\cdot\|$ is the $\ell_2$ norm and place the following assumptions:
\begin{assumption}\label{assum:generalized_linear_model}
Suppose that the following is true for the generalized linear model:
\begin{enumerate}
	\item $u$ is $C^2$-smooth, $\gamma$-strongly convex (where $\gamma$ = 0 means $u$ is convex), and $u'$ and $u''$ are $L_1$ and $L_2$ Lipschitz continuous, respectively. Denote the unique minimizer of $F$ by $x^\star$.
	\item  $\xi$ is a sub-Gaussian vector with parameter $\sigma^2$, i.e., for any unit vector $v \in \mathbb{S}^{d-1}$, the  norm $\|\dotp{v,\xi}\|_{\psi_2} \le \sigma$.  Moreover, $\EE[\xi \xi^\top] \succeq \sigma_{\min} I$.
	\item We have $\|y\|_{\psi_2} \le \sigma_y$ and $ \|u'(\dotp{x^\star, \xi})\|_{\psi_2} \le \sigma_\star$.
\end{enumerate} 
\end{assumption}
A notable special case of the GLM is logistic regression, in which we have the exponential family with
\[
r(\phi) = 1,\qquad
\theta_i = \langle x,\xi_i\rangle,\qquad
u(\theta) = \log\bigl(1 + e^{\theta}\bigr),\qquad
s(y,\phi) = 0.
\]
In particular, each $y_i$ takes values in the set $\{0, 1\}$ and the conditional density can be written as
\[
p(y_i \mid \theta_i)
= \exp\Bigl\{ y_i \theta_i - \log\bigl(1 + e^{\theta_i}\bigr) \Bigr\},
\qquad \theta_i = \langle x,\xi_i\rangle.
\]
The corresponding (unscaled) negative log-likelihood loss for a single sample is
\[
\ell(x,\xi_i,y_i)
= u(\langle x,\xi_i\rangle) - y_i \langle x,\xi_i\rangle
= \log\bigl(1 + e^{\langle x,\xi_i\rangle}\bigr) - y_i \langle x,\xi_i\rangle.
\]
Thus, the regularized population risk in~\eqref{eqn:generalized_linear_model}
specializes to
\[
F(x)
= \EE_{(\xi,y)\sim P}\!\left[\log\bigl(1 + e^{\langle x,\xi\rangle}\bigr)
- y \langle x,\xi\rangle\right]
+ \frac{\lambda}{2}\,\|x\|_2^2,
\]
which is the standard $\ell_2$-regularized logistic regression objective. If the distribution of each feature vector $\xi$ is standard Gaussian, one can verify that Assumption~\ref{assum:generalized_linear_model} holds with $\gamma = 0$, $L_1 = L_2 = \frac{1}{4}$, and $\sigma = \sigma_y = \sigma_\star =2.$

We now verify that Assumptions~\ref{assum: non-quadratic} and~\ref{assum:vr_assum} hold in the regularized GLM, recalling our choice $\| \cdot \| = \| \cdot \|_2$.  Note that for any $(\xi,y)$, we have
\begin{align*}
\nabla_x \ell(x,\xi, y)
&= (u'(\dotp{x,\xi})-y)\xi, \\
\nabla_{xx}^2 \ell( x,\xi ,y)
&= u''(\dotp{x,\xi})\xi\xi^\top,
\end{align*}
and hence
\begin{align*}
\nabla F(x) &= \EE\!\left[(u'(\dotp{x,\xi})-y)\xi\right] + \lambda x,\\
\nabla^2 F(x) &= \EE\!\left[u''(\dotp{x,\xi})\xi\xi^\top\right] + \lambda I.
\end{align*}

\noindent\emph{Verifying Assumption~\ref{assum: non-quadratic}.}
Since $u$ is $\gamma$-strongly convex, for any $x$, we have $\nabla^2 F(x)\succeq (\gamma  \sigma_{\min}+\lambda) I$,
so $F$ is $(\gamma  \sigma_{\min}+ \lambda)$-strongly convex. 
Next for any  $x,x'\in\RR^d$,
\begin{align*}
\|\nabla F(x)-\nabla F(x')\|_2 &= \|\EE[(u'(\dotp{x,\xi} )- u'(\dotp{x',\xi}))\xi] + \lambda (x - x')\|_2\\
&\le  (L_1\EE[\|\xi\|_2^2] + \lambda) \|x' - x\|_2.
\end{align*}
Finally,
there exists a universal constant $C>0$ such that
\begin{align*}
\|(\nabla^2 F(x^\star))^{-1}(\nabla^2 F(x)-\nabla^2 F(x'))\|_2 &\leq \frac{1}{\gamma  \sigma_{\min}+ \lambda} \cdot \sup_{v \in \mathbb{S}^{d-1}} |\EE[ |u''(\dotp{x,\xi})-u''(\dotp{x',\xi}) | \dotp{v,\xi}^2]| \\
&\le  \frac{1}{\gamma  \sigma_{\min}+ \lambda} \cdot \sup_{v \in \mathbb{S}^{d-1}} \EE[ L_2 \abs{\dotp{x-x',\xi}} \dotp{v,\xi}^2]\\
&\le  \frac{CL_2 \sigma^3 \|x - x'\|_2}{\gamma  \sigma_{\min}+ \lambda},
\end{align*}
where the  first inequality follows from the variational form of the operator norm and the fact that $\nabla^2 F(x^\star)\succeq (\gamma  \sigma_{\min}+ \lambda) I$, and the last inequality follows from properties of the sub-Gaussian norm. Thus, we have $\liph\leq \frac{ CL_2 \sigma^3}{\gamma  \sigma_{\min} + \lambda}$.

\noindent \emph{Verifying Assumption~\ref{assum:vr_assum}.} Define $f(x,\xi, y) = \ell(x,\xi,y) + \frac{\lambda}{2} \|x\|_2^2$. There exists some constant $C>0$ such that
\begin{align*}
&\EE[ \| (\nabla_x f(x,\xi, y) - \nabla F(x))  - (\nabla_x f(x', \xi, y) - \nabla F(x')) \|_2^2] \\
&\qquad = \EE[ \|u'(\dotp{x,\xi}) \xi  - u'(\dotp{x',\xi}) \xi  - \EE[u'(\dotp{x,\xi}) \xi  - u'(\dotp{x',\xi}) \xi  ] \|_2^2]\\
&\qquad \le   \EE[ \|u'(\dotp{x,\xi}) \xi  - u'(\dotp{x',\xi}) \xi \|_2^2] \\	
&\qquad \le   L_1^2 \EE[\|\dotp{x-x',\xi} \xi\|_2^2]\\
&\qquad \le CL_1^2 d\sigma^4 \|x-x'\|_2^2,
\end{align*}
where the last inequality follows from Hölder's inequality together with a standard fourth-moment bound for sub-Gaussian random variables; see, e.g.,~\cite[Proposition 2.6.6]{vershynin2018high}.  
\hfill $\clubsuit$

Having set the stage, we now examine the instance-dependent performance of existing methods more closely on a specific example.

\section{Warmup: Quadratic optimization} \label{sec:suboptimality}
In this section, we use quadratic optimization as a testbed to show how existing algorithms fail to attain instance-optimal performance. We begin by establishing a non-asymptotic, instance-dependent lower bound for this class of problems. We then show that neither SAA nor (averaged) SA achieves this lower bound. 

\subsection{Information-theoretic lower bound}

In our lower bounds, we focus on the $\ell_2$ norm, setting $\| \cdot \| = \| \cdot \|_2$.
To derive a meaningful non-asymptotic local minimax lower bound, we must carefully specify the class of problem instances over which worst-case risk is evaluated. 
To this end, we first fix a strongly convex quadratic population objective $F(x) = \frac{1}{2} x^\top A x + b^\top x$. For any sample size $n$ and symmetric positive semi-definite matrix $\Sigma$, we define the collection of instances 
\begin{align*}
\nclass(n, F, \Sigma) :=  \left\{ (f,P) \left| 
\begin{aligned}
	&\text{ $\nabla^2 F_{f,P} \equiv \nabla^2 F$ and $\|x^\star(F_{f,P}) - x^\star(F)\|_2 \le 2\cdot\sqrt{\frac{\trace{A^{-1} \Sigma A^{-1}}}{n}}$,} \\
	& \text{ $\nabla f(x^\star(F_{f,P}),z)$ has distribution  $N(0, \Sigma)$ when $z \sim P$.}
\end{aligned}
\right.\right\}.
\end{align*}
This construction captures three essential constraints. First, all problem instances $(f,P)$ in $\nclass(n, F, \Sigma)$ are close to the population reference problem, in that they have population objectives that are quadratic with identical Hessian structure as the reference problem $F$. Second, the minimizers of these population objectives lie within a shrinking neighborhood of $x^\star(F)$, with the neighborhood radius decreasing at a rate $O(n^{-\frac{1}{2}})$. The specific radius $2\cdot\sqrt{\frac{\trace{A^{-1} \Sigma A^{-1}}}{n}}$ is carefully chosen to match the scale of our lower bound, as will become apparent in the theorem statement. Third, the gradient noise structure is precisely controlled: at each instance's minimizer, the gradient noise follows a mean-zero Gaussian distribution with covariance $\Sigma$.

Our instance class $\nclass(n, F, \Sigma)$ occupies a level of granularity between existing notions of global problem classes (used to assess global minimax risk) and local neighborhoods considered in asymptotic lower bounds~\cite{duchi2021asymptotic}. In particular, we focus on problem instances sharing a fixed population Hessian structure and noise geometry, without insisting that the (stochastic) instance be close in KL divergence to the reference $P$ (cf. Eq.~\eqref{eqn:asymp_lb}). As a result, the neighborhood is fine enough to capture the local geometric term $\Lambda$ from the asymptotic lower bound~\eqref{eqn:asymp_lb}, but coarse enough to enable tractable analysis with finite sample size $n$. 

To state our result, we require some setup. Let $\widehat \cX_n$ be the set of estimators based on $n$ samples and the sample objective function, i.e., each $\widehat x_n \in \widehat \cX_n$ is a map taking $(\{z_i\}_{i=1}^n, f)$ as inputs, and $\widehat x_n(\cdot, f)$ is a measurable map from $\cZ^{n}$ to $\RR^d$ for any fixed $f$. We remove parentheses and simply write $\widehat x_n$ when it is clear from context.
\begin{proposition}\label{prop: mainlowerboundtheorem_quadratic_intro}
Let \(F\) be a quadratic function with Hessian matrix \(A\) that satisfies Assumption~\ref{assum: non-quadratic} with parameters \(L \ge \mu > 0\) and \(\liph = 0\). For any  positive semi-definite covariance matrix $\Sigma$ and integer $n \ge 1$, we have 
\begin{align}\label{eqn:non-asymp_quadratic}
	&\quad \inf_{\widehat x_n \in \widehat \cX_n} \sup_{(f,P)\in \nclass(n,F,\Sigma)} \EE[\|\widehat x_n - x^\star(F_{f, P})\|_2^2] \ge
	\frac{ \trace{A^{-1} \Sigma A^{-1}}}{4(\pi^2 +1) n}.
\end{align}
In addition, there is a stochastic first-order method such that  for any $(f,P) \in \nclass(n,F,\Sigma)$, when Assumption~\ref{assum:vr_assum} holds with parameter $\zeta$,  the output $\widehat x_n^{\texttt{(FOM)}}$ using $n$ stochastic gradients satisfies 
\begin{align*}
	\EE[ \|\widehat x_n^{\texttt{(FOM)}} - x^\star(F_{f,P})\|_2^2 ]  \le 	\frac{C \cdot  \trace{A^{-1} \Sigma A^{-1}}}{ n} \quad \text{for any $n  = \tilde \Omega \left(\sqrt{\frac{L}{\mu}} +\frac{\zeta^2}{\mu^2} \right)$},
\end{align*}
where $C$ is a positive universal constant.
\end{proposition} 
This proposition is an immediate consequence of Corollary~\ref{cor:complexity_quad_ASGD} and Theorem~\ref{thm: mainlowerboundtheorem_intro}, which we state shortly. In contrast to the asymptotic optimality result of~\cite{duchi2021asymptotic}, Proposition~\ref{prop: mainlowerboundtheorem_quadratic_intro} is non-asymptotic. In particular, it simultaneously establishes the tightness of the lower bound and the optimality of the stochastic first-order method whenever \(n\) is bounded below by an explicit problem-dependent quantity of order
$\sqrt{\frac{L}{\mu}} + \frac{\zeta^2}{\mu^2}.$
This lower bound on \(n\) appears to be necessary. On the one hand, the term \(\sqrt{\frac{L}{\mu}}\) is required for any first-order method to minimize a smooth, strongly convex function, by classical oracle complexity results~\cite{complexity}. On the other hand, the necessity of the term $\frac{\zeta^2}{\mu^2}$ has also been observed in prior work; see, e.g.,~\cite{khamaru2021temporal, li2022root,li2023accelerated,mou2023optimal}. Indeed, consider the sample objective \(f(x,a)=\frac{a}{2}x^2\) and the testing problem between \(a\sim N(0,\zeta^2)\) and \(a\sim N(\mu,\zeta^2)\). By a standard application of Le Cam's method, these two hypotheses cannot be distinguished with fewer than $\Omega(\frac{\zeta^2}{\mu^2})$ samples. Thus, even in the one-dimensional setting, multiplicative noise creates an information-theoretic barrier: with fewer than \(\Omega(\frac{\zeta^2}{\mu^2})\) samples, one cannot reliably distinguish the population objective \(x\mapsto \frac{\mu}{2}x^2\) from the constant zero function. This implies that no reasonable algorithm can be expected to make meaningful progress before the sample size reaches this scale.

Moreover, note that this optimality result is inherently local. It shows that, when $n$ is sufficiently large, the stochastic first-order method is minimax optimal over a neighborhood consisting of instances whose minimizers are at most $O\left(\sqrt{\frac{\trace{A^{-1}\Sigma A^{-1}}} {n}}\right)$ away from the solution to the given problem. Importantly, this neighborhood radius matches our lower bound up to a universal constant, showing that the result is local in the sharpest possible sense.

Having established these local fundamental limits, we now examine the performance of Sample Average Approximation (SAA) and robust (averaged) Stochastic Approximation (SA) methods. We demonstrate that both approaches can fail to achieve the lower bound~\eqref{eqn:non-asymp_quadratic}, even when provided with $\Omega\left(\frac{\zeta^2}{\mu^2}\right)$ samples.
\subsection{Arbitrarily large  error of SAA}\label{sec:SAA_fail}

To see why SAA can fail even on a basic quadratic optimization problem,
consider the  one-dimensional  problem where $A_i \overset{i.i.d.}{\sim} N(1,1)$ and $b_i \overset{i.i.d.}{\sim} N(-1, 1)$ are drawn independently of each other. Suppose $\| \cdot \| = \| \cdot \|_2$, which reduces to the absolute value in the one-dimensional case. It is straightforward to verify that Assumption~\ref{assum:vr_assum} holds with $\zeta = 1$. A straightforward calculation shows that
\begin{align*}
F(x) = \frac{x^2}{2} - x, \qquad \text{and} \qquad x^\star = 1.
\end{align*}
Given $n$ i.i.d. samples $\{(A_i, b_i)\}_{i=1}^{n}$, the unique critical point of the sample objective is 
$$
\cp = -  \bar b_n/\bar A_n,
$$
where $\bar b_n =  \frac{1}{n} \sum_{i=1}^{n} b_i \sim N(-1, \frac{1}{n}) $ and $\bar A_n = \frac{1}{n} \sum_{i=1}^{n} A_i \sim N(1, \frac{1}{n})$. The SAA estimator (defined as the set of minimizers of the sample loss) is given by 
$$\xSAA_n =\begin{cases}
\cp &  \text{if $\bar A_n > 0$} \\
\pm \infty &  \text{if $\bar A_n <  0$} \\
\mathbb{R} &\text{otherwise}.
\end{cases}.$$ 
No matter which point in this set is chosen, the $\ell_2^2$ error of the SAA estimator is pointwise larger than the $\ell_2^2$ error of $\cp$. Moreover, we have
\begin{align*}
\EE[|\cp - x^\star(F)|^2]
&= \EE\left[\left(\frac{\bar b_n + \bar A_n}{\bar A_n}\right)^2\right] \\
&= \int_{-\infty}^{\infty}\int_{-\infty}^{\infty}
\left(\frac{x+y}{1+x}\right)^2
\frac{n}{2\pi}
e^{-\frac{n}{2}(x^2+y^2)}
\, dx\, dy \\
&= \int_{-\infty}^{\infty}
\frac{1}{(1+x)^2}
\left(
\int_{-\infty}^{\infty} (x+y)^2 \sqrt{\frac{n}{2\pi}} e^{-\frac{n}{2}y^2}\,dy
\right)
\sqrt{\frac{n}{2\pi}} e^{-\frac{n}{2}x^2}\,dx \\
&= \int_{-\infty}^{\infty}
\frac{x^2+\frac{1}{n}}{(1+x)^2}
\sqrt{\frac{n}{2\pi}} e^{-\frac{n}{2}x^2}\,dx \\
&\ge
\frac{1}{n}
\int_{-\infty}^{\infty}
\frac{1}{(1+x)^2}
\sqrt{\frac{n}{2\pi}} e^{-\frac{n}{2}x^2}\,dx \\
&= \infty,
\end{align*}
where in the second equality we used the change of variables
\[
x=\bar A_n-1,
\qquad
y=\bar b_n+1.
\]
This implies the whole integral diverges.
Therefore, SAA can incur arbitrarily large $\ell^2_2$ error for any sample size $n$. 
Similar issues persist if the norm is given by other natural choices.

\subsection{Sub-optimality of robust SA}\label{sec:suboptimality_of_SA}

We now turn to vanilla stochastic approximation (SA) (also called the stochastic gradient method) as well as a popular variant that averages the iterates of this algorithm. Recall that vanilla SA is initialized at some point $x_1 \in \mathbb{R}^d$, and for a sequence of nonnegative stepsizes $\{\eta_k\}_{k \geq 1}$, it executes the iteration
\begin{equation}\label{eqn:SGD}
x_{k+1} =x_k - \eta_k \nabla f(x_k, z_k),
\end{equation} 
where $z_k$ is a new sample drawn from the distribution $P$.

To stabilize the algorithm, we average the iterates---an idea dating back to Ruppert, Polyak, and Juditsky---to obtain \emph{robust SA}, given by either
\begin{align}
\rpj = \frac{1}{n}\sum_{k=1}^{n} x_k \text{ or its tail-averaged variant } \rpj = \frac{1}{\floor{n/2}}\sum_{k=\ceil{n/2}}^{n} x_k.
\end{align}

The first question is whether stochastic gradient descent (SGD) in~\eqref{eqn:SGD} provides an instance-optimal estimator. This question has received significant interest, and even asymptotically, the answer is negative. Bach and Moulines~\cite[Theorem 2]{moulines2011non} showed that for $\beta \in (0,1)$, the last iterate of SA exhibits the suboptimal convergence rate
\begin{align*}
\EE[\|\sqrt{n}(x_n - x^\star)\|_2^2] \rightarrow \infty,
\end{align*}
and when $\beta = 1$, the classical result of Fabian~\cite[Theorem 3.4]{fabian1968asymptotic} shows that even when the problem is quadratic, the asymptotic distribution of $\sqrt{n}(x_n - x^\star)$ is mean-zero Gaussian with a sub-optimal covariance matrix (see also~\cite{benveniste2012adaptive}). %

Let us now turn to robust SA, which is known to satisfy asymptotic optimality~\cite{polyak1992acceleration,duchi2021asymptotic}. To investigate its non-asymptotic properties, we undertake a systematic study of the problem class introduced in~\eqref{eqn:example_break_robust_SA}. 
This is a family of quadratic problems parametrized by $\zeta$, such that for any $\zeta \ge 1$, we have $ x^\star = \begin{bmatrix}
1\\
1
\end{bmatrix}$ and the optimal covariance is always  $\Lambda = \begin{bmatrix}
1 & 0\\
0&  1
\end{bmatrix}$. Moreover, when selecting $\|\cdot\|$ as the $\ell_2$ norm, it is straightforward to verify that Assumption~\ref{assum: non-quadratic} holds with $\mu=1, L = \zeta^2$, and $\liph = 0$, and Assumption~\ref{assum:vr_assum} holds with parameter $2\zeta$.
According to Proposition~\ref{prop: mainlowerboundtheorem_quadratic_intro} and the discussion following it, the ``best" estimator $\widehat x_n$  should have $\EE[\|\sqrt{n}(\widehat x_n -x^\star)\|_2^2]$ bounded by a universal constant once $n$ is at the order of $ \Omega(\zeta^2)$.
However,  we will show in numerical experiments that $O(\zeta^2)$ samples are not sufficient for robust SA to achieve an expected squared error independent of $\zeta$. We examine two typical stepsize choices: constant and diminishing stepsizes of the form $\eta_k = \eta k^{-\beta}$.
\paragraph{Constant stepsize.} We compare the performance of averaging with constant stepsize with variance reduction. We report the numerical results for $\eta \in \{\frac{1}{\zeta^4}, \frac{1}{\zeta^3}, \frac{1}{\zeta^2}\}$ and $\beta = 0$ in Figure~\ref{fig:constant_stepsize}. %
\begin{figure}[ht!]
\centering
\begin{minipage}{0.88\textwidth}\
\begin{adjustbox}{valign=t}
	\begin{minipage}{0.65\linewidth}
		\centering
		\includegraphics[width=\linewidth]{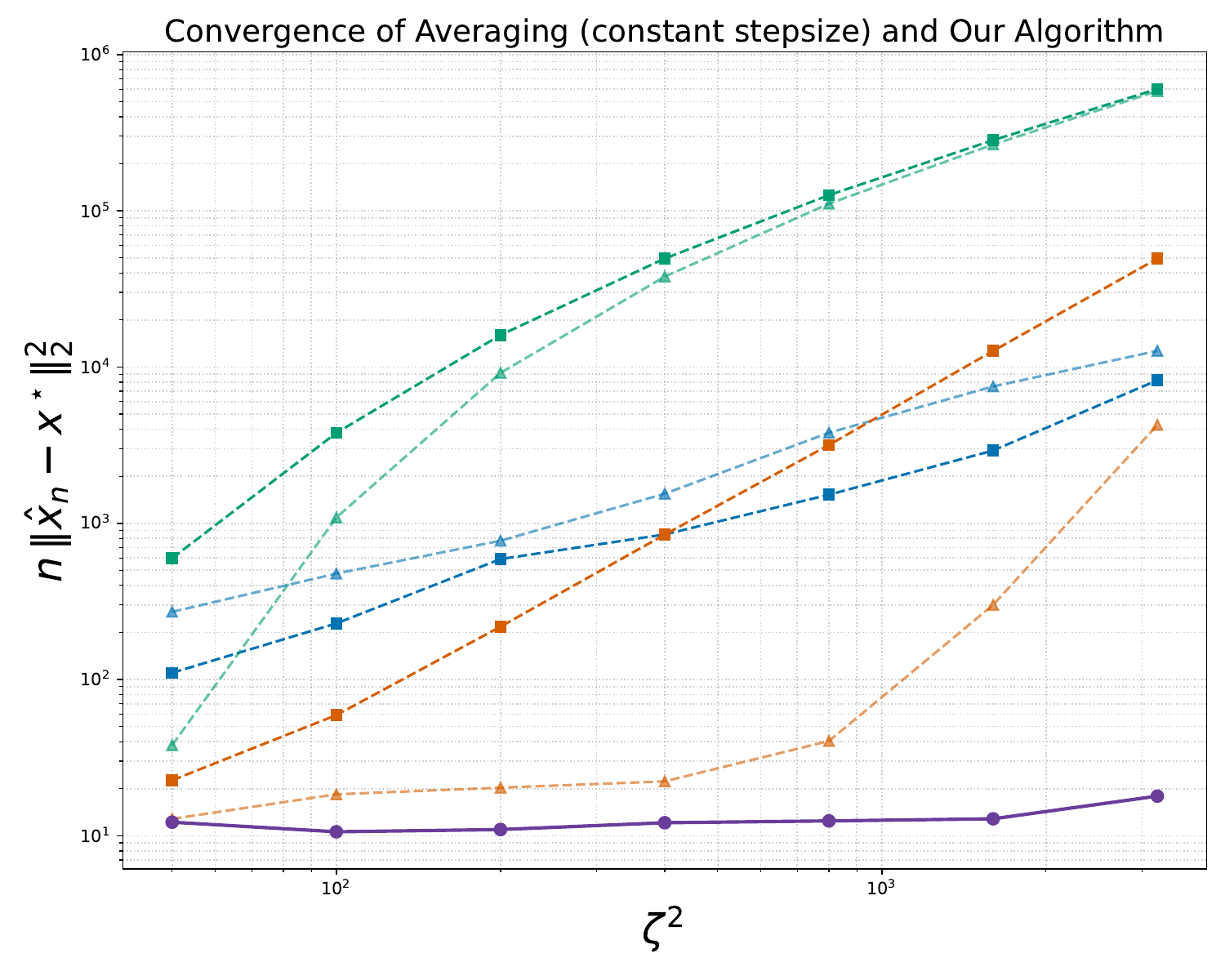}
	\end{minipage}
\end{adjustbox}
\hfill
\begin{adjustbox}{valign=t}
	\begin{minipage}{0.3\linewidth}
		\centering
		\includegraphics[width=1.2\linewidth]{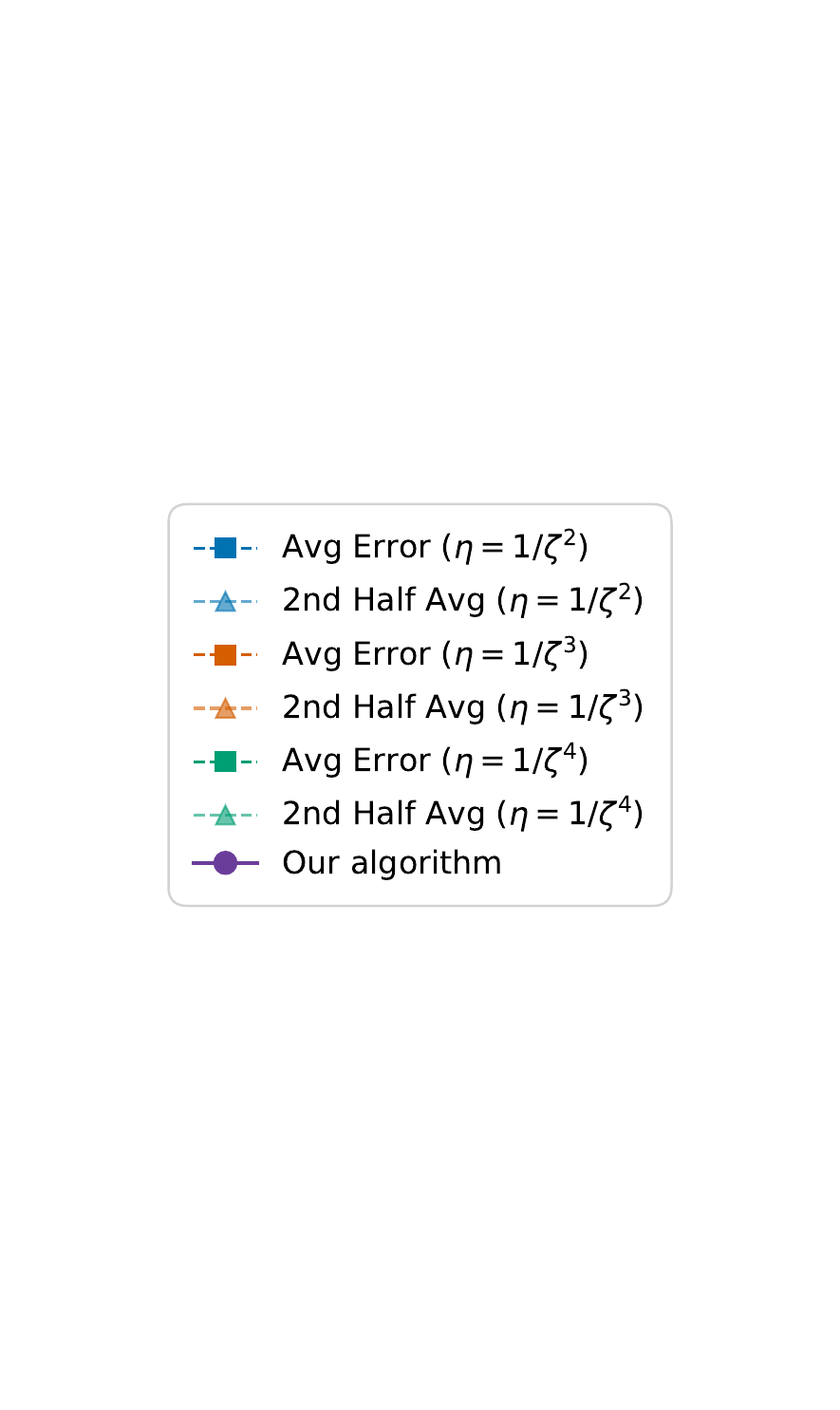}
	\end{minipage}
\end{adjustbox}

\caption{ \small Comparison of averaging (constant stepsize) and our algorithm.  All the algorithms are initialized at the origin. The total number of samples is $n = 200\cdot \zeta^2$. The error (y-axis) $n\|\widehat x_n -x^\star\|_2^2$ is averaged over 100 runs, where $\widehat x_n$ denotes the output of each algorithm with a certain parameter setting.}
\label{fig:constant_stepsize}
\end{minipage}
\end{figure}

\paragraph{Diminishing stepsize.} We compare the performance of averaging with diminishing stepsize with variance reduction. We run experiments for $\eta \in \{\frac{1}{\zeta}, \frac{1}{\zeta^2},  \frac{1}{\zeta^3}  \}$ and $\beta \in \{0.2, 0.5,0.8\}$, and report the algorithm performance in Figure~\ref{fig:diminishing_stepsize}. When  $\eta = \frac{1}{\zeta}$, the iterates completely blow up and cannot be plotted, so we do not include them. 
\begin{figure}[ht!]
\centering
\begin{minipage}{0.88\textwidth}
\centering 
\begin{adjustbox}{valign=t}
	\begin{minipage}{0.65\linewidth}
		\centering
		\includegraphics[width=\linewidth]{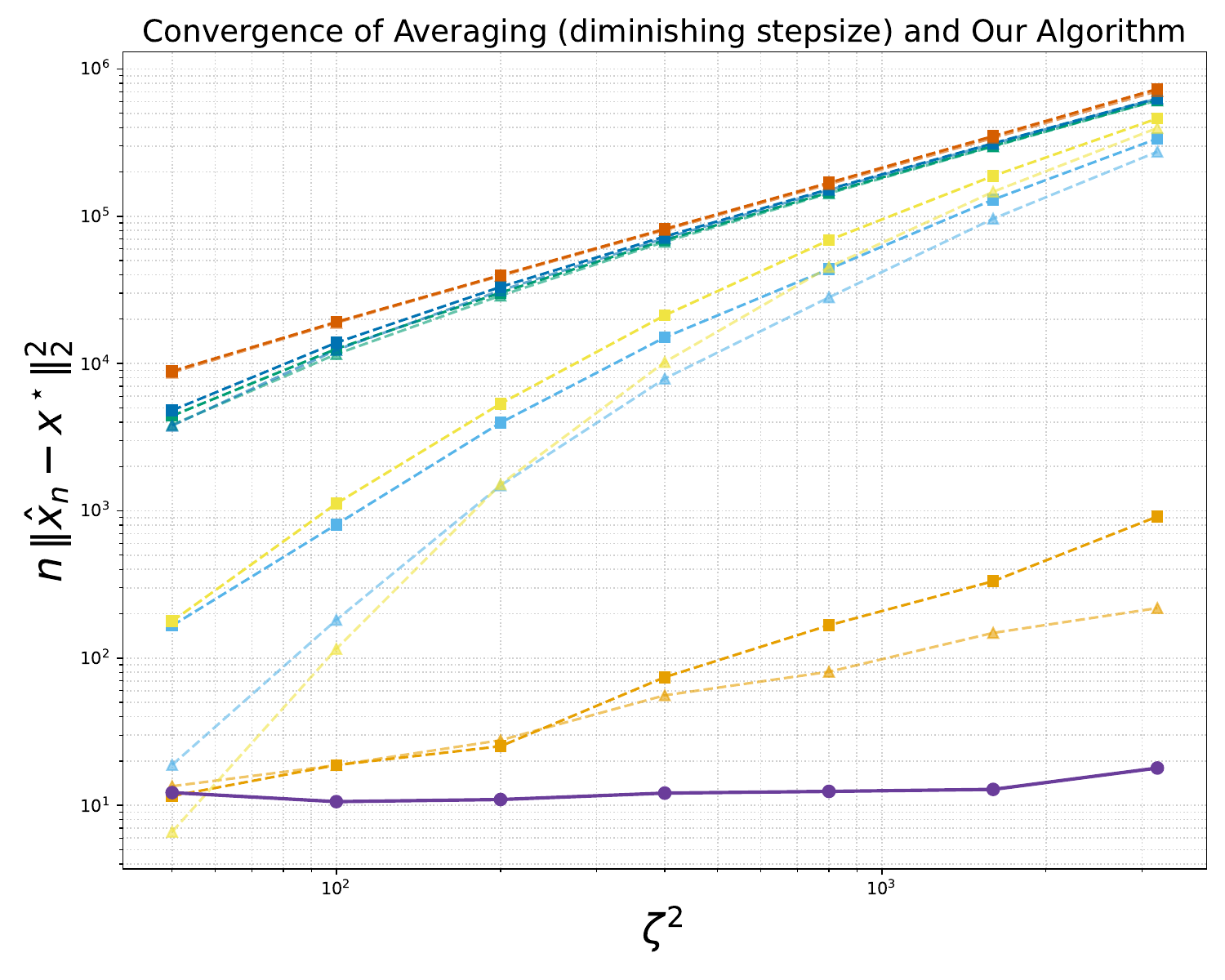}
	\end{minipage}
\end{adjustbox}
\hfill
\begin{adjustbox}{valign=t}
	\begin{minipage}{0.3\linewidth}
		\centering
		\includegraphics[width=1.2\linewidth]{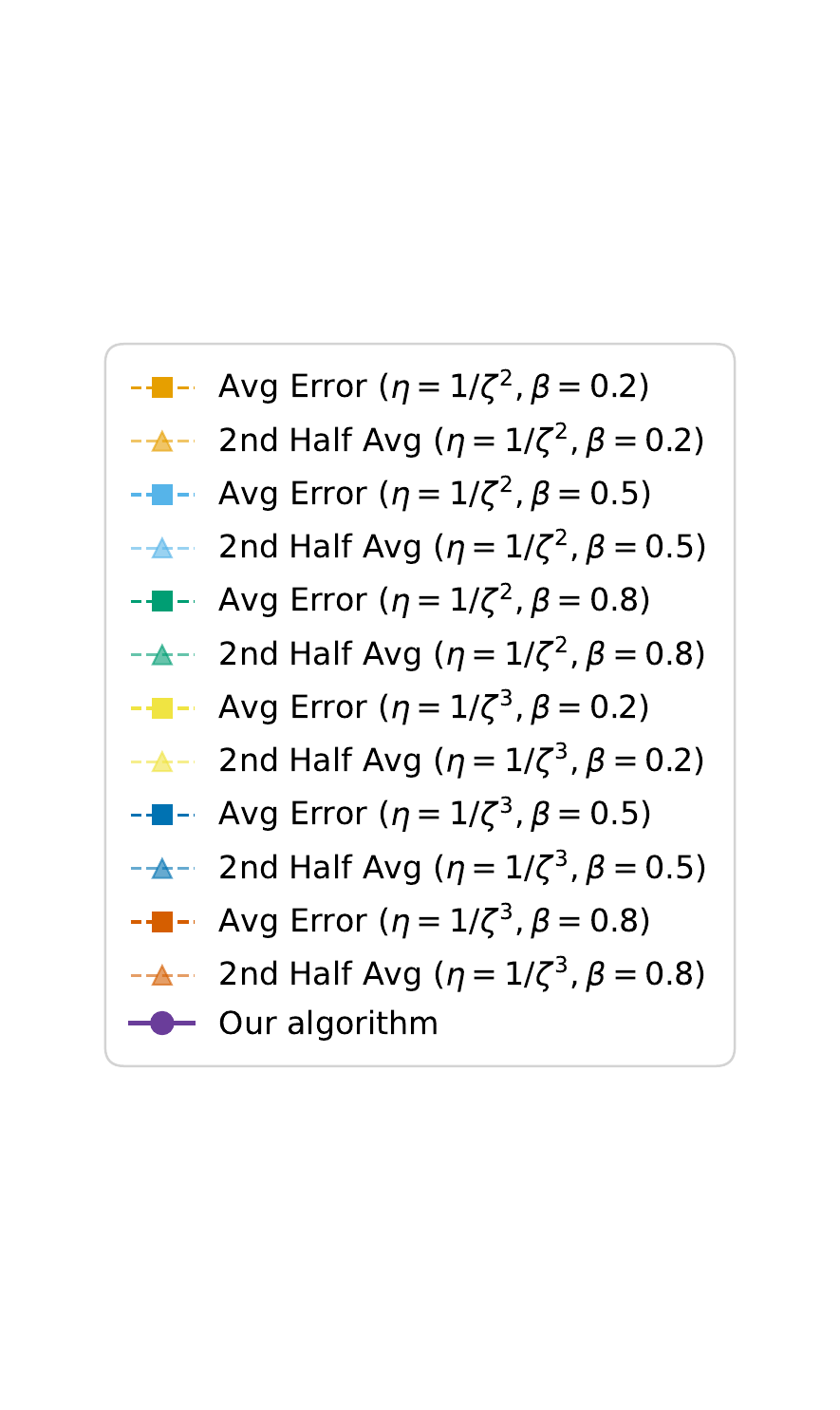}
	\end{minipage}
\end{adjustbox}

\caption{ \small Comparison of averaging (diminishing stepsize) and our algorithm.  All the algorithms are initialized at the origin. The total number of samples is $n = 200\cdot \zeta^2$. The error (y-axis) $n\|\widehat x_n -x^\star\|_2^2$ is averaged over 100 runs, where $\widehat x_n$ denotes the output of each algorithm with a certain parameter setting.}
\label{fig:diminishing_stepsize}
\end{minipage}
\end{figure}

\medskip

For averaged SA, our observations in Figures~\ref{fig:constant_stepsize} and~\ref{fig:diminishing_stepsize} reveal an interesting pattern: while the asymptotic convergence rate remains constant, the size of $n\|\widehat x_n - x^\star\|_2^2$ increases as we increase $\zeta$. This suggests that  $200 \cdot \zeta^2$ samples are insufficient for averaged SA to achieve the non-asymptotic lower bound~\ref{prop: mainlowerboundtheorem_quadratic_intro}. In contrast, our proposed algorithm (introduced in the following section) maintains a constant value of $n\|\widehat x_n - x^\star\|_2^2$ that matches the asymptotic lower bound, regardless of the value of $\zeta$.

\section{Our general framework and inner-loop algorithms}\label{sec:better_algorithm}

Having seen that both SAA and robust SA can be non-asymptotically suboptimal even on simple instances, we now introduce our algorithmic framework for approaching instance-optimality. The general framework, presented as Algorithm~\ref{alg:metaVR}, is actually a family of algorithms that involves a Variance Reduction wrap-around for Instance-optimal Stochastic Optimization ($\algname$). Concretely, the variance reduction device wraps around a base algorithm that is executed epoch-wise in an inner loop. For the inner loop, suppose we have access to an algorithm $\mathcal{A}$ (such as the stochastic gradient method) that takes as input a general sample-wise objective $g$, initial point $\widetilde{x}$, and runs for $T$ iterations---using $T$ fresh samples from the distribution $P$---to produce the iterate $\mathcal{A}(\widetilde{x}, g, T)$.

\begin{algorithm}[H]
\caption{$\algname (\widehat x_0, \{N_k\}_{k=1}^{K}, T, \cA)$ }
\label{alg:metaVR}
\begin{algorithmic}[1]
\State {\bfseries Input}  Initialization  $\widehat x_0 \in \RR^d$,  $\{N_k\}_{k=1}^{K}$, $T \ge 0$ , a stochastic optimization algorithm $\cA$ 
\For {$k = 1,\ldots, K$}
\State Set $ \tilde x = \widehat x_{k-1}$. Collect $N_k$ new samples $\{z_i^k\}_{i=1}^{N_k}$ .
\State Calculate 
\begin{align}\label{eqn:averging_meta}
	\widehat{\nabla} f(\tilde x) = \frac{1}{N_k} \sum_{i=1}^{N_k} \nabla f(\tilde x, z_i^k).
\end{align} 
\State Set $g(x,z) = f(x,z) - \langle \nabla f(\tilde x, z) - \widehat \nabla f(\tilde x), x \rangle$.
\State $\widehat x_k = \cA(\tilde x, g, T)$ \label{eqn:subroutine}
\EndFor
\State \Return $\widehat x_K$.
\end{algorithmic}
\end{algorithm}	
In more detail, the $\algname$ algorithm proceeds in epochs and consists of two key steps:
\begin{enumerate}
\item At the beginning of the $k$-th epoch, the algorithm collects $N_k$ fresh samples to compute the averaged gradient $\widehat \nabla f$ at the current base point $\tilde x$. 
\item\label{step:II} Within the $k$-th epoch, we run the inner-loop algorithm $\cA$ designed to optimize the following (population) objective using samples:
\begin{equation*}
\min_{x \in \RR^d} G(x) =  \EE_z [g(x,z)] =  F(x) - \dotp{\nabla F(\tilde x) - \widehat \nabla f(\tilde x), x}.
\end{equation*} 
\end{enumerate} 
Under Assumption~\ref{assum: non-quadratic}, function $G$ is smooth and strongly convex, so Step~\ref{step:II}  is equivalent to finding $\underline x$ such that $\nabla G(\underline x) = 0$, i.e.,
\begin{align} \label{eqn: underline_x_def}
\nabla F(\underline x) - (\nabla F(\tilde x) - \widehat \nabla f(\tilde x)) = 0.
\end{align}

Our main algorithm depends critically on two design choices: (i) selecting the sample sizes $N_k$ so that the epoch-wise population solution $\underline{x}$ approaches $x^\star$ at the desired rate, and (ii) choosing a proper stochastic optimization subroutine $\cA$ that transfers this progress from $\underline{x}$ to the epoch output. The next two subsections develop these choices in parallel. Section~\ref{subsec:epoch_solution} addresses (i) by quantifying how $N_k$ controls the error $\EE[\|\underline{x}-x^\star\|^2]$. Section~\ref{sec:stochastic_opt_subroutine} addresses (ii) by presenting two stochastic optimization subroutines: vanilla SGD as well as an accelerated variant.

\subsection{Analysis of the epoch-wise population solution}\label{subsec:epoch_solution}
Recall that for any epoch $k \in [K]$, $\tilde x$ is the initialization within this epoch. Throughout this subsection, we let vector $\underline{x}$ be the solution to~\eqref{eqn: underline_x_def}.
Under Assumption~\ref{assum: non-quadratic}, such a solution exists uniquely since $F$ is smooth and strongly convex.  We first state and prove the result when the population objective is a quadratic function. 	\begin{lemma}\label{lem:vr_quadratic_key}
Suppose that  Assumption~\ref{assum: non-quadratic} holds with parameters  $\liph = 0$ and  $L \ge \mu = \lmin >0$,   and that Assumption~\ref{assum:vr_assum} holds. We write $A = \nabla^2 F(x^\star)$. For any fixed epoch $k$, let $\underline{x}$ be defined as in equation~\eqref{eqn: underline_x_def}. We have
\begin{align}\label{eqn:vr_quadratic_key}
\EE[\|\underline{x} - x^\star\|^2 \mid \tilde x] \le \frac{2}{N_k}\EE[\|A^{-1}\nabla f(x^\star,z)\|^2]+ \frac{2 \zeta^2}{N_k \mu^2 } \|\tilde x - x^\star\|^2.
\end{align}
\end{lemma}	
\begin{proof}
Recall that $\nabla F(x) = Ax +b = A(x-x^\star)$, so we have
\begin{align*}
\underline{x} -x^\star &= - A^{-1} (\widehat{\nabla} f(\tilde x) - \nabla F(\tilde x))\\
&= - \frac{1}{N_k}\sum_{i=1}^{N_k}A^{-1}(\nabla f(\tilde x, z_i ) - \nabla F(\tilde x) )\\
&= -\frac{1}{N_k}\sum_{i=1}^{N_k}A^{-1}(\nabla f(x^\star,z_i)) -    \frac{1}{N_k}\sum_{i=1}^{N_k} A^{-1}(\nabla f(\tilde x, z_i ) - \nabla F(\tilde x) - \nabla f(x^\star,z_i))
\end{align*}
As a result,
\begin{align*}
&\quad \EE[\|\underline{x} - x^\star\|^2\mid \tilde x]\\
&\le  \frac{2}{N_k} \EE[\|A^{-1}\nabla f(x^\star,z) \|^2] + \frac{2}{N_k}\EE\left[\left\| A^{-1}(\nabla f(\tilde x, z_1 ) - \nabla F(\tilde x) - \nabla f(x^\star,z_1))\right\|^2 \big| \;\tilde x\right]\\
&\le \frac{2}{N_k} \EE[\|A^{-1}\nabla f(x^\star,z) \|^2] + \frac{2 \zeta^2}{ N_k \mu^2} \|\tilde x  -x^\star\|^2,
\end{align*}
where the first inequality follows from Young's inequality and  the independence of samples, and the second inequality follows from Assumptions~\ref{assum:vr_assum} and Lemma~\ref{lem:general_lambda_min}.
\end{proof}
When $\|\cdot\|$ is the $\ell_2$ norm, $\frac{2}{N_k} \EE[\|A^{-1}\nabla f(x^\star,z) \|_2^2]  = \frac{2}{N_k} \trace{A^{-1}\cov(\nabla f(x^\star,z)) A^{-1}}$; the numerator captures the correct geometry (see the local minimax lower bound in Proposition~\ref{prop: mainlowerboundtheorem_quadratic_intro}). Therefore, if we select $N_k \ge \frac{4\zeta^2}{\mu^2}$, then ~\eqref{eqn:vr_quadratic_key} yields  
$$
\EE[\|\underline{x} - x^\star\|_2^2 \mid \tilde x] \le \frac{2}{N_k} \trace{A^{-1}\cov(\nabla f(x^\star,z)) A^{-1}}+ \frac{1}{2}\cdot \|\tilde x - x^\star\|_2^2.
$$  
Thus, the initial expected error is halved, up to an additive statistical error 
$$\frac{2}{N_k}\trace{A^{-1}\cov(\nabla f(x^\star,z)) A^{-1}}.$$

We now turn to the general setting. In this non-quadratic setting, the behavior of the Hessian plays an important role. Specifically, the Hessian $\nabla^2 F(x)$ remains close to $\nabla^2 F(x^\star)$ only within a neighborhood around $x^\star$, whose size is dictated by the Hessian Lipschitz constant $\liph$. Outside this region, the geometry of the problem can vary considerably. Since the optimization algorithm's behavior at each iterate typically depends on the Hessian at its current position, one cannot expect to achieve identical behavior to the quadratic case (which is determined solely by the Hessian at $x^\star$) if heavy-tailed noise constantly pushes iterates outside this neighborhood. Therefore, we assume that the gradient noise has sub-exponential tails according to the following definition:
\begin{definition}[Sub-exponential random vectors]
A random vector $X \in \RR^d$ sub-exponential with parameter $(\nu^2, \alpha)$ if for any unit vector $\| v\| = 1$, $\dotp{v, X - \EE[X]}$ is a sub-exponential random variable with parameters $(\nu^2,\alpha)$.
\end{definition}
\begin{assumption}\label{assum:light_tail_non_quadratic}
For $z \sim P$, the  noisy gradient  $\nabla f(x, z)$ is a sub-exponential vector with parameter $(\sigma_1^2 + \sigma_2^2 \|x - x^\star\|^2, \sigma_1 + \sigma_2 \| x- x^\star\|)$.
As a consequence, $\nabla f(x,z)$ has finite moments for any $x$. 
\end{assumption}
This particular assumption, while convenient for clarity, is not strictly necessary. Indeed, we can establish similar convergence guarantees under weaker conditions, such as merely requiring bounded fourth moments. Identifying the weakest possible noise assumptions remains an interesting direction for future work.
\begin{remark}
We use the same $\sigma_1$ and $\sigma_2$ in Assumption~\ref{assum:light_tail_non_quadratic} for both sub-exponential parameters primarily for notational simplicity, since this suffices for all our examples. Our analysis also extends to the more general setting where $\nabla f(x, z)$ is sub-exponential with parameters $(\sigma_1^2 + \sigma_2^2\|x - x^\star\|^2,\ \tilde{\sigma}_1 + \tilde{\sigma}_2 \|x - x^\star\|)$.
\end{remark}

Below, we present our epoch-wise bound for general non-quadratic functions, with the proof deferred to Appendix~\ref{appendix:vr_nonquadratic}. 
\begin{lemma}\label{lem:vr_key_non_qud}
Suppose that Assumptions~\ref{assum: non-quadratic}, ~\ref{assum:vr_assum}, and~\ref{assum:light_tail_non_quadratic} hold. We write $A = \nabla^2 F(x^\star)$. For any fixed epoch $k$, let $\underline{x}$ be defined as in equation~\eqref{eqn: underline_x_def} and $n$ be the number of samples used in step~\eqref{eqn:averging_meta}. Suppose that
$$n \ge  \max \left\{ \frac{1024 \liph^2  \sigma_1^2}{\lmin^2}, \frac{128\liph \sigma_1}{\lmin}\right\} \cdot \max\left\{ d, \log \left(\frac{4 \lmin^2}{\mu^2 \liph^2 \EE[\|A^{-1} \nabla f(x^\star,z)\|^2]}\right) \right\},$$
and 
$$
n \ge   2\ceil{\max \left\{ \frac{1024 \liph^2  \sigma_1^2}{\lmin^2}, \frac{128\liph \sigma_1}{\lmin}\right\}} \left( \log \left( \max \left\{ \frac{1024 \liph^2  \sigma_1^2}{\lmin^2}, \frac{128\liph \sigma_1}{\lmin}\right\}\right) +1\right).
$$
Then we have
\begin{align*}
	\EE[\|\underline{x} - x^\star\|^2 \mid \tilde x] \le \frac{9 \cdot \EE[\|A^{-1} \nabla f(x^\star,z)\|^2]}{n} + \frac{8(\zeta^2  + d\sigma_2^2)}{n \mu^2} \|\tilde x - x^\star\|^2.
\end{align*}
\end{lemma}
In a spirit similar to Lemma~\ref{lem:vr_quadratic_key}, the above result shows that if $n \ge \frac{16(\zeta^2 + d \sigma_2^2)}{\mu^2}$, then the initial error in the epoch is halved, up to an additive statistical error term $\frac{9 \cdot \EE[\|A^{-1} \nabla f(x^\star,z)\|^2]}{n}$.

\subsection{Stochastic optimization subroutines for the inner loop}\label{sec:stochastic_opt_subroutine}
In this section, we present two stochastic optimization algorithms that can serve as subroutines $\cA$ within Algorithm~\ref{alg:metaVR}. 
We begin with vanilla SGD.
\begin{algorithm}
\caption{SGD($\tilde x, g, T$)}
\label{alg:SGD}
\begin{algorithmic}[1]
	\State {\bfseries Input}   Initialization $\tilde x \in \RR^d$, sample objective $(x,z) \rightarrow g(x,z)$, $ T\ge 0$.
	\State Extra input: stepsize $\eta$,  weights $\{w_t\}_{t=0}^{T+1}$   %
	\State Set $x_0 = \tilde x$. 
	\For {$t = 0,\ldots, T$}
	\State Collect a new sample $z_t$ and compute
	\begin{equation}\label{eqn:VRupdate_innerloop}
		x_{t+1} = \argmin_{x \in \RR^d}\left\{ \eta  \dotp{\nabla g(x_t,z_t), x - x_t} +  \frac{1}{2}\|x_t - x\|^2\right\}.
	\end{equation}
	\EndFor
	\State \Return  $\widehat x = \frac{\sum_{t=0}^{T+1} w_t x_t}{\sum_{t=0}^{T+1} w_t}$
\end{algorithmic}
\end{algorithm}

For convenience, we present below a convergence guarantee for SGD---this is a slightly nonstandard (but still straightforward) result owing to state-dependent noise. Recall that $g(x,z)= f(x,z) - \langle \nabla f(\tilde x, z) - \widehat \nabla f(\tilde x), x \rangle$ and $G(x) = \EE_{z\sim P}[g(x,z)]$. Here $\tilde x$ and $\widehat \nabla f(\tilde x)$ are produced by other subroutines and are treated as random variables. Let $\cF_0 = \sigma\left(\tilde x, \widehat \nabla f(\tilde x) \right) $ and $\cF_t = \sigma\left(\tilde x, \widehat \nabla f(\tilde x), z_0, z_1,\ldots,z_{t-1}\right)$ denote the $\sigma$-algebra generated by all the random variables up to time $t$.
\begin{proposition}\label{prop:convergence_SGD}
Denote the minimizer of $G$ by $\underline{x}$ and suppose that Assumptions~\ref{assum: non-quadratic} and~\ref{assum:vr_assum} hold. Suppose that Algorithm~\ref{alg:SGD} is run with parameters $\eta \le \min \left\{\frac{1}{2L}, \frac{\mu }{256\zeta^2} \right\}$ and $T \ge \max\{\frac{128}{\eta \mu}, 64\}$. Setting  $w_0 = 0$ and $w_i = \frac{1}{T+1}$ for all $1\le i \le T+1$,  we have
\begin{align*}
	\EE[\|\widehat x - \underline{x}\|^2 \mid \cF_0] \le \frac{1}{16} \|\tilde x - \underline{x}\|^2.
\end{align*}
\end{proposition}
We defer the proof to Appendix~\ref{appendix:stochastic_opt_subroutine_SGD}. Note that SGD requires $T = \Omega(\frac{1}{\eta \mu}) = \Omega( \frac{L}{\mu} + \frac{\zeta^2}{\mu^2})$ new samples within each epoch, and the term $\frac{L}{\mu}$ is suboptimal when compared to the oracle complexity lower bound. To remedy this issue, we introduce an accelerated SGD (ASGD), which draws inspiration from and extends the AC-SA algorithm proposed by Ghadimi and Lan~\cite{ghadimi2012optimal}. 
\begin{algorithm}[H]
\caption{ASGD($\tilde x, g, T$)}
\label{alg:ASGD}
\begin{algorithmic}[1]
	\State {\bfseries Input}   Initialization $\tilde x \in \RR^d$, sample objective $(x,z) \rightarrow g(x,z)$, $ T\ge 0$.
	\State Extra input: stepsize $\{\alpha_t\}_{t\ge 1}$ and $\{\gamma_t\}_{t \ge 1}$ s.t. $\alpha_1=1, \alpha_t \in (0,1)$ for $t \ge 2$, and $\gamma_t > 0$ for any $t \ge 1$, sampling parameter $\{m_t\}_{t\ge 1}$, parameter $\tmu $.  %
	\State Set the initial point $y_0=x_0 = \tilde x$ and $t =1$.
	\For {$t = 1,\ldots, T$}
	\State Set
	\begin{align*}
		r_t = \frac{(1-\alpha_t)(\tmu + \gamma_t)}{\gamma_t + (1-\alpha_t^2)\tmu} y_{t-1} + \frac{\alpha_t [(1-\alpha_t) \tmu + \gamma_t]}{\gamma_t + (1-\alpha_t^2) \tmu} x_{t-1}; 
	\end{align*}
	\State Collect $m_t$ new i.i.d. samples $\{z_i^{(t)}\}_{i=1}^{m_t}$ and write $G_t = \frac{1}{m_t} \sum_{i=1}^{m_t} \nabla g(r_t, z_i^{(t)})$. Set 
	\begin{align*}
		x_{t} &= \argmin_{x \in \RR^d} \left \{\alpha_t [\dotp{G_t, x} + \frac{\mu}{2} \|r_t - x\|^2] + \frac{(1-\alpha_t) \tmu + \gamma_t}{2} \|x_{t-1} - x\|^2 \right\} 
	\end{align*} 
	\begin{align}
		y_{t} = \alpha_t x_t + (1-\alpha_t) y_{t-1}.
	\end{align}
	\EndFor
	\State \Return  $\widehat x = y_T.$
\end{algorithmic}
\end{algorithm}
Note that our algorithm is slightly different from AC-SA in~\cite{ghadimi2012optimal}: we allow the minibatch size $m_t$---used in building the stochastic gradient $G_t$---to be time-dependent. This modification is key to accommodating state-dependent noise and retaining the desired instance-dependent performance. Let $\tilde \cF_1 = \sigma\left(\tilde x, \widehat \nabla f(\tilde x) \right) $ and  \(\tilde \cF_t = \sigma\left(\tilde x, \widehat \nabla f(\tilde x) ,\bigcup_{s=1}^{t-1}\bigcup_{i=1}^{m_s}\{z_i^{(s)}\}\right)\) denote the \(\sigma\)-algebra generated by all random variables observed up to time \(t\), excluding \(\bigcup_{i=1}^{m_t}\{z_i^{(t)}\}\). We now state the convergence guarantee of ASGD and defer its proof to Appendix~\ref{appendix:stochastic_opt_subroutine_ASGD}.
\begin{proposition}\label{prop:convergence_ASGD}
Denote the minimizer of $G$ by $\underline{x}$ and suppose that Assumptions~\ref{assum: non-quadratic} and~\ref{assum:vr_assum} hold. Let $\{y_t\}_{t \ge 1}$ be computed by Algorithm~\ref{alg:ASGD} with parameters
\begin{align*}
	\alpha_t = \frac{2}{t+1}, \quad  \gamma_t = \frac{8L}{ t(t+1)}, \quad \tilde \mu = \frac{\mu}{2}, \quad \text{and}\quad m_t = \ceil{\frac{256\zeta^2  t}{\mu L}}.
\end{align*}
For any $T \ge 1 $, we have 
\begin{align*}
	\EE[G(y_T) - \inf G \mid \tilde \cF_1] \le \frac{4L \|\tilde x - \underline{x}\|^2}{T(T+1)} + \frac{ \mu \|\tilde x -\underline{x}\|^2}{64}.
\end{align*}
In particular, for any $T \ge 16 \sqrt{\frac{L}{ \mu}}$, we have 
\begin{align*}
	\EE[\|\widehat x - \underline{x}\|^2 \mid \tilde \cF_1] \le \frac{1}{16} \|\tilde x - \underline{x}\|^2.
\end{align*}
\end{proposition}

Unlike SGD, ASGD requires $\sum_{t=1}^{T} m_t = O\left(\sqrt{\frac{L}{\mu}} + \frac{\zeta^2}{\mu^2}\right)$ samples when $T = O\left(\sqrt{\frac{L}{\mu}}\right)$, which matches the optimal oracle complexity. 

\section{Instance-dependent guarantees for overall algorithm} \label{sec:overall_complexity}
In this section, we establish convergence guarantees for Algorithm~\ref{alg:metaVR}. We first prove a general convergence theorem for Algorithm~\ref{alg:metaVR} under abstract assumptions on the inner loop. We then specialize this result to concrete settings and specific subroutines, yielding explicit, instance-optimal non-asymptotic guarantees.
\begin{proposition}\label{prop:meta_convergence}
Suppose that there exist scalars $C_1, C_2 \ge 0$ such that  for any $1\le k \le K$, when $N_k \ge C_1$, the output of each epoch of Algorithm~\ref{alg:metaVR}  satisfies 
\begin{align}\label{eqn:half_contraction_exp}
	\EE[\|\widehat x_k - x^\star\|^2] \le \frac{1}{2} \EE[\|\widehat x_{k-1} - x^\star\|^2] + \frac{C_2}{N_k}.
\end{align}
For any $N \ge 1$, set $N_k \ge \max \left\{ C_1, \left(\frac{3}{4} \right)^{K+1 -k} \cdot N \right\}$. Then  we have 
\begin{align*}
	\EE[\|\widehat x_K - x^\star\|^2] \le \frac{1}{2^K} \|\widehat x_0 - x^\star\|^2 + \frac{4C_2 }{N}.
\end{align*}
\end{proposition}
\begin{proof}
Applying the bound~\eqref{eqn:half_contraction_exp} recursively, we have
\begin{align*}
	\EE[\|\widehat x_K - x^\star\|^2] &\le \frac{1}{2^K}\|\widehat x_0 -x^\star\|^2 +  \sum_{k=1}^{K} \frac{C_2}{2^{K-k} N_k}\\
	&\le \frac{1}{2^K}\|\widehat x_0 -x^\star\|^2 + \sum_{k=1}^{K} \left(\frac{2}{3}\right)^{K-k}\frac{4C_2}{3 N} \\
	&\le   \frac{1}{2^K}\|\widehat x_0 -x^\star\|^2 + \frac{4C_2}{N}.
\end{align*}
\end{proof}
Next, we present the main convergence results for quadratic problems and general non-quadratic problems. In the first case, we present one result for when the algorithm $\cA$ is SGD and another result for when the algorithm $\cA$ is ASGD. For non-quadratic problems, we only present the result for $\cA$ given by ASGD.

\subsection{Quadratic objectives and least squares}\label{sec:quadratic_objective}
Consider the case where the population objective is a quadratic function, i.e., $\liph =0$. Below is the convergence result when SGD is used as the subroutine $\cA$. 
\begin{theorem}\label{thm:vr_quadratic}
Suppose that  Assumption~\ref{assum: non-quadratic} holds with parameter $L \ge \mu >0$ and $\liph = 0$, and that Assumption~\ref{assum:vr_assum} holds. Write $A = \nabla^2 F(x^\star)$. Let $N$ be a positive integer. Assume that we use Algorithm~\ref{alg:SGD} (SGD) as a subroutine for Algorithm~\ref{alg:metaVR} and its parameters satisfy
\begin{align*}
	\eta \le \min \left\{\frac{1}{2L}, \frac{\mu }{256\zeta^2}\right\}, \quad T \ge \max\left\{\frac{128}{\eta \mu}, 64\right\} \quad \text{and} \quad N_k \ge \max \left\{\frac{32\zeta^2}{\mu^2}, \left(\frac{3}{4}\right)^{K+1-k} \cdot N \right\}.
\end{align*}
Setting  $w_0 = 0$ and $w_i = \frac{1}{T+1}$ for all $1\le i \le T+1$, the output of $\algname$ satisfies
\begin{align*}
	\EE[\|\widehat x_K - x^\star\|^2] \le \frac{1}{2^K} \|\widehat x_0 - x^\star\|^2 + \frac{20\cdot \EE[\|A^{-1} \nabla f(x^\star, z)\|^2]}{N} .
\end{align*}

\end{theorem}
\begin{proof}
In each epoch $k$ of Algorithm~\ref{alg:metaVR}, we denote the minimizer of $G(x) = F(x) - \dotp{\nabla F(\tilde x) - \widehat \nabla f(\tilde x), x}$ by $\underline{x}_k$. 
By Proposition~\ref{prop:convergence_SGD} and Young's inequality, we have 
\begin{align*}
	\EE[\|\widehat x_k - \underline{x}_k\|^2] &\le \frac{1}{16} \EE[\|\widehat x_{k-1} - \underline{x}_k\|^2] \\
	&\le \frac{1}{8} \EE[\|\widehat x_{k-1} - x^\star\|^2] + \frac{1}{8} \EE[\| \underline{x}_k - x^\star\|^2]
\end{align*}
On the other hand, by Lemma~\ref{lem:vr_quadratic_key} and our choice of $N_k$,
\begin{align*}
	\EE[\|\underline{x}_k - x^\star\|^2] \le \frac{2 \cdot \EE[\|A^{-1} \nabla f(x^\star, z)\|^2]}{N_k} + \frac{1}{16} \EE[\|\widehat x_{k-1} - x^\star\|^2].
\end{align*}
Therefore, by Young's inequality, we have
\begin{align*}
	\EE[\|\widehat x_k - x^\star\|^2] &\le 2 \EE[\|\widehat x_k - \underline{x}_k\|^2] + 2 \EE[ \|\underline{x}_k - x^\star\|^2] \\
	&\le \frac{3}{8} \EE[\|\widehat x_{k-1} - x^\star\|^2] + \frac{1}{4}\EE[\|\underline{x}_k - x^\star\|^2 ] + \frac{4 \cdot \EE[\|A^{-1} \nabla f(x^\star, z)\|^2]}{N_k}\\
	&\le \frac{1}{2} \EE[\|\widehat x_{k-1} - x^\star\|^2] + \frac{5 \cdot \EE[\|A^{-1} \nabla f(x^\star, z)\|^2]}{N_k} .
\end{align*}
The desired inequality then follows from Proposition~\ref{prop:meta_convergence}.  
\end{proof}
Below is a direct application of Theorem~\ref{thm:vr_quadratic}. We defer its proof to Appendix~\ref{appendix:quadratic_complexity}.
\begin{corollary}\label{cor:quadratic_complexity}
Suppose that  Assumption~\ref{assum: non-quadratic} holds with parameter $L \ge \mu >0$ and $\liph = 0$, and that Assumption~\ref{assum:vr_assum}  holds.  Write $A = \nabla^2 F(x^\star)$. Let $n$ denote the total number of samples. Suppose that $$n \gtrsim_{\log} \frac{L}{\mu} + \frac{\zeta^2}{\mu^2} + 1.$$ 	There exists a choice of parameters $(\eta, T, K, N_k)$ (made explicit in the proof) of Algorithm~\ref{alg:metaVR} with Algorithm~\ref{alg:SGD} (SGD) such that the total number of samples used satisfies 
$TK + \sum_{k=1}^{K} N_k \le n,$
and the output of $\algname$ satisfies
\begin{align*}
	\EE[\|\widehat x_K - x^\star\|^2] \le \frac{121 \cdot  \EE[\|A^{-1} \nabla f(x^\star, z)\|^2]}{n}.
\end{align*}
In particular, when $\|\cdot\|$ is the $\ell_2$ norm, we have 
\begin{align}\label{eqn:upper_bound_quadratic_SGD}
	\EE[\|\widehat x_K - x^\star\|_2^2] \le   \frac{121 \cdot\trace{A^{-1}\Sigma A^{-1}}}{n} .
\end{align}
\end{corollary}

We make a few comments to interpret the results when $\|\cdot\|$ is the $\ell_2$ norm. First, by Proposition~\ref{prop: mainlowerboundtheorem_quadratic_intro}, the upper bound~\eqref{eqn:upper_bound_quadratic_SGD} matches the non-asymptotic instance-dependent lower bound up to universal constant, so our algorithm is optimal when the sample size $n$ exceeds $n_0 = \tilde O\left(\frac{L}{\mu} + \frac{\zeta^2}{\mu^2}\right)$.  Second,  the threshold $n_0$ consists of two parts: a deterministic contribution $\tilde O\paren{\frac{L}{\mu}}$  and a stochastic contribution $\tilde O\left(\frac{\zeta^2}{\mu^2}\right)$. The deterministic contribution characterizes the complexity of the algorithm when there is no stochastic noise in the problem. This oracle complexity matches the standard gradient descent algorithm to solve smooth and strongly convex optimization problems~\cite{nesterov2018lectures}. The stochastic part $\tilde O(\frac{\zeta^2}{\mu^2})$ is due to the noise in the stochastic observation model, and it goes to zero as the noise level diminishes to zero.  This is a necessary threshold, as pointed out in Section~\ref{sec:suboptimality}. 

Next, we improve the deterministic contribution to~$\tilde O\left(\sqrt{\frac{L}{\mu}} \right)$ by switching the subroutine from Algorithm~\ref{alg:SGD} to Algorithm~\ref{alg:ASGD}. The proof of this result follows verbatim the proof of Theorem~\ref{thm:vr_quadratic}, except that we apply Proposition~\ref{prop:convergence_ASGD} instead of Proposition~\ref{prop:convergence_SGD}. Consequently, we omit the proof for brevity.
\begin{theorem}\label{thm:vr_quadratic_ac}
Suppose that  Assumption~\ref{assum: non-quadratic} holds with parameter $L \ge \mu >0$ and $\liph = 0$, and that Assumption~\ref{assum:vr_assum} holds. Write $A = \nabla^2 F(x^\star)$. Let $N$ be a positive integer. Assume that we use Algorithm~\ref{alg:ASGD} (ASGD) as a subroutine for Algorithm~\ref{alg:metaVR} and its parameters satisfy
\begin{align*}
	\alpha_t = \frac{2}{t+1}, \quad \gamma_t = \frac{8L}{ t(t+1)}, \quad \tilde \mu = \frac{\mu}{2}, \quad m_t = \ceil{\frac{256\zeta^2  t}{\mu L}},  \quad T = \ceil{16\sqrt{\frac{L}{ \mu}}},
\end{align*}
and $ N_k \ge \max \left\{\frac{32\zeta^2}{\mu^2}, \left(\frac{3}{4}\right)^{K+1-k} \cdot N \right\}.$ 
The output of $\algname$ satisfies
\begin{align*}
	\EE[\|\widehat x_K - x^\star\|^2] \le \frac{1}{2^K} \|x_0 - x^\star\|^2  + \frac{20 \cdot  \EE[\|A^{-1} \nabla f(x^\star, z)\|^2]}{N} .
\end{align*}
\end{theorem}
Below is a direct application of Theorem~\ref{thm:vr_quadratic_ac}. We defer its proof to Appendix~\ref{appendix:complexity_quad_ASGD}.

\begin{corollary}\label{cor:complexity_quad_ASGD}
Suppose that  Assumption~\ref{assum: non-quadratic} holds with parameter $L \ge \mu >0$ and $\liph = 0$, and that Assumption~\ref{assum:vr_assum} holds.  Write $A = \nabla^2 F(x^\star)$.  Let $n$ denote the total number of samples. Suppose that 
$$n \gtrsim_{\log} \sqrt{\frac{L}{\mu}} + \frac{\zeta^2}{\mu^2} +1.$$ 
There exists a choice of parameters of Algorithm~\ref{alg:metaVR} with Algorithm~\ref{alg:ASGD} (ASGD) such that the total number of samples used is less than $n$
and the output of $\algname$ satisfies
\begin{align*}
	\EE[\|\widehat x_K - x^\star\|^2] \le \frac{121 \cdot \EE[\|A^{-1} \nabla f(x^\star, z)\|^2]}{n} .
\end{align*}
In particular, when $\|\cdot\|$ is the $\ell_2$ norm, we have 
\begin{align*} 
	\EE[\|\widehat x_K - x^\star\|_2^2] \le   \frac{121 \cdot\trace{A^{-1}\Sigma A^{-1}}}{n} .
\end{align*}
\end{corollary}

As a consequence, when $\|\cdot\|$ is the $\ell_2$ norm and the sample size satisfies
\[
n \ge n_0 = \tilde O\left(\sqrt{\frac{L}{\mu}} + \frac{\zeta^2}{\mu^2}\right),
\]
Algorithm~\ref{alg:metaVR} with Algorithm~\ref{alg:ASGD} as a subroutine achieves the optimal statistical rate up to logarithmic factors.

This result may be viewed as a refinement of the classical worst-case guarantees of Ghadimi and Lan~\cite{ghadimi2013optimal}. In their framework, the stochastic oracle is characterized by a single scalar $\sigma^2$ that uniformly bounds the variance of gradient noise. In the smooth, strongly convex setting, their multi-stage AC-SA method attains an expected suboptimality bound of order $\sigma^2/(\mu n)$ once the number of samples exceeds $\tilde O(\sqrt{L/\mu})$. By strong convexity, this translates into a squared-distance guarantee of order $\sigma^2/(\mu^2 n)$.  Since in our notation
\[
\trace{\Sigma}=\EE\!\left[\|\nabla f(x^\star,z)\|_2^2\right]\le \sigma^2,
\]
their statistical term is at least of order $\frac{\trace{\Sigma}}{\mu^2 n}$. Our bound,
$\frac{\trace{A^{-1}\Sigma A^{-1}}}{n}$, is finer because it resolves the interaction between the local Hessian $A=\nabla^2 F(x^\star)$ and the noise covariance $\Sigma$. Using only $A \succeq \mu I$ recovers
\[
\frac{\trace{A^{-1}\Sigma A^{-1}}}{n} \le \frac{\sigma^2}{\mu^2 n}.
\]
Thus, we recover the classical worst-case rate as a coarse upper bound, but our instance-dependent quantity on the LHS can be much smaller in practice, as evidenced by the following example.

\subsubsection{Consequence for least-squares regression}\label{sec:improve_least_squares}
Now we apply the results in Section~\ref{sec:quadratic_objective} to the least-squares regression introduced in Section~\ref{sec:background}. Notably, we improve the best-known guarantees by a factor of the condition number.
\begin{corollary}
Let $\tilde \kappa$ be the statistical condition number defined in~\eqref{eqn:stochastic_cond_number}. Then, when $n  \gtrsim_{\log} \tilde \kappa$, with proper choice of hyperparameters, the $\algname$ Algorithm~\ref{alg:metaVR} implemented with Algorithm~\ref{alg:SGD} and $\|\cdot\| = \|\cdot\|_H$ satisfies
\begin{align*}
	\EE[F(\widehat x_K) - \inf F]  = \EE[\|\widehat x_K - x^\star\|_H^2] \le \frac{121\cdot \trace{H^{-1}\Sigma}}{n}.
\end{align*}
\end{corollary}
\begin{proof}
Note  that by the discussion in Section~\ref{sec:background}, Assumption~\ref{assum: non-quadratic} holds with $\mu = L = 1$ and $\liph = 0$, and Assumption~\ref{assum:vr_assum} holds with $\zeta^2 = \tilde \kappa$. The result then follows from Corollary~\ref{cor:quadratic_complexity}.
\end{proof}
\begin{remark}
First, we note that one can also use acceleration (i.e., ASGD instead of SGD in the inner loop) and obtain a similar result; there is no improvement since the condition number is unity in this norm (i.e., $\frac{L}{\mu} =1$). Second, note that our result holds under exactly the same condition as~\cite{jain2018accelerating}, but the sample size we need to match the asymptotically optimal rate is $\tilde O(\tilde \kappa)$. To our knowledge,~\cite[Corollary 2]{jain2018accelerating} has the best-known instance-dependent rate for least-squares regression, and it requires $\tilde O(\sqrt{\kappa \tilde \kappa})$ samples, where $\kappa = \frac{R^2}{\mu}$ and $R^2$ is the smallest non-negative number such that 
\begin{align*}
	\EE[\|\xi\|_2^2 \xi \xi^\top] \preceq R^2 H.
\end{align*}
It is straightforward to verify that $\kappa \ge \tilde \kappa$, so our complexity is always better. When $\xi$ is Gaussian,  one can verify that $\tilde \kappa = O(d)$ while $\kappa = O(\frac{\trace{H}}{\mu})$~\cite{jain2018accelerating}. In this setting, our rate improves theirs by a factor depending on the condition number of $H$.
\end{remark}

\subsection{General non-quadratic objective}
We now turn to the general setting where $\liph >0$. For simplicity, we present only the stronger accelerated results. The proof follows the same structure as that of Theorem~\ref{thm:vr_quadratic}---except that Lemma~\ref{lem:vr_quadratic_key} is replaced by Lemma~\ref{lem:vr_key_non_qud}, and Proposition~\ref{prop:convergence_SGD} is replaced by Proposition~\ref{prop:convergence_ASGD}---we omit the details for brevity.
\begin{theorem}\label{thm:vr_nonquadratic}
Suppose that  Assumption~\ref{assum: non-quadratic} holds with parameter $L \ge \mu >0$ and $\liph > 0$, and that Assumptions~\ref{assum:vr_assum} and~\ref{assum:light_tail_non_quadratic} hold. Write $A = \nabla^2 F(x^\star)$. Let $N$ be a positive integer. Assume that we use Algorithm~\ref{alg:ASGD} (ASGD) as a subroutine  for Algorithm~\ref{alg:metaVR} and its parameters satisfy
\begin{align*}
	\alpha_t = \frac{2}{t+1}, \quad \gamma_t = \frac{8L}{t(t+1)}, \quad \tilde \mu = \frac{\mu}{2}, \quad m_t = \ceil{\frac{256\zeta^2 t}{\mu L}},  \quad T = \ceil{32\sqrt{L/\mu}},
\end{align*}
and 
\begin{align*}
	N_k \ge \max\left\{ \frac{128(\zeta^2 + d\sigma_2^2)}{\mu^2}, \left( \frac{3}{4} \right)^{K+1-k} \cdot N, \ch \right\}, 
\end{align*}
where $\ch$ is the smallest integer larger than 
\begin{align*}
	\max \left\{ \frac{1024 \liph^2  \sigma_1^2}{\lmin^2}, \frac{128\liph \sigma_1}{\lmin}\right\} \cdot \max\left\{ d, \log \left(\frac{4 \lmin^2}{\mu^2 \liph^2 \trace{\Lambda}}\right) \right\}
\end{align*}
and 
\begin{align*}
	2\ceil{\max \left\{ \frac{1024 \liph^2  \sigma_1^2}{\lmin^2}, \frac{128\liph \sigma_1}{\lmin}\right\}} \left( \log \left( \max \left\{ \frac{1024 \liph^2  \sigma_1^2}{\lmin^2}, \frac{128\liph \sigma_1}{\lmin}\right\}\right) +1\right).
\end{align*}
Then the output of $\algname$ satisfies 
\begin{align*}
	\EE[\|\widehat x_K -  x^\star\|^2] \le \frac{1}{2^K} \| \widehat x_0 - x^\star\|^2 + \frac{84 \cdot \EE[\|A^{-1} \nabla f(x^\star, z)\|^2]}{N}  .
\end{align*}
\end{theorem}
The following is a direct consequence of Theorem~\ref{thm:vr_nonquadratic}. Its proof mirrors the proof of Corollary~\ref{cor:quadratic_complexity}, so we omit it for brevity.
\begin{corollary}\label{cor:vr_nonquadratic}
Suppose that  Assumptions~\ref{assum: non-quadratic},~\ref{assum:vr_assum}, and~\ref{assum:light_tail_non_quadratic} hold.  Write $A = \nabla^2 F(x^\star)$, and let $C_H$ be defined as in Theorem~\ref{thm:vr_nonquadratic}. Let $n$ denote the total number of samples. Suppose that 
$$n \gtrsim_{\log} \sqrt{ \frac{L}{\mu} }+ \frac{\zeta^2 + d\sigma_2^2}{\mu^2} + C_H.$$ 
There exists a universal constant \(C>0\) and a parameter choice for Algorithm~\ref{alg:metaVR}, with Algorithm~\ref{alg:ASGD} (ASGD) as a subroutine, such that the total sample complexity is less than \(n\), and the output of $\algname$ satisfies
\begin{align*}
	\EE[\|\widehat x_K - x^\star\|^2] \le \frac{C \cdot  \EE[\|A^{-1} \nabla f(x^\star, z)\|^2]}{n}.
\end{align*}
In particular, when $\|\cdot\|$ is the $\ell_2$ norm, we have 
\begin{align}\label{eqn:upper_bound_non_quad}
	\EE[\|\widehat x_K - x^\star\|_2^2] \le   \frac{C \cdot\trace{A^{-1}\Sigma A^{-1}}}{n} .
\end{align}
\end{corollary}

As we will show in Section~\ref{sec:lower_bound}, when $\|\cdot\| = \| \cdot \|_2$ and the sample size $n$ exceeds $n_0 = \tilde O\left( \sqrt{ \frac{L}{\mu} }+ \frac{\zeta^2 + d\sigma_2^2}{\mu^2} + C_H \right)$, the upper bound~\ref{eqn:upper_bound_non_quad} matches the non-asymptotic instance-dependent lower bound up to a constant factor. Similar to the quadratic case, the deterministic portion of $n_0$, given by $\tilde O\left(\sqrt{\frac{L}{\mu}}\right)$, is   optimal for first-order methods. 

Compared to the quadratic case, the stochastic part of the sample requirement contains two additional terms: $\ch$ and $\frac{d\sigma_2^2}{\mu^2}$. The first term, $\ch$, captures the effect of the nonconstant Hessian in the non-quadratic setting. It diminishes as the Hessian Lipschitz constant $\liph$ tends to zero and vanishes entirely when the objective is quadratic. As we will show in Proposition~\ref{prop:necessity_liph_intro}, the dependence of $C_H$ on the $\liph^2$ term is necessary for general non-quadratic problems. The second term, $\frac{d\sigma_2^2}{\mu^2}$, plays a role analogous to $\frac{\zeta^2}{\mu^2}$, reflecting the need for a sufficiently large sample size to control the effect of multiplicative noise. In many natural cases, $d\sigma_2^2$ and $\zeta^2$ are of the same order; for instance, when $\nabla f(x, \tilde A, \tilde b) = \nabla F(x) + \tilde A x + \tilde b$ with $\tilde A$ having i.i.d. standard Gaussian entries and $\tilde b$ a standard Gaussian vector, under the standard Euclidean norm, we have $\zeta^2 = d\sigma_2^2$.

\subsubsection{Consequence for generalized linear model with $\ell_2$-regularization}
We now apply Theorem~\ref{thm:vr_nonquadratic} and Corollary~\ref{cor:vr_nonquadratic} to the generalized linear model introduced in Section~\ref{sec:background}.  Recall that the population objective is given by
\begin{align*}
	\min_{x\in\RR^d}\; F(x)
	:= \EE_{(\xi,y)\sim P}[\ell(x,\xi,y)] +  \frac{\lambda}{2}\|x\|_2^2 = \EE_{(\xi,y)\sim P}\left[u(\dotp{x,\xi}) - y \dotp{x,\xi}\right] + \frac{\lambda}{2}\|x\|_2^2,
\end{align*}
\begin{proposition}
Set $\|\cdot\| = \| \cdot \|_2$ and suppose that the generalized linear model~\eqref{eqn:generalized_linear_model} satisfies Assumption~\ref{assum:generalized_linear_model}.  Let $n$ denote the total number of samples and suppose that 
\begin{align*}
	n \gtrsim_{\log} \sqrt{\frac{d L_1 \sigma^2 + \lambda}{\gamma + \lambda}} + \frac{d L_1^2 \sigma^4}{(\gamma +\lambda)^2} + \frac{d L_2^2 \sigma^8 (\sigma_\star^2 + \sigma_y^2)}{(\gamma +\lambda)^4}.
\end{align*}
There exists a choice of parameters of Algorithm~\ref{alg:metaVR} with Algorithm~\ref{alg:ASGD} as a subroutine, such that the total number of samples is less than $n$ and the output of $\algname$ satisfies
\begin{align*}
	\EE[\|\widehat x_K - x^\star\|_2^2] \le \frac{361 \cdot \trace{\Lambda}}{n}.
\end{align*}
Here $\Lambda= \nabla^2 F(x^\star)^{-1} \cov(\nabla \ell (x^\star, \xi, y))  \nabla^2 F(x^\star)^{-1} $.
\end{proposition}
\begin{proof} 
	First, for any vector $v$ of unit $\ell_2$ norm, we have 
	\begin{align*}
		\|\dotp{v, \nabla_x \ell(x, \xi, y)}\|_{\psi_1} &\le \|\dotp{v, \nabla_x \ell(x^\star, \xi, y)}\|_{\psi_1} +\|\dotp{v, \nabla_x \ell(x, \xi, y) - \nabla_x \ell(x^\star, \xi, y)}\|_{\psi_1}\\
		&\le  \|  u'(\dotp{x^\star,\xi}) \dotp{v,\xi}  \|_{\psi_1} + \| y \dotp{v, \xi} \|_{\psi_1} + \| |u'(\dotp{x,\xi})- u'(\dotp{x^\star,\xi})| \abs{\dotp{v, \xi}}\|_{\psi_1}\\
		&\le  \|  u'(\dotp{x^\star,\xi}) \dotp{v,\xi}  \|_{\psi_1} + \| y \dotp{v, \xi} \|_{\psi_1} + \| L_1\abs{\dotp{x-x^\star,\xi}} \abs{\dotp{v, \xi}}\|_{\psi_1}\\
		&\le \sigma (\sigma_\star + \sigma_y)  + L_1 \sigma^2 \|x-x^\star\|_2,
	\end{align*}
	where the first two inequalities are due to the triangle inequality, the third follows from the Lipschitz continuity of $u'$, and the last follows from Assumption~\ref{assum:generalized_linear_model} combined with the standard bound $\|XY\|_{\psi_1} \le \|X\|_{\psi_2}\|Y\|_{\psi_2}$ for any random variables $X$ and $Y$; see~\cite[Lemma 2.8.6]{vershynin2018high}. Moreover, by Proposition 2.8.1 and Exercise 2.44 in~\cite{vershynin2018high}, there exists some universal constant $C$ such that Assumption~\ref{assum:light_tail_non_quadratic} holds for $\nabla \ell(x, \xi,y)$ with $\sigma_1 = C\sigma (\sigma_\star + \sigma_y)$ and $\sigma_2 =  CL_1 \sigma^2 $. Combining Corollary~\ref{cor:vr_nonquadratic} with the calculations in Section~\ref{sec:background}, the result follows.
\end{proof}

\section{Local lower bounds in the general setting}\label{sec:lower_bound}
In this section, we prove a non-asymptotic local minimax lower bound for the general problem~\eqref{eqn:objective}. To this end, we need to define the class of instances over which the lower bound will hold. For any given $n \ge 1$, a $C^2$ strongly convex function $F$, and a positive semi-definite covariance matrix $\Sigma$, we denote $\nabla^2 F(x^\star(F))$ by $A$ and recall the instance class we defined in Section~\ref{sec:suboptimality}:
\begin{align*}
\nclass (n, F, \Sigma) :=  \left\{ (f,P) \left| 
\begin{aligned}
	&\text{ $\nabla^2 F_{f,P} \equiv \nabla^2 F$ and $ \|x^\star(F_{f,P}) - x^\star(F)\|_2 \le 2\cdot \sqrt{\frac{\trace{A^{-1} \Sigma A^{-1}}}{n}}$,} \\
	& \text{ $\nabla f(x^\star(F_{f,P}), z)$ has distribution  $N(0, \Sigma)$ when $z\sim P$}
\end{aligned}
\right.\right\},
\end{align*}
This instance class is identical to the one defined earlier, and we refer the reader to Section~\ref{sec:suboptimality} for a detailed discussion of why this is a reasonable local neighborhood to define. 

Next, we state the general lower bound, which yields Proposition~\ref{prop: mainlowerboundtheorem_quadratic_intro} as a special case by letting $\liph \rightarrow 0$.

\begin{theorem}\label{thm: mainlowerboundtheorem_intro}
Suppose that $F$ satisfies Assumption~\ref{assum: non-quadratic} with parameters $L \ge \mu >0 , \liph > 0$. We denote $\nabla^2 F(x^\star(F))$ by $A$. For any  positive definite matrix $\Sigma$ and any $n \ge  64\cdot  \liph^2 \cdot \trace{A^{-1} \Sigma  A^{-1}} $,  we have 
\begin{align}
	\quad \inf_{\widehat x_n \in \widehat \cX_n} \sup_{(f,P)\in \nclass(n,F,\Sigma)} \EE_{z_i\overset{\mathrm{iid}}{\sim} P}[\|\widehat x_n(\{z_i\}_{i=1}^{n}, f) - x^\star(F_{f,P})\|_2 ^2] \ge	
	\frac{ \trace{A^{-1} \Sigma A^{-1}}}{4(\pi^2 +1) n}  .\label{eqn:two_regimes_non_quadratic}
\end{align}
\end{theorem}

We defer the proof to Appendix~\ref{sec:lower_bound_proof}.  Theorem~\ref{thm: mainlowerboundtheorem_intro} shows that the geometry-dependent quantity $\frac{ \trace{A^{-1} \Sigma A^{-1}}}{4(\pi^2 +1) n}$ characterizes the fundamental difficulty of the problem when $n \ge 64  \liph^2  \trace{A^{-1} \Sigma  A^{-1}}$. This quantity is of the same order as the squared radius of the neighborhood in the definition of $\nclass(n,F,\Sigma)$. Moreover, one can replace the condition
$$\|x^\star(F_{f,P}) - x^\star(F)\|_2  \le 2\cdot\sqrt{\frac{ \trace{A^{-1} \Sigma A^{-1}}}{n}} $$
in the definition of $\nclass(n, F, \Sigma)$ by 
$$
\|x^\star(F_{f,P}) - x^\star(F)\|_2  \le c_n,
$$
where $c_n \ge 0$ is any constant smaller than $2\cdot\sqrt{\frac{ \trace{A^{-1} \Sigma A^{-1}}}{n}}$. Upon doing so,~\eqref{eqn:two_regimes_non_quadratic} still holds with 	$\frac{ \trace{A^{-1} \Sigma A^{-1}}}{4(\pi^2 +1) n}$ replaced by $\Omega(c_n^2)$. Thus, for $c_n \leq 2\cdot\sqrt{\frac{ \trace{A^{-1} \Sigma A^{-1}}}{n}}$, no estimator can improve on the squared error of the trivial estimator, which always outputs $x^\star(F)$, by more than a universal constant factor.
It remains to explain whether the threshold $ 64 \cdot \liph^2 \cdot \trace{A^{-1} \Sigma  A^{-1}}$ is necessary for non-quadratic problems in which $L_H > 0$. To this end, we define a modified instance class
\begin{align*}
	\tilde \nclass (n, F, \Sigma, r) :=  \left\{ (f,P) \left| 
	\begin{aligned}
		&\text{ $\nabla^2 F_{f,P} \equiv \nabla^2 F$ and $ \|x^\star(F_{f,P}) - x^\star(F)\|_2 \le r$,} \\
		& \text{ $\nabla f(x^\star(F_{f,P}), z)$ has distribution  $N(0, \Sigma)$ when $z\sim P$}
	\end{aligned}
	\right.\right\}.
\end{align*}

In the next proposition, we construct a specific instance showing that if the sample size is not at the order of $\liph^2 \cdot \trace{A^{-1} \Sigma  A^{-1}}$, no estimator can obtain a guarantee that  depends solely on the local geometry of the problem. 

\begin{proposition}\label{prop:necessity_liph_intro}
Suppose that $d > 1$. For any parameters $L \ge 3\mu >0$, $\liph > 0$ there exists a function $F\colon  \RR^d \rightarrow \RR$ satisfying Assumption~\ref{assum: non-quadratic} with parameters $(\mu, L,\liph)$ and having Hessian $\nabla^2 F(x^\star(F)) = \frac{L+\mu}{2} I$ such that for any $n \le \frac{  \log 2 \cdot \liph^2  d}{2\cdot 144^2}$,  we have

\begin{align}\label{eqn:necessity_liph}
	\inf_{\widehat x_n \in \widehat \cX_n} \sup_{(f,P)\in \tilde  \nclass\left(n,F, \frac{(L-\mu)^2}{4} I, \frac{36 (L-\mu)}{\liph \mu}\right)} \EE_{z_i\overset{\mathrm{iid}}{\sim} P}[\|\widehat x_n(\{z_i\}_{i=1}^{n}, f) - x^\star(F_{f,P})\|_2 ^2] \ge \frac{9 (L-\mu)^2}{ \liph^2  \mu^2}.
\end{align}
\end{proposition}
We defer the proof to Appendix~\ref{sec:necessity_liph_proof}. Now, fix $\liph$ and $L$, and let $\mu \rightarrow 0$ in Proposition~\ref{prop:necessity_liph_intro}. In this regime, if the number of samples is less than 
$$\Omega(\liph^2 d) = \Omega\left( \liph^2 \trace{\underbrace{\nabla^2 F(x^\star(F))^{-1} \cdot \cov(\nabla f(x^\star(F), z)) \cdot \nabla^2 F(x^\star(F))^{-1}}_{A^{-1} \Sigma A^{-1}}} \right),$$
then the lower bound~\eqref{eqn:necessity_liph} applies and can be made arbitrarily large.
Note that as $\mu \rightarrow 0$, the Hessian $\nabla^2 F(x^\star(F))$ converges to $\frac{L}{2} I$ and the covariance of the gradient noise stabilizes at $\frac{L^2}{4} I$. 
In other words, while the local geometry of $F$ at the solution stabilizes, the lower bound~\eqref{eqn:necessity_liph} can diverge. Therefore, the $\liph^2 \cdot \trace{A^{-1} \Sigma  A^{-1}}$ threshold in Theorem~\ref{thm: mainlowerboundtheorem_intro} is necessary in general.

\medskip

To conclude, we compare the lower bound given by Theorem~\ref{thm: mainlowerboundtheorem_intro} with the upper bounds derived in Section~\ref{sec:better_algorithm}. First, note that for quadratic problems, $\liph = 0$, and Theorem~\ref{thm: mainlowerboundtheorem_intro} specializes to Proposition~\ref{prop: mainlowerboundtheorem_quadratic_intro}. As discussed in Section~\ref{sec:quadratic_objective}, we have that for any $n \ge 1$, the non-asymptotic local minimax lower bound is 
$$ \Omega \left(\frac{ \trace{A^{-1} \Sigma A^{-1}}}{n}\right)$$ 
and any first-order algorithm requires  $\Omega\left(\sqrt{\frac{L}{\mu}} + \frac{\zeta^2}{\mu^2}  \right)$ samples to achieve it.
By Theorem~\ref{thm:vr_quadratic_ac}, our  algorithm is thus non-asymptotically instance-optimal up to logarithmic factors. 

For general non-quadratic objectives, consider in conjunction Theorem~\ref{thm: mainlowerboundtheorem_intro}, Proposition~\ref{prop:necessity_liph_intro},   our discussion on the sample size requirement $\Omega\left(\frac{\zeta^2}{\mu^2}\right)$  in Section~\ref{sec:suboptimality}, and the classical oracle complexity lower bound for first-order methods~\cite{complexity}. These together imply that,  for any  $n \ge \Omega(\liph^2 \trace{A^{-1}\Sigma A^{-1}})$, the non-asymptotic local minimax lower bound is 
$$ \Omega \left(\frac{ \trace{A^{-1} \Sigma A^{-1}}}{n}\right),$$
and any reasonable algorithm requires $\Omega\left( \sqrt{\frac{L}{\mu}} + \frac{\zeta^2}{\mu^2} + \liph^2  \trace{A^{-1} \Sigma A^{-1}}\right)$ samples to achieve this rate. On the other hand, focusing on the sample complexity requirement of Theorem~\ref{thm:vr_nonquadratic}, we see that our algorithm requires 
$$
\tilde O\left(\sqrt{\frac{L}{\mu}} + \frac{\zeta^2 + d\sigma_2^2 }{\mu^2}  + \frac{ \liph^2 d \sigma_1^2}{\lambda_{\min}^2(A)} \right)
$$
samples to achieve the instance-optimal rate.
As noted following Corollary~\ref{cor:vr_nonquadratic}, the $d\sigma_2^2$ term in the numerator of $\frac{\zeta^2 + d\sigma_2^2}{\mu^2}$ arises from a proof artifact. Moreover,  $d\sigma_2^2$ is often of the same order as $\zeta^2$, suggesting that this term does not worsen the complexity order in certain scenarios. Considering the last term, note that the upper bound involves $\frac{ \liph^2 d \sigma_1^2}{\lambda_{\min}^2(A)}$, which may be larger than the corresponding lower bound term $\liph^2 \cdot \trace{A^{-1} \Sigma A^{-1}}$.  Obtaining a fully matched characterization in general problems is an interesting open problem.

\section{Discussion}

Our paper undertakes a non-asymptotic analysis of instance optimality in stochastic strongly convex and smooth optimization. While classical methods such as SAA and robust SA can be asymptotically optimal, we showed that they may still perform poorly at realistic sample sizes, even on simple quadratic problems. In particular, they fail to match a non-asymptotic local minimax lower bound that we developed in this work. To remedy this issue,  we introduced a framework based on a careful variance-reduction device that achieves the optimal instance-dependent statistical error up to logarithmic factors, and can be wrapped around an accelerated stochastic optimization subroutine. As a notable consequence, we obtained improved results for the generalization error of stochastic methods in linear regression, a problem that has seen extensive investigation in the past decade. Taken together, our results demonstrate that taking a non-asymptotic and instance-dependent perspective can yield robust algorithms that have strong theoretical guarantees as well as reliable practical performance.

Several directions remain open. On the technical side, it would be interesting to sharpen the remaining gaps between our upper and lower bounds in the general non-quadratic setting, and to determine whether the additional sample-size requirements in our analysis are fundamental or just proof artifacts. Beyond the i.i.d.\ setting studied here, an important next step is to extend the theory to dependent noise, where temporal dependence may substantially alter both the lower bounds and optimal algorithm design. Indeed, related investigations have been recently undertaken in several problems~\cite{mou2022optimal1, li2023accelerated,wu2025uncertainty,nakul2026multiscale}. Another natural direction is to generalize our framework to nonsmooth or constrained problems, where the relevant local geometry is more delicate, and the correct non-asymptotic notion of instance-optimality is still an active area of investigation~\cite{cheng2025geometrycomputationoptimalitystochastic}. More broadly, we hope this work helps motivate a general theory of finite-sample instance-optimality that bridges optimization complexity, statistical efficiency, and local problem geometry.

\subsection*{Acknowledgments}

This work was supported in part by NSF grant CCF-2107455, a Google Research Scholar award, and research awards/gifts from Adobe, Amazon and Mathworks. Liwei Jiang was also partially supported by ONR award N00014-22-1-215 and an ARC postdoctoral fellowship. We thank Katya Scheinberg for several helpful discussions.

\small
\bibliographystyle{plain}
\bibliography{bibliography}

\normalsize
\newpage

\begin{center}
\LARGE{\textbf{Appendix}}
\end{center}
\appendix
\section{Proofs deferred from  Section~\ref{sec:better_algorithm}}
In this appendix, we prove Lemma~\ref{lem:vr_key_non_qud}, Proposition~\ref{prop:convergence_SGD} and Proposition~\ref{prop:convergence_ASGD}.

\subsection{Proof of Lemma~\ref{lem:vr_key_non_qud}} \label{appendix:vr_nonquadratic}
Recall that $\underline{x}$ is the unique solution to the following equation: \begin{align*}
	\nabla F(\underline x) - (\nabla F(\tilde x) - \widehat \nabla f(\tilde x)) = 0,
\end{align*}
where $\widehat \nabla f(\tilde x)  = \frac{1}{n} \sum_{i=1}^{n} \nabla f(\tilde x,z_i)$. 	In addition, by the fundamental theorem of calculus, we have
\begin{align*}
	\nabla F(\underline{x}) = \underbrace{\int_{0}^{1} \nabla^2 F(x^\star + t(\underline{x} - x^\star)) dt}_{:= B(\underline{x})} \cdot (\underline{x} - x^\star).
\end{align*}

We denote $\frac{1}{n} \sum_{i=1}^{n} (\nabla f(\tilde x,z_i)  - \nabla F(\tilde x))$ by $\tilde \xi_n$ and $\frac{1}{n} \sum_{i=1}^{n} (\nabla f(x^\star,z_i)  - \nabla F(x^\star))$ by $\bar \xi_n$. We have
\begin{align}
\EE[\|\underline{x} - x^\star\|^2 \mid \tilde x] &= \EE[\|B(\underline{x})^{-1}\tilde \xi_n \|^2 \mid \tilde x ] \notag \\
&= \EE[\|B(\underline{x})^{-1} AA^{-1}\tilde \xi_n \|^21_{\|\underline{x} - x^\star\| \le  \frac{1}{2\liph}} \mid \tilde x ] +  \EE[\|B(\underline{x})^{-1}\tilde \xi_n \|^21_{\|\underline{x} - x^\star\|\ge \frac{1}{2\liph}} \mid \tilde x] \notag\\
& \le 4 \cdot \EE[\|  A^{-1}\tilde \xi_n \|^21_{\|\underline{x} - x^\star\| \le  \frac{1}{2\liph}} \mid \tilde x] +  \EE[\|B(\underline{x})^{-1}\tilde \xi_n \|^21_{\|\underline{x} - x^\star\|  \ge  \frac{1}{2\liph}} \mid \tilde x] \label{eqn: non_qua_eq1} \\
&\le 8\cdot \EE[\| A^{-1}\bar  \xi_n \| ^2 1_{\|\underline{x} - x^\star\| \le  \frac{1}{2\liph}}] + \frac{8 \zeta^2}{n  \lmin^2} \|\tilde x - x^\star\|^2 \notag \\
&\quad + \EE\left[\|B(\underline{x})^{-1} \tilde   \xi_n\|^2 1_{\|\underline{x} - x^\star\|  > \frac{1}{2\liph}}  \Big|\; \tilde x\right]   \label{eqn: non_qua_eq2}\\
& \le \frac{8 \cdot \EE[\|A^{-1} \nabla f(x^\star,z)\|^2]}{n} + \frac{8 \zeta^2}{n  \lmin^2}\|\tilde x - x^\star\| ^2  +  \EE\left[\|B(\underline{x})^{-1} \tilde   \xi_n\|^2 1_{\|A^{-1} \tilde  \xi_n\| > \frac{1}{4\liph}} \Big|\; \tilde x\right]\label{eqn: non_qua_eq3}\\
&\le  \frac{8\cdot \EE[\|A^{-1} \nabla f(x^\star,z)\|^2]}{n} + \frac{8\zeta^2}{n  \lmin^2}\|\tilde x - x^\star\|^2  +  \frac{1}{\mu^2}\EE\left[\|\tilde   \xi_n\|_*^2 1_{\|\tilde  \xi_n\|_* > \frac{\lmin}{4\liph}} \Big|\; \tilde x\right] \label{eqn: non_qua_eq4}
\end{align}
where the estimate~\eqref{eqn: non_qua_eq1} follows from Lemma~\ref{lem:optimal_trace_when_close}, the estimate~\eqref{eqn: non_qua_eq2} follows from Assumption~\ref{assum:vr_assum} and Lemma~\ref{lem:general_lambda_min},  the estimate~\eqref{eqn: non_qua_eq3} follows from  Lemma~\ref{lem:small_noise_close_solution}, and the estimate~\eqref{eqn: non_qua_eq4} follows from Lemma~\ref{lem:general_lambda_min}. Next, we bound the last term in~\eqref{eqn: non_qua_eq4}. Note that by Assumption~\ref{assum:light_tail_non_quadratic} and independence of samples, $ \tilde \xi_n$ is a sub-exponential vector with parameters $(\frac{\sigma_1^2 + \sigma_2^2\|\tilde x -x^\star\|^2}{n}, \frac{\sigma_1 + \sigma_2 \|\tilde x - x^\star\| }{n})$. By Lemma~\ref{lem:truncted_second_moment_light_tail} and Young's inequality, we have 
\begin{align*}
\frac{1}{\mu^2}\EE\left[\|\tilde   \xi_n\|_*^2 1_{\|\tilde  \xi_n\|_*  > \frac{\lmin}{4\liph}} \Big|\; \tilde x\right] 
&\le \frac{1}{\mu^2} \left(\frac{\lmin^2}{8\liph^2} + \frac{16(\sigma_1^2+\sigma_2^2 \|\tilde x - x^\star\| ^2)}{n}\right) e^{-\frac{n \cdot \lmin^2}{128 \liph^2  (\sigma_1^2 + \sigma_2^2\|\tilde x -x^\star\|^2)}+2d}\\
&\quad + \frac{1}{\mu^2}\left(\frac{\lmin^2}{4\liph^2} + \frac{96(\sigma_1+\sigma_2 \|\tilde x - x^\star\| )^2}{n^2}\right) e^{-\frac{n \cdot \lmin}{16 \liph  (\sigma_1+ \sigma_2\|\tilde x -x^\star\|)}+2d}.
\end{align*}
We now split the proof into two cases:

\paragraph{Case 1: $\|\tilde x - x^\star\|^2  \le \frac{\sigma_1^2}{\sigma_2^2}$.} First, by our assumption on $n$, it is straightforward to verify that
\begin{align}\label{eqn: bound_on_n}
\frac{n \cdot \lmin^2}{1024 \liph^2 \sigma_1^2} \ge   \max \left \{d, \log \left(\frac{\lmin^2}{2\mu^2 \liph^2 \EE[\|A^{-1} \nabla f(x^\star,z)\|^2]}\right),  \log n \right\}. 
\end{align}
Applying the $\|\tilde x -x^\star\|^2 \le \frac{\sigma_1^2}{\sigma_2^2}$ and~\eqref{eqn: bound_on_n}, we have
\begin{align*}
&\frac{1}{\mu^2} \left(\frac{\lmin^2}{8\liph^2} + \frac{16(\sigma_1^2 + \sigma_2^2\|\tilde x - x^\star\| ^2)}{n}\right) e^{-\frac{n \cdot\lmin^2}{128 \liph^2  (\sigma_1^2 + \sigma_2^2\|\tilde x -x^\star\|^2)}+2d}\\
&\qquad \le  \frac{\lmin^2}{4 \mu^2 \liph^2}  e^{-\frac{n \cdot \lmin^2}{512 \liph^2  \sigma_1^2} }\\
&\qquad \le  \frac{\lmin^2}{4 \mu^2 \liph^2}  e^{- \log \left(\frac{\lmin^2}{2\mu^2 \liph^2 \EE[\|A^{-1} \nabla f(x^\star,z)\|^2]}\right) - \log n}\\
&\qquad = \frac{\EE[\|A^{-1} \nabla f(x^\star,z)\|^2]}{2n}.
\end{align*}
Similarly, we can verify 
\begin{align*}
\frac{n \lmin^2}{1024 \liph^2 \sigma_1^2} \ge 1 \quad \text{and} \quad  \frac{n \lmin}{128 \liph \sigma_1} \ge \max\left\{d, \log \left(\frac{4\lmin^2}{\mu^2 \liph^2 \EE[\|A^{-1} \nabla f(x^\star,z)\|^2]}\right), \log n\right\}.
\end{align*}
and have
\begin{align*}
&\frac{1}{\mu^2}\left(\frac{\lmin^2}{4\liph^2} + \frac{96(\sigma_1+\sigma_2 \|\tilde x - x^\star\| )^2}{n^2}\right) e^{-\frac{n \cdot \lmin}{16 \liph  (\sigma_1+ \sigma_2\|\tilde x -x^\star\| )}+2d}\\
&\qquad \le  \frac{2\lmin^2}{\mu^2 \liph^2} e^{-\frac{n \cdot\lmin}{64 \liph \sigma_1}}\\
&\qquad \le  \frac{2\lmin^2}{ \mu^2 \liph^2}  e^{-\log \left(\frac{4\lmin^2}{\mu^2 \liph^2 \EE[\|A^{-1} \nabla f(x^\star,z)\|^2]}\right) - \log n}\\
&\qquad = \frac{\EE[\|A^{-1} \nabla f(x^\star,z)\|^2]}{2n}.	
\end{align*}

\paragraph{Case 2: $\|\tilde x -x^\star\|^2 >\frac{\sigma_1^2}{\sigma_2^2}$.} 
In this case, we have
\begin{align}\label{eqn:second_moment}
\frac{1}{\mu^2}\EE\left[\|\tilde   \xi_n\|_*^2 1_{\|\tilde  \xi_n\|_*  > \frac{\lmin}{4\liph}}\Big|\; \tilde x \right] & \le \frac{1}{\mu^2}\EE\left[\|\tilde   \xi_n\|_* ^2 \Big|\; \tilde x\right] \notag\\
&\le \frac{d \sigma_1^2 + d \sigma_2^2\|\tilde x - x^\star\|^2}{n\mu^2} \\
&\le \frac{2d\sigma_2^2 \|\tilde x - x^\star\|^2}{n\mu^2}, \notag
\end{align}
where the estimate~\eqref{eqn:second_moment} follows from Lemma~\ref{lem:second_moment_sub_exponential}.

\smallskip
Combining~\eqref{eqn: non_qua_eq4} and the two cases above, we have 
\begin{align*}
\EE[\|\underline{x} - x^\star\|^2 \mid \tilde x]  \le  \frac{9 \cdot \EE[\|A^{-1} \nabla f(x^\star,z)\|^2]}{n} + \frac{8 (\zeta^2 + d\sigma_2^2)}{n \mu^2}\|\tilde x - x^\star\|^2 ,
\end{align*}
as desired.
\qed

\subsection{Proof of  Proposition~\ref{prop:convergence_SGD}} \label{appendix:stochastic_opt_subroutine_SGD}

 Recall that $g(x,z)=  f(x,z) - \langle \nabla f(\tilde x, z) - \widehat \nabla f(\tilde x), x \rangle$ and $G(x) = F(x) - \langle \nabla F(\tilde x) -  \widehat \nabla f(\tilde x), x\rangle $.  Let $\delta_t = \nabla g(x_t,z_t) - \nabla G(x_t)$ denote the gradient noise at time $t$. Recall the definition of the $\sigma$-algebras $\{\mathcal{F}_t \}_{t \geq 0}$. By Assumption~\ref{assum:vr_assum}, we have
\begin{equation}\label{eqn:delta_t_var}
	\begin{aligned}
		\EE[\|\delta_t\|_*^2 \mid \cF_t] &= \EE[\| (\nabla f(x_t,z_t) - \nabla F(x_t)) - (\nabla f(\tilde x,z_t) - \nabla F(\tilde x) )\|_*^2 \mid \cF_t]\\
		&\le \zeta^2 \|x_t -\tilde x\|^2.
	\end{aligned}
\end{equation}
Note that $G$ is $L$-smooth. By convexity, the result~\cite[Lemma 2]{ghadimi2012optimal}, and the assumption that $\eta \le \frac{1}{2L}$, we have
\begin{align*}
G(x_{t+1}) &\le G(x_t) + \dotp{\nabla G(x_t), x_{t+1} - x_t} + \frac{L}{2} \|x_{t+1} - x_t\|^2\\
& = G(x_t) + \dotp{\nabla g(x_t,z_t), x_{t+1} - x_t} + \frac{1}{2\eta} \|x_t - x_{t+1}\|^2 -  \frac{1}{2\eta} \|x_t - x_{t+1}\|^2 + \frac{L}{2} \|x_{t+1} - x_t\|^2 - \dotp{\delta_t, x_{t+1} - x_t}\\
&\le G(x_t) +\dotp{\nabla g(x_t, z_t),\underline{x} - x_t} + \frac{1}{2\eta} \|x_t - \underline{x}\|^2  -\frac{1}{2\eta} \|x_{t+1} - \underline{x}\|^2 -  \frac{1}{4\eta} \|x_t - x_{t+1}\|^2   - \dotp{\delta_t, x_{t+1} - x_t}\\
&\le G(\underline{x}) +  \frac{1}{2\eta} \|x_t - \underline{x}\|^2  -\frac{1}{2\eta} \|x_{t+1} - \underline{x}\|^2 -  \frac{1}{4\eta} \|x_t - x_{t+1}\|^2   + \dotp{\delta_t, \underline{x} - x_{t+1}} \\
&\le G(\underline{x}) +  \frac{1}{2\eta} \|x_t - \underline{x}\|^2  -\frac{1}{2\eta} \|x_{t+1} - \underline{x}\|^2 -  \frac{1}{4\eta} \|x_t - x_{t+1}\|^2   + \dotp{\delta_t, \underline{x} - x_t} + \|\delta_t\|_* \|x_t - x_{t+1}\|. 
\end{align*}
Note that by Young's inequality,
\begin{align*}
\|\delta_t\|_* \|x_t - x_{t+1}\| \le \eta \|\delta_t\|_*^2 + \frac{ 1}{4 \eta} \|x_t - x_{t+1}\|^2. 
\end{align*}
Combining the two displays above with~\eqref{eqn:delta_t_var}, we have
\begin{align*}
\frac{\mu}{2} \EE[ \|x_{t+1} - \underline{x}\|^2 \mid \cF_t]  & \le \EE[G(x_{t+1}) - G(\underline{x}) \mid \cF_t]\\
&\le \frac{1}{2\eta} \|x_t - \underline{x}\|^2  -\frac{1}{2\eta} \EE[\|x_{t+1} - \underline{x}\|^2 \mid \cF_t]+ \eta \zeta^2 \|x_t - \tilde x\|^2\\
&\le\frac{1}{2\eta} \|x_t - \underline{x}\|^2  -\frac{1}{2\eta} \EE[\|x_{t+1} - \underline{x}\|^2 \mid \cF_t]+  2\eta \zeta^2 \|x_t - \underline{x} \|^2+ 2\eta \zeta^2 \|\tilde x - \underline{x} \|^2.
\end{align*}
Taking a further expectation and summing the inequality above from $t= 0$ to $T$, we have
\begin{align*}
\frac{\mu}{2}\sum_{t=1}^{T+1} \EE[ \|x_t - \underline{x}\|^2 \mid \cF_0] &\le  \frac{1}{2\eta} \|\tilde x - \underline{x}\|^2 + \frac{\mu}{4} \sum_{t=0}^{T} \EE[\|x_t -\underline{x}\|^2 \mid \cF_0] + 2(T +1) \eta\zeta^2 \|\tilde x - \underline{x}\|^2 \\
&\le \frac{1}{2\eta} \|\tilde x - \underline{x}\|^2 + \frac{\mu}{4} \| \tilde x - \underline{x}\|^2 + \frac{\mu}{4} \sum_{t=1}^{T+1} \EE[\|x_t -\underline{x}\|^2 \mid \cF_0] + \frac{(T+1) \mu \|\tilde x - \underline{x}\|^2}{128},
\end{align*}
where both inequalities follow from the assumption that $\eta \le \frac{\mu  }{256 \zeta^2}$. Rearranging, we obtain
\begin{align*}
\frac{1}{T+1} \sum_{t=1}^{T+1} \EE[\|x_t - \underline{x}\|^2 \mid \cF_0] &\le \left(\frac{2}{(T+1)\eta \mu} + \frac{1}{T+1} + \frac{1}{32}\right)\|\tilde x - \underline{x}\|^2 \\
&\le \frac{1}{16} \|\tilde x -\underline{x}\|^2.
\end{align*}
The result then follows from our choice of output $\frac{1}{T+1} \sum_{t=1}^{T+1} x_t$ and Jensen's inequality, since $\| \cdot \|^2$ is a convex function and so $\EE[\|\frac{1}{T+1} \sum_{t=1}^{T+1} x_t - \underline{x}\|^2 \mid \cF_0] \leq \frac{1}{T+1} \sum_{t=1}^{T+1} \EE[\|x_t - \underline{x}\|^2 \mid \cF_0]$.
\qed

\subsection{Proof of Proposition~\ref{prop:convergence_ASGD}} \label{appendix:stochastic_opt_subroutine_ASGD}
 Recall that $g(x,z)= f(x,z) - \langle \nabla f(\tilde x, z) - \widehat \nabla f(\tilde x), x \rangle$ and $G(x) = \EE_{z\sim P}[g(x,z)] = F(x) - \langle \nabla F(\tilde x) - \widehat \nabla f(\tilde x), x \rangle$.   Let $\tilde \delta_t = \frac{1}{m_t} \sum_{i=1}^{m_t} \nabla g(r_t, z_i^{(t)}) - \nabla G(r_t)$ denote the gradient noise at iteration $t$. Recall the definition of the $\sigma$-algebras $\{\widetilde{\mathcal{F}}_t \}_{t \geq 1}$. By Assumption~\ref{assum:vr_assum}, we have
\begin{equation}\label{eqn:delta_t_var_duplicate}
	\begin{aligned}
		\EE[\|\tilde \delta_t\|_*^2 \mid  \tilde \cF_t] &= \EE \left[\left\| \frac{1}{m_t}\sum_{i=1}^{m_t}\left[(\nabla f(r_t,z_i^{(t)}) - \nabla F(r_t)) - (\nabla f(\tilde x,z_i^{(t)}) - \nabla F(\tilde x)) \right]\right\|_*^2 \;\middle | \; \tilde \cF_t\right]\\
		&\le \frac{\zeta^2}{m_t} \|r_t -\tilde x\|^2.
	\end{aligned}
\end{equation}
Define $\Gamma_t = \frac{2}{t(t+1)}$. It is straightforward to verify that by our choice of $\alpha_t$ and $\gamma_t$, the following relations hold:
\begin{align}\label{eqn:relations}
\tilde \mu + \gamma_t - L \alpha_t^2  \ge \tilde \mu + \frac{\gamma_t}{2} , \quad \frac{\gamma_t}{\Gamma_t} \equiv 4L, \quad \text{and} \quad \Gamma_t = \begin{cases}
	1 & t =1 \\
	(1-\alpha_t) \Gamma_{t-1} & t \ge 2.
\end{cases}
\end{align}
Define $l_G(z,x) := G(z) + \dotp{\nabla G(z), x-z} + \frac{\tilde \mu}{2} \|z-x\|^2$ and $\Delta_t(x) := \alpha_t \dotp{\tilde \delta_t, x- x_{t-1}^+} + \frac{\alpha_t^2 \|\tilde \delta_t\|_*^2}{\tilde \mu + \gamma_t - L\alpha_t^2}$, where 
\begin{align*}
 x_{t-1}^+ = \frac{\alpha_t \tilde \mu}{\tilde \mu + \gamma _t} r_t + \frac{(1-\alpha_t)\tilde \mu + \gamma_t}{\tilde \mu + \gamma_t} x_{t-1}.
\end{align*}
By ~\cite[Proposition 5]{ghadimi2012optimal}, for any $x \in \RR^d$ and $t \ge 1$,  
\begin{align*}
G(y_t) + \frac{\tmu}{2} \|x_t - x\|^2 \le \Gamma_t \sum_{\tau = 1}^{t} \frac{\alpha_\tau}{\Gamma_\tau} l_G(r_{\tau}, x) + \Gamma_t \sum_{\tau=1}^{t} \frac{\gamma_\tau}{\Gamma_\tau} \left(\frac{\|x_{\tau-1}-  x\|^2}{2} - \frac{\|x_{\tau} - x\|^2}{2}\right) + \Gamma_t \sum_{\tau = 1}^t \frac{\Delta_{\tau} (x)}{\Gamma_\tau}.
\end{align*}
By the choice $\tilde \mu = \frac{\mu}{2}$ and since $G$ is $\mu$-strongly convex, we have 
\begin{align*}
l_G(r_{\tau},x) \le G(x) - \frac{\mu}{4}\|r_{\tau} - x\|^2. 
\end{align*}
Note also that by our choice of $\alpha_t$ and $\gamma_t$, we have $\frac{\gamma_\tau}{\Gamma_\tau}  = 4L$ and $\Gamma_t\sum_{\tau = 1}^{t} \frac{\alpha_\tau}{\Gamma_\tau} = 1$.
Therefore,  for any $x \in \RR^d$, 
\begin{align}\label{eqn:bound_from_lan}
G(y_t) \le G(x) + \frac{4L}{ t(t+1)}\|\tilde x - x\|^2   + \Gamma_t \sum_{\tau = 1}^t \left(\frac{\Delta_\tau(x)}{\Gamma_\tau} - \frac{\alpha_\tau \mu }{4\Gamma_\tau}  \|r_\tau - x\|^2 \right)
\end{align}
 Then, for any $x \in \RR^d$, we have
\begin{align}
\EE\left[\Delta_{\tau}(x) - \frac{\alpha_\tau \mu}{4} \|r_{\tau} - x\|^2 \mid \tilde \cF_\tau \right] &= \frac{\alpha_\tau^2}{\tilde \mu+\gamma_\tau - L\alpha_\tau^2} \EE[\|\tilde \delta_\tau\|_*^2 \mid \tilde \cF_\tau] - \frac{\alpha_\tau \mu}{4} \|r_\tau - x\|^2 \label{eqn:key_bound_eq1}\\
&\le \frac{2\alpha_\tau^2}{\gamma_\tau}  \frac{\zeta^2}{m_\tau} \|r_\tau - \tilde x\|^2  - \frac{\alpha_\tau \mu}{4}  \|r_\tau - x\|^2 \label{eqn:key_bound_eq2} \\
&\le \left(\frac{4\alpha_\tau^2}{\gamma_\tau}  \frac{\zeta^2}{m_\tau} - \frac{\alpha_\tau \mu}{4}  \right)\|r_\tau -  x\|^2 + \frac{4\alpha_\tau^2}{\gamma_\tau}  \frac{\zeta^2}{m_\tau}  \|x- \tilde x\|^2,  \label{eqn:key_bound_eq3}\\
&\le \frac{\mu}{128 (\tau +1) }  \|x- \tilde x\|^2  \label{eqn:key_bound_eq4},
\end{align}
where the equality~\eqref{eqn:key_bound_eq1} follows from the fact that $\alpha_\tau \dotp{ \tilde \delta_\tau, x - x_{\tau-1}^+}$ is a martingale difference sequence with respect to $\widetilde{\cF}_\tau$, the estimate~\eqref{eqn:key_bound_eq2}  follows from~\eqref{eqn:delta_t_var_duplicate} and~\eqref{eqn:relations}, the estimate~\eqref{eqn:key_bound_eq3} follows from Young's inequality, and the final bound~\eqref{eqn:key_bound_eq4} follows from our choice of parameters.
Substituting~\eqref{eqn:key_bound_eq4} into~\eqref{eqn:bound_from_lan}, applying the law of total expectation, and taking $x = \underline{x}$, we have for any $T \ge 1$ that
\begin{align*}
\EE[G(y_T) - \inf G \mid \tilde \cF_1] &\le \frac{4L}{ T(T+1)} \|\tilde  x - \underline{x}\|^2 + \Gamma_T \sum_{\tau =1}^{T} \frac{\mu}{128 (\tau +1) \Gamma_\tau} \|\tilde x - \underline{x}\|^2 \\
& \le  \frac{4L}{ T(T+1)} \|\tilde  x - \underline{x}\|^2 +  \frac{\mu}{64} \|\tilde x - \underline{x}\|^2 .
\end{align*}
When $T \ge 16\sqrt{\frac{L}{ \mu}}$, we have 
$$\EE[G(y_T) - \inf G \mid \tilde \cF_1]  \le \frac{\mu}{32} \|\tilde x - \underline{x}\|^2.$$

Note also that we have $G(y_T) - \inf G \ge \frac{\mu}{2} \|y_T - \underline{x}\|^2$ by strong convexity, and combining these yields
\begin{align*}
\EE[ \|y_T - \underline{x} \|^2 \mid \tilde \cF_1 ] \le \frac{1}{16} \| \tilde x - \underline{x}\|^2,
\end{align*}
as desired.
\qed

\section{Proofs deferred from Section~\ref{sec:overall_complexity}}

In this section, we prove Corollaries~\ref{cor:quadratic_complexity} and~\ref{cor:complexity_quad_ASGD}.
In both proofs, we ignore rounding issues for cleanliness -- the proof still holds when parameters that are supposed to be integers are given by the smallest integer greater than or equal to the given expression.

\subsection{Proof of Corollary~\ref{cor:quadratic_complexity}} \label{appendix:quadratic_complexity} 
Set the total number of epochs $K = \log_2\paren{\frac{n \|x_0-x^\star\|^2}{ \EE[\|A^{-1} \nabla f(x^\star, z)\|^2]}}$ and select algorithm parameters 
\begin{align*}
\eta =  \min\left\{\frac{1}{2L}, \frac{\mu }{256\zeta^2}\right\}, \quad T = \max\left\{\frac{256}{\eta\mu}, 64\right\} ,\quad N = \frac{n}{6}, \quad \text{and} \quad N_k = \max\left\{\frac{32\zeta^2}{\mu^2}, (\frac{3}{4})^{K+1-k} \cdot N \right\}.
\end{align*}
Evidently, the conditions of Theorem~\ref{thm:vr_quadratic} are satisfied. Consequently, we have
\begin{align*}
\EE[\|\widehat x_K - x^\star\|^2] &\le \frac{1}{2^K} \|\hat x_0 - x^\star\|^2 + \frac{20\cdot  \EE[\|A^{-1} \nabla f(x^\star, z)\|^2]}{N} \\
&= \frac{121\cdot  \EE[\|A^{-1} \nabla f(x^\star, z)\|^2]}{n},
\end{align*}
where the second line follows from our choice of $K$ and $N$. In addition, performing some algebra on our parameter choices yields
\begin{align*}
(T+1)K  \lesssim_{\log} \frac{ L}{\mu} + \frac{ \zeta^2}{\mu^2}  +1.
\end{align*}
In addition, 
\begin{align*}
\sum_{k=1}^{K} N_k \le  K\cdot \frac{32\zeta^2}{\mu^2} + N \cdot \sum_{k=1}^{K} \left(\frac{3}{4}\right)^{K-k+1} \le K\cdot \frac{32\zeta^2}{\mu^2}  + \frac{n}{2}.
\end{align*}
Therefore, the total number of samples used $(T+1)K + \sum_{k=1}^K N_k$ can be bounded by $n$ when 
$$n \gtrsim_{\log} \frac{L}{\mu} + \frac{\zeta^2}{\mu^2} + 1,$$
as claimed.
\qed

\subsection{Proof of Corollary~\ref{cor:complexity_quad_ASGD}}\label{appendix:complexity_quad_ASGD}
Suppose that we select the algorithm parameters as prescribed by Theorem~\ref{thm:vr_quadratic_ac}.  Set 	the total number of epochs $K = \log_2\paren{\frac{n \|x_0-x^\star\|^2}{ \EE[\|A^{-1} \nabla f(x^\star, z)\|^2]}}$ and choose  $N = \frac{n}{6}$ and $N_k = \max \left\{\frac{32\zeta^2}{\mu^2}, \left(\frac{3}{4}\right)^{K+1-k} \cdot N \right\}.$ 
By Theorem~\ref{thm:vr_quadratic_ac}, we have
\begin{align*}
	\EE[\|\widehat x_K - x^\star\|^2] &\le \frac{1}{2^K} \|x_0 - x^\star\|^2 + \frac{20\cdot  \EE[\|A^{-1} \nabla f(x^\star, z)\|^2]}{N} \\
	&= \frac{121\cdot  \EE[\|A^{-1} \nabla f(x^\star, z)\|^2]}{n},
\end{align*}
where the second line follows from our choice of $K$ and $N$.  Note that in each epoch $k$, ASGD requires a sample size of
\begin{align*}
\sum_{t=1}^T m_t  \le  \sum_{t=1}^{T+1}  \left(\frac{256\zeta^2 t}{\mu L} + 1 \right) \le \frac{256 \zeta^2  T^2}{\mu L} + 16 \sqrt{\frac{L}{ \mu}}.
\end{align*}
Therefore, 
\begin{align*}
K \sum_{t=1}^T m_t \lesssim_{\log} \sqrt{\frac{L}{\mu}} + \frac{\zeta^2}{\mu^2}.
\end{align*}
In addition, 
\begin{align*}
\sum_{k=1}^{K} N_k \le  K\cdot \frac{32\zeta^2}{\mu^2} + N \cdot \sum_{k=1}^{K} \left(\frac{3}{4}\right)^{K-k+1} \le K\cdot \frac{32\zeta^2}{\mu^2}  + \frac{n}{2}.
\end{align*}
Therefore, the total sample size $	K \sum_{t=1}^T m_t + \sum_{k=1}^K N_k$ can be bounded by $n$ when 
$$n \gtrsim_{\log} \sqrt{\frac{L}{\mu}} + \frac{\zeta^2}{\mu^2} + 1,$$
as claimed.
\qed

\section{Proofs deferred from Section~\ref{sec:lower_bound}}\label{sec:lower_bound_proof}
This section is organized as follows. We begin by presenting the general Bayesian Cram\'{e}r-Rao lower bound, and then apply this framework to the setting of stochastic optimization to derive the lower bound in Theorem~\ref{thm: mainlowerboundtheorem_intro}. We conclude with an application of Fano's method to prove the lower bound in Proposition~\ref{prop:necessity_liph_intro}.

\subsection{Bayesian Cram\'{e}r-Rao lower bounds for a general functional}\label{sec:lower_bound_proof_prep}

We begin by stating the following general version of the Bayesian Cram\'{e}r-Rao lower bound.

\begin{theorem}[Theorem 1 in  \cite{gill1995applications}]\label{thm: bayesianCR}
Let $\Theta \subset \RR^d$ denote a general parameter space, and let $\rho$ be a prior distribution with bounded support contained within $\Theta$. Let $\cT: \supp(\rho) \rightarrow \RR^p$ be a $C^1$-smooth map. Suppose the samples $ \{z_i\}_{i=1}^n$ are i.i.d. drawn from a distribution $P_\lambda$ parameterized by $\lambda \in \Theta$. Then, for any estimator $\widehat \cT_n$ based on the samples $\{z_i\}_{i=1}^{n}$ and any smooth matrix-valued function $C:\RR^d \rightarrow \RR^{p\times d}$, we have
\begin{align*}
	&\quad \EE_{\lambda \sim \rho} \EE_{\{z_i\}_{i=1}^{n} \sim P_\lambda^n} \|\widehat \cT_n(\{z_i\}_{i=1}^{n}) - \cT(\lambda)\|_2^2\\
	&\ge   \frac{\left(\int \trace{C(\lambda)  \nabla \cT(\lambda)} \rho(\lambda) d\lambda\right)^2}{n \int \trace{C(\lambda) I(\lambda)C(\lambda)^\top} \rho(\lambda)d\lambda + \int \|\nabla C(\lambda) + C(\lambda) \nabla \log \rho(\lambda) \|_2^2 \rho(\lambda) d\lambda}.
\end{align*}
\end{theorem}
Now suppose that we have a map $\cT:\RR^d \rightarrow \RR^d$ and a fixed symmetric positive definite matrix $\Sigma$.  Assume that the sample distribution $P_\lambda=N(\lambda, \Sigma)$, where $\lambda$ is unknown and we want to estimate $\cT(\lambda)$ using samples $\{z_i\}_{i=1}^{n}$. We will choose a specific function $C:\RR^d \rightarrow \RR^{p\times d}$ and prior distribution $\rho$ such that the Bayesian lower bound in Theorem~\ref{thm: bayesianCR} has a simpler form.

We consider the following one-dimensional density function borrowed from Section 2.7 of~\cite{tsybakov2008nonparametric}. Let $\mu(t): = \cos^2\left(\frac{\pi t}{2}\right) \cdot 1_{[-1,1]}$, and denote by $\mu^{\otimes d}$ the $d$-fold product measure of $\mu$. Let $Z$ denote a random vector drawn from $\mu^{\otimes d}$. Let $Q$ be any fixed orthogonal matrix, and we assign a prior distribution to $\lambda$ by letting
\begin{align}\label{eqn:prior}
\lambda =  \frac{1}{\sqrt{n}} \Sigma^{1/2} Q Z.
\end{align} 
We denote the density function of $\lambda$ by $\rho$. Our prior differs from the work~\cite{mou2023optimal} in that we have an extra orthogonal matrix $Q$, and this flexibility allows us to prove a tighter lower bound than existing results. 

We proceed with a lower bound for a general functional $\cT$, placing the following regularity condition.
\begin{assumption}\label{assum: regularityofT}
The map $\cT: \RR^d \rightarrow \RR^d$ is bijective and $C^1$ continuous. We denote the Jacobian of $\cT$ by $\nabla \cT$.  We assume that $\nabla \cT(0)$ is invertible.
\end{assumption}
Below is our main theorem of this subsection.
\begin{theorem}\label{thm: lowerboundtheoremgeneral}
Suppose that the map $\cT:\RR^d \rightarrow \RR^d$ satisfies Assumption~\ref{assum: regularityofT}. Let $\widehat \cX_n$ be the set of estimators based on $n$ samples, i.e., each $\widehat x_n \in \widehat \cX_n$ is a measurable map from $(\RR^d)^{\otimes n}$ to $\RR^d$.  Fix an orthogonal matrix $Q$ and let $\rho$ denote the density of $\lambda$ defined in~\eqref{eqn:prior}. For any $n$  large enough so that 
\begin{align*}
	\EE_{\rho}\left[\norm{\nabla \cT(0)^{-\top} (\nabla \cT(\lambda) - \nabla \cT(0))}_2\right] \le \frac{1}{2},
\end{align*} 
we have
\begin{align}\label{eqn: lowerboundforestglambdageneral}
	\inf_{\widehat x_n \in \widehat \cX_n}\EE_{\lambda \sim \rho} \EE_{z_i \overset{\mathrm{iid}}{\sim} N(\lambda,\Sigma)} \|\widehat x_n (\{z_i\}_{i=1}^{n}) - \cT(\lambda)\|_2^2 \ge  \frac{ \trace{\nabla \cT(0) \Sigma \nabla \cT(0)^{\top}}}{4(\pi^2 +1) n}.
\end{align}
\end{theorem}
\begin{proof}
To apply Theorem~\ref{thm: bayesianCR}, we use the following constant map $C$:
$$
C(\lambda) = \nabla T(0) \cdot I(\lambda)^{-1} = \nabla \cT(0) \Sigma,
$$
where the second equality follows from Lemma~\ref{lem: Fisherinfo:lambda}. 
By Lemma~\ref{lem: tracelb} and our assumption on $n$, we have 
\begin{align}
	\EE_\rho\left[\trace{C(\lambda) \nabla  \cT(\lambda)^\top}\right] &= \EE_\rho[\trace{\nabla \cT(0) \Sigma \nabla \cT(\lambda)^{\top}}] \notag \\
	&\ge  \frac{1}{2} \trace{\nabla \cT(0) \Sigma \nabla \cT(0)^\top} \label{eqn: piece1general}
\end{align}

On the other hand, by Lemma~\ref{lem: Fisherinfo:lambda},
\begin{align}
	\EE_\rho\left[\trace{C(\lambda) I(\lambda) C(\lambda)^\top}\right] =  \trace{\nabla \cT(0) \Sigma \nabla \cT(0)^{\top}}. \label{eqn: piece2general}
\end{align}

Additionally, by Lemma~\ref{lem: nablalog}, we have
\begin{align}
	\EE_\rho \|\nabla C(\lambda) + C(\lambda) \nabla \log(\rho(\lambda))\|_2^2 &= \trace{C(\lambda) \EE_\rho \left[\nabla \log(\rho(\lambda)) \nabla \log(\rho(\lambda))^\top\right] C(\lambda)^\top}\notag \\
	&= n\pi^2\trace{\nabla \cT(0) \Sigma \nabla \cT(0)^{\top}}. \label{eqn: piece3general}
\end{align} 

Applying Theorem~\ref{thm: bayesianCR} with equations \eqref{eqn: piece1general},\eqref{eqn: piece2general},\eqref{eqn: piece3general}, for any $\widehat x_n \in \widehat \cX_n$, we have 
\begin{align}
	\EE_{\lambda \sim \rho} \EE_{z_i \overset{\mathrm{iid}}{\sim} N(\lambda,\Sigma)} \|\widehat x_n (\{z_i\}_{i=1}^{n}) - \cT(\lambda)\|_2^2 \ge  \frac{ \trace{\nabla \cT(0) \Sigma \nabla \cT(0)^{\top}}}{4(\pi^2 +1) n},
\end{align}
as desired. 
\end{proof}
\begin{remark}\label{remark: linearcasetheorem}
We note that if $\cT$ is a linear map, the conclusion of Theorem~\ref{thm: lowerboundtheoremgeneral} holds for any $n \ge 1$ because of Remark~\ref{remark: linearcaselemma}.
\end{remark}
The rest of this subsection consists of supporting lemmas for Theorem~\ref{thm: lowerboundtheoremgeneral}.
\begin{lemma}\label{lem: nablalog}
Let $\rho:\RR^d \rightarrow \RR_+$ denote the density of $\lambda$ defined in~\eqref{eqn:prior}. Then 
$$
\EE[\nabla \log \rho(\lambda) \nabla \log \rho(\lambda)^\top ] = n\pi^2 \Sigma^{-1}.
$$
\end{lemma}
\begin{proof}
By a change of variables, we have 
\begin{align*}
	\rho(\lambda) = n^{d/2} \det(  \Sigma^{-1/2}) \mu^{\otimes d}(\sqrt{n} Q^{T} \Sigma^{-1/2} \lambda).
\end{align*}
Therefore,
\begin{align*}
	\EE[\nabla \log \rho(\lambda) \nabla \log \rho(\lambda)^\top ] &= \int \nabla \log \rho(\lambda) \paren{\nabla \log \rho(\lambda)}^\top  \rho(\lambda) d\lambda \\
	&= \int \nabla_\lambda \log \mu^{\otimes d}(\sqrt{n} Q^{T} \Sigma^{-1/2} \lambda) \paren{ \nabla_\lambda \log \mu^{\otimes d}(\sqrt{n} Q^{T} \Sigma^{-1/2} \lambda)}^{\top}  \rho(\lambda) d\lambda \\
	&= \int \sqrt{n} (Q^{\top} \Sigma^{-1/2})^\top \nabla \log \mu^{\otimes d}(z)\paren{ \nabla \log \mu^{\otimes d}(z)} ^\top  \sqrt{n} Q^{\top} \Sigma^{-1/2}  \mu^{\otimes d}(z) dz \\
	&= n \Sigma^{-1/2} Q \cdot \EE[\nabla \log \mu^{\otimes d}(Z) \paren{\nabla \log \mu^{\otimes d}(Z)}^\top] Q^\top \Sigma^{-1/2}\\
	&= n\pi^2 \Sigma^{-1},
\end{align*}
where the last equality follows from  $ \EE[\nabla \log \mu^{\otimes d}(Z) \paren{\nabla \log \mu^{\otimes d}(Z)}^\top] = \pi^2 I$.
\end{proof}
\begin{lemma}\label{lem: Fisherinfo:lambda}
The Fisher information matrix of the observation model is given by
\begin{align*}
	I(\lambda) = \Sigma^{-1}.
\end{align*}
\end{lemma}
\begin{proof}
Note that $S~\sim N(\lambda,\Sigma)$ and the Fisher information of a Gaussian vector with respect to the mean parameter is its inverse covariance matrix.
\end{proof}
\begin{lemma}\label{lem: tracelb}
Under Assumption~\ref{assum: regularityofT}, when $n$ is large enough so that 
\begin{align}\label{eqn: distg(lambda)} 
	\EE_{\rho}\left[\norm{\nabla \cT(0)^{-\top} (\nabla \cT(\lambda) - \nabla \cT(0))}_2\right] \le \frac{1}{2},
\end{align} 
we have 
$$
\EE_\rho\left[\trace{\nabla \cT(0) \Sigma \nabla \cT(\lambda)^\top}\right]
\ge  \frac{1}{2} \trace{\nabla \cT(0) \Sigma \nabla \cT(0)^\top}.
$$
\end{lemma}
\begin{proof}

Performing some basic linear algebra, we have

\begin{align*}
	& \quad \EE_\rho \left[\trace{\nabla \cT(0) \Sigma \nabla \cT(\lambda)^\top}\right]\\
	&= \EE_\rho\left[\trace{\nabla \cT(0) \Sigma \nabla \cT(0)^\top}\right] +  \EE_\rho\left[\trace{\nabla \cT(0) \Sigma \nabla \cT(0)^\top \nabla \cT(0)^{-\top} (\nabla \cT(\lambda) - \nabla \cT(0))}\right].
\end{align*}
Moreover, 
\begin{align*}
	&\quad\EE_\rho\left[\left|\trace{\nabla \cT(0) \Sigma \nabla \cT(0)^\top \nabla \cT(0)^{-\top} (\nabla \cT(\lambda) - \nabla \cT(0))}\right|\right]\\
	&\le \EE\left[\|\nabla \cT(0) \Sigma \nabla \cT(0)^\top \nabla \cT(0)^{-\top} (\nabla \cT(\lambda) - \nabla \cT(0)) \|_{nuc}\right]\\
	&\le \EE_\rho\left[\|\nabla \cT(0) \Sigma \nabla \cT(0)^\top\|_{nuc} \norm{\nabla \cT(0)^{-\top} (\nabla \cT(\lambda) - \nabla \cT(0))}_2\right]\\
	&\le \frac{1}{2} \trace{\nabla \cT(0) \Sigma \nabla \cT(0)^\top},
\end{align*}
	where the first inequality follows from the fact that $\trace{A}\le \|A\|_{nuc}$ for any square matrix $A$, the second inequality follows from $\norm{AB}_{nuc} \le \|A\|_{nuc} \norm{B}_2$, and the third inequality follows from the positive semi-definiteness of $\nabla \cT(0) \Sigma \nabla \cT(0)^\top $ and equation~\eqref{eqn: distg(lambda)}. The result follows from combining the pieces.
\end{proof}
\begin{remark}\label{remark: linearcaselemma}
	Note that if $\cT$ is a linear map, the conclusion of Lemma~\ref{lem: tracelb} holds for any $n \ge 1$ because $\nabla \cT(\lambda) = \nabla \cT(0)$ for any $\lambda$.
\end{remark}

Next, we apply the main theorem of this section to stochastic optimization problems and prove Theorem~\ref{thm: mainlowerboundtheorem_intro}.

\subsection{Proof of Theorem~\ref{thm: mainlowerboundtheorem_intro}}\label{sec:lower_bound_proof_stochastic}

Let us consider the sample level objective function  $f_0(x, z) = F(x) - \dotp{z,x}$. We denote the distribution $N(\lambda,\Sigma)$ by $P_\lambda$.  Direct calculation shows that $$F_{f_0,P_{\lambda}}(x) = F(x) - \dotp{\lambda, x}.$$ 
Since $F_{f_0,P_{\lambda}}$ is strongly convex, its minimizer exists and is unique for any $\lambda \in \RR^d$. Let us define the map $\cT:\RR^d \rightarrow \RR^d$ that maps $\lambda$ to the minimizer of $F_{f_0,P_\lambda}$.  Strong convexity of $F$  implies that $\cT$ is bijective. For notational simplicity, we define the parameterized gradient map 
\begin{align*}
	G(x, \lambda) = \nabla F_{f_0,P_\lambda}(x) = \nabla F(x) - \lambda. 
\end{align*}                         
A direct calculation shows that 
\begin{align*}
	\nabla_x G(x,\lambda) = \nabla^2 F(x) \qquad \text{and} \quad \nabla_\lambda G(x,\lambda) = -I.
\end{align*}
By the definition of $\cT$, we have $G(\cT(\lambda), \lambda) = 0$ and $\cT(0) = x^\star$. Since $\nabla^2 F(x)$ is positive definite for any $x$, we also have that the map $\cT$ is $C^1$ by the implicit function theorem. Additionally, 
\begin{align}\label{eqn: gradientofg}
	\nabla \cT(\lambda) = - \nabla_x G(\cT(\lambda), \lambda)^{-1} \nabla_\lambda G(\cT(\lambda), \lambda) = \nabla^2 F(\cT(\lambda))^{-1}.
\end{align}        
Let $U \Gamma V^\top$ be the Singular Value Decomposition (SVD) of $\nabla^2 F(x^\star)^{-1} \Sigma^{1/2}$. We suppose that the parameter $\lambda$ takes the form of~\eqref{eqn:prior} with $Q=V$ and let $\rho$ be the density of $\lambda$. For any $\lambda \in \supp(\rho)$, we have
\begin{align}
	\|\nabla^2 F(x^\star)^{-1} \lambda\|_2^2 &= \frac{1}{n} \|\nabla^2 F(x^\star)^{-1} \Sigma^{1/2} V Z\|_2^2 \notag\\
	&= \frac{1}{n} \|\Gamma Z\|_2^2\notag\\
	&\le \frac{1}{n} \sum_{i=1}^{d} \Gamma_{ii}^2 \label{eqn:property_Z}\\
	&= \frac{1}{n} \trace{\nabla^2 F(x^\star)^{-1} \Sigma \nabla^2 F(x^\star)^{-1}},\label{eqn:def_of_n_tight}\\
	&\le \frac{1}{64 \liph^2} \label{eqn:def_of_n}
\end{align}
where the estimate~\eqref{eqn:property_Z} follows since $Z \in [-1,1]^{\otimes d}$ pointwise and the estimate~\eqref{eqn:def_of_n} follows from our assumption on $n$. Thus, for any $\lambda \in \supp(\rho)$, we have 
$$\|\nabla^2 F(x^\star)^{-1} \lambda\|_2 \le \frac{1}{8\liph}.$$
By Lemma~\ref{lem:small_noise_close_solution},  we also have 
\begin{align}\label{eqn:Tlambda_bound}
	\|\cT(\lambda)  - \cT(0)\|_2 \le 2 \|\nabla^2 F(x^\star)^{-1} \lambda\|_2 \le \frac{1}{4 \liph}.
\end{align}
Next, we show that $\cT$ satisfies the conditions of Theorem~\ref{thm: lowerboundtheoremgeneral}.  To this end, denote  $\nabla^2 F(\cT(\lambda)) - \nabla^2 F(x^\star)$ by $\Delta(\lambda)$.  Since the Hessian matrices are symmetric and $\nabla^2 F$ is Lipschitz continuous, we have
\begin{equation}\label{eqn:delta_nablaf}
		\begin{aligned}
			\norm{\Delta(\lambda) \nabla^2 F(x^\star)^{-1}}_2 &= \norm{\nabla^2 F(x^\star)^{-1} (\nabla^2 F(\cT(\lambda)) - \nabla^2 F(x^\star))}_2\\
			&\le \liph \|\cT(\lambda) - \cT(0)\|_2\\
			&\le \frac{1}{4}. 
		\end{aligned}
	\end{equation}
As a result, we have
\begin{align}
		\EE_\rho\left[\norm{\nabla \cT(0)^{-\top}(\nabla \cT(\lambda) - \nabla \cT(0))}_2\right] &= \EE_\rho[\norm{\nabla^2 F(x^\star)(\nabla^2 F(\cT(\lambda))^{-1} - \nabla^2 F (x^\star)^{-1})}_2] \notag \\
		&= \EE_\rho[\norm{\nabla^2 F(x^\star)[ \nabla^2 F(x^\star)^{-1}(I + \Delta(\lambda) \nabla^2 F(x^\star)^{-1})^{-1} - \nabla^2 F(x^\star)^{-1}]}_2] \notag\\
		&= \EE_\rho[\norm{(I + \Delta(\lambda) \nabla^2 F(x^\star)^{-1})^{-1} - I}_2] \notag\\
		&\le 2\EE_\rho[\norm{\Delta(\lambda) \nabla^2 F(x^\star)^{-1}}_2] \label{eqn:key_bound_1} \\
		&\le \frac{1}{2},\label{eqn:key_bound_2}
\end{align}
where the estimate~\eqref{eqn:key_bound_1} follows from~\eqref{eqn:delta_nablaf} and Lemma~\ref{lem: opnormineq}, and the estimate~\eqref{eqn:key_bound_2} follows from~\eqref{eqn:delta_nablaf}. Applying Theorem~\ref{thm: lowerboundtheoremgeneral}, we then have 
\begin{align}\label{eqn: lowerboundforestglambda}
	\inf_{\widehat x_n \in \widehat \cX_n}\EE_{\lambda \sim \rho} \EE_{z_i \overset{\mathrm{iid}}{\sim} N(\lambda,\Sigma)} \|\widehat x_n (\{z_i\}_{i=1}^{n}, f) - \cT(\lambda)\|_2^2 \ge  \frac{ \trace{\nabla^2 F(x^\star)^{-1} \Sigma \nabla^2 F(x^\star)^{-1}}}{4(\pi^2 +1) n}.
\end{align}
Using~\eqref{eqn:def_of_n_tight}, it is then straightforward to verify that $(f_0,P_\lambda) = (f_0,N(\lambda,\Sigma)) \in \nclass(n,F,\Sigma)$ for any $\lambda \in \supp(\rho)$.
Combining this with the fact that $x^\star(F_{f_0,P_\lambda}) =  \cT(\lambda)$, we have 
\begin{align*}
	&\quad \inf_{\widehat x_n \in \widehat \cX_n} \sup_{(f,P)\in \nclass(n,F,\Sigma)} \EE_{z_i\overset{\mathrm{iid}}{\sim} P}[\|\widehat x_n(\{z_i\}_{i=1}^{n}, f) - x^\star(F_{f,P})\|_2^2]\\
	&\ge \inf_{\widehat x_n \in \widehat \cX_n}\EE_{\lambda \sim \rho} \EE_{z_i \overset{\mathrm{iid}}{\sim} N(\lambda,\Sigma)} \|\widehat x_n (\{z_i\}_{i=1}^{n},f_0) - \cT(\lambda)\|_2^2 \\
	& \ge \frac{ \trace{\nabla^2 F(x^\star)^{-1} \Sigma \nabla^2 F(x^\star)^{-1}}}{4(\pi^2 +1) n},
\end{align*}	
as claimed.
\qed

\subsection{Proof of Proposition~\ref{prop:necessity_liph_intro}} \label{sec:necessity_liph_proof}

Our proof will follow Fano's method. We first provide a construction for the packing and state Fano's lower bound. We then use this to prove the proposition to conclude the section.

\subsubsection{Construction and Fano lower bound}
For any $\mu >0$ and $\liph > 0$, we first construct a population objective function satisfying Assumption~\ref{assum: non-quadratic} with parameters $(\mu, 2+\mu, \liph).$
\begin{lemma}\label{lem:function_justify_liph}
	For any $\liph' >0$, define a one-dimensional function $g$ by
	\begin{align*}
		g(t) = \begin{cases}
			\frac{1}{2} t^2 - \frac{\liph'^2}{6} t^4  + \frac{\liph'^4}{30} t^6 & |t| < \frac{1}{\liph'}\\
			\frac{8}{15\liph'}\left(|t| - \frac{1}{\liph'} \right) + \frac{9}{10\liph'^2} & |t| \ge \frac{1}{\liph'}
		\end{cases}
	\end{align*}
	and another function $G \colon \RR^d \rightarrow \RR$ by $G(x) = g(\|x\|_2).$  Then $ F(x) := G(x) + \frac{\mu}{2}\|x\|_2^2$ satisfies Assumption~\ref{assum: non-quadratic} with parameters $(\mu, 2+\mu, \frac{48\liph'}{15}).$ In particular, for any $\liph >0$, we can take $\liph' = \frac{ \liph}{4}$ and obtain a function satisfying Assumption~\ref{assum: non-quadratic} with parameters $(\mu, 2+\mu, \liph)$.
\end{lemma}
\begin{proof}
	It is straightforward to verify that $g$ is a convex function on $\RR$ and increasing on $[0,\infty)$. Therefore, $G$ is a convex function and $ F$ is $\mu$-strongly convex. Moreover, by construction, $g$ is $C^3$-smooth at the point $\frac{1}{\liph'}$, and as a result $G$ is $C^3$-smooth on $\mathbb{R}^d$. The Hessian of \(G\) is given by
	\begin{align*}
		\nabla^2 G(x)= u(\|x\|_2) I + v(\|x\|_2) xx^\top,
	\end{align*}		
	with
	\begin{align*}
		u(t)= \begin{cases}
			1-\frac{2\liph'^2}{3} t^2 + \frac{\liph'^4}{5}t^4 &\text{if } |t|<\frac{1}{\liph'}\\
			\frac{8}{15\liph' |t|} & 	\text{if } |t|\ge\frac{1}{\liph'}
		\end{cases}
	\end{align*}
	and 
	\begin{align*}
		v(t)= \begin{cases}
			-\frac{4}{3} \liph'^2 + \frac{4}{5} \liph'^4 t^2 &\text{if } |t|<\frac{1}{\liph'}\\
			- \frac{8}{15 \liph' t^3}  & 	\text{if }|t|\ge\frac{1}{\liph'}
		\end{cases}.
	\end{align*}
	By definition of $u$ and $v$, we have for any $x \in \RR^d$ that
	\begin{align*}
		\norm{\nabla^2 G(x)}_2 &\le |u(\|x\|_2)| + |v(\|x\|_2)| \|x\|_2^2 \le 2.
	\end{align*}
	By the mean value theorem, $\nabla G$ is $2$-Lipschitz. Since $\nabla  F(x) = \nabla G(x) + \mu x$, $\nabla F$ is $(2+ \mu)$-Lipschitz.
	
	Finally, we show that $\nabla^2  F$ is $\frac{48\liph'}{15}$-Lipschitz continuous. To this end, we denote the map $x \mapsto v(\|x\|_2) xx^\top$ by $h$. The Fr\'echet derivative of $h$ in the direction \(z\in\mathbb{R}^d\) is given by
	\begin{align*}
		Dh(x)[z] = w(\|x\|_2) \langle x, z \rangle xx^\top + v(\|x\|_2)\left(zx^\top + xz^\top\right),
	\end{align*}
	where 
	\begin{align*}
		w(t) = \begin{cases}
			\frac{8\liph'^4 }{5}  & \text{if } |t| < \frac{1}{\liph'}\\
			\frac{8}{5 \liph' t^5}  & \text{if } |t| \ge \frac{1}{\liph'}. 
		\end{cases}
	\end{align*}
	Using properties of the operator norm, we obtain the bound
	\begin{align*}
		\|Dh(x)[z]\|_2 & \le |w(\|x\|)| |\langle x, z \rangle| \|xx^\top\|_2 + |v(\|x\|)|\left(\|zx^\top\|_2+\|xz^\top\|_2\right)\\
		&\le \left(|w(\|x\|_2)| \|x\|_2^3 + 2|v(\|x\|_2)| \|x\|_2  \right)\|z\|_2.
	\end{align*}
	By definition of $w$ and $v$, it is straightforward to verify that for any $x \in \RR^d$,
	\begin{align*}
		|w(\|x\|_2)| \|x\|_2^3 + 2|v(\|x\|_2)| \|x\|_2 \le \frac{8 \liph'}{3}.
	\end{align*}
	Therefore, by mean value theorem, $h$ is $\frac{8 \liph'}{3}$-Lipschitz with respect to the operator norm. Additionally, one can verify that $|u'(t)| \le \frac{8 \liph'}{15}$  for all $t \in \RR$, so $x \mapsto u(\|x\|_2) I$ is $\frac{8 \liph'}{15}$ Lipschitz with respect to the operator norm. Combining the pieces, we see that $\nabla^2 G$ is $\frac{48\liph'}{15}$-Lipschitz with respect to the operator norm, and so is $\nabla^2  F$. The result follows.
\end{proof}
For any $0 < \tilde \mu  \le 1$ and $\tilde L_H >0$, we let $\tilde F \colon \RR^d \rightarrow \RR$ be the function from Lemma~\ref{lem:function_justify_liph} satisfying Assumption~\ref{assum: non-quadratic} with parameters $(\tilde\mu, 2+\tilde\mu, \tilde L_H)$. Specifically, 
\begin{align}\label{eqn:function_justify_liph}
	\tilde F(x) = \begin{cases}
		\frac{1 + \tilde \mu}{2} \|x\|_2^2 - \frac{\tilde L_H^2}{6\cdot 4^2} \|x\|_2^4 + \frac{\tilde L_H^4}{30 \cdot 4^4} \|x\|_2^6 & \|x\|_2< \frac{4}{\tilde L_H}\\
		\frac{\tilde \mu}{2} \|x\|_2^2 + \frac{8 \cdot 4}{15\tilde L_H} \left(\|x\|_2 - \frac{4}{\tilde L_H}\right) + \frac{9\cdot 4^2}{10 \tilde L_H^2} & \|x\|_2 \ge \frac{4}{\tilde L_H}.
	\end{cases}
\end{align}
Our proof strategy is to derive a lower bound for minimizing the function $\tilde F$; we then establish Proposition~\ref{prop:necessity_liph_intro} by appropriately choosing $\tilde \mu$ and $\tilde L_H$ and rescaling $\tilde F$ accordingly. To this end, we define the sample objective function as follows:
\begin{align}\label{eqn:sample_obj}
	\tilde f_0(x,z) = \tilde F(x) - \dotp{z, x},
\end{align}
Let the sample distribution be $\tilde P_\theta = N(\theta, I)$. Direct calculation shows that $F_{\tilde f_0, \tilde P_\theta} (x) = \tilde F(x) - \dotp{\theta, x}$. The next proposition applies the Fano lower bound to this observation model.
\begin{proposition}[Fano lower bound]\label{prop:Fano}                                                                                                                                                                                         
	For $n \ge 1$, let $\{z_i\}_{i=1}^{n}$ denote i.i.d. samples drawn from the distribution $\tilde P_\theta$. Let $\widehat x_n$  denote a measurable function  of $\{z_i\}_{i=1}^{n}$.  Suppose that $0< \tilde \mu \le 1$. There exists a finite set $\tilde \Theta$, with each $\theta \in \tilde \Theta$ satisfying  $\|\theta\|_2 \le \frac{72}{\tilde L_H}$, such that the minimax risk
	\begin{align*}
		\inf_{\widehat x_n} \sup_{\theta \in \tilde \Theta} \EE_{z_i\overset{\mathrm{iid}}{\sim} \tilde P_\theta}[\|\widehat x_n(\{z_i\}_{i=1}^{n}) - x^\star(F_{\tilde f_0, \tilde P_\theta})\|_2^2] \ge \frac{12^2}{ \tilde L_H^2 \tilde \mu^2} \left(1- \frac{1}{d} - \frac{2\cdot 72^2 \cdot n}{\log 2 \cdot d\tilde L_H^2}\right).
	\end{align*} 
	In particular, if $d > 1$ and $n \le   \frac{\log 2 \cdot d\tilde L_H^2 }{2\cdot 144^2}$, then
	$$
	\inf_{\widehat x_n} \sup_{\theta \in \tilde \Theta} \EE_{z_i\overset{\mathrm{iid}}{\sim} \tilde P_\theta}[\|\widehat x_n(\{z_i\}_{i=1}^{n}) - x^\star(F_{\tilde f_0, \tilde P_\theta})\|_2^2] \ge \frac{36}{ \tilde L_H^2 \tilde\mu^2}.
	$$
\end{proposition}
\begin{proof}
	Set $r = \frac{72}{\tilde{L}_H}$. Let $\{\theta_1,\theta_2,\ldots, \theta_M\}$ be a $r/3$-packing of $\overline B_r(0)$. Standard results (e.g.~\cite[Lemma 5.7]{wainwright2019high}) imply that we can find such a packing with $M \ge 3^d$.  By the definition of packing, at most one point of $\{\theta_1,\theta_2,\ldots, \theta_M\}$ can be in $\overline B_{r/6}(0)$. Therefore, there exists an $r/3$-packing of the annulus 
	$\left\{\theta \colon r/6 <  \|\theta\|_2 \le  r \right\}$ with number of elements at least $3^d -1$.
	Let $ \tilde \Theta := \{\theta_1, \theta_2,\ldots, \theta_M\}$ be such a packing and note that $M \ge 3^d - 1 \ge 2^d$. 
	By Lemma~\ref{lem:far_noise_far_sol}, we have  
	\begin{align*}
		\|x^\star(F_{\tilde f_0, \tilde P_{\theta_i}}) - x^\star(F_{\tilde f_0, \tilde P_{\theta_j}})\|_2 \ge \frac{12}{\tilde L_H \tilde  \mu} \qquad \text{ for any $1\le i <j \le M$.}
	\end{align*}
	On the other hand, since the distribution is standard Gaussian, we have
	\begin{align*}
		\mathrm{D_{KL}}(\tilde P_{\theta_i}^n || \tilde P_{\theta_j}^n) &= \frac{n}{2} \|\theta_i - \theta_j\|_2^2 \\
		&\le n(\|\theta_i\|_2^2 +  \|\theta_j\|_2^2)\\
		&\le \frac{2\cdot 72^2 \cdot n}{\tilde L_H^2}.
	\end{align*}
	By Fano's lower bound~\cite[Proposition 15.12]{wainwright2019high} and ~\cite[Equation 15.34]{wainwright2019high}, we have 
	\begin{align*}
		\inf_{\widehat x_n} \sup_{\theta \in \tilde \Theta} \EE_{z_i\overset{\mathrm{iid}}{\sim} \tilde P_\theta}[\|\widehat x_n(\{z_i\}_{i=1}^{n}) - x^\star(F_{\tilde f_0, \tilde P_\theta})\|_2^2] \ge \frac{12^2}{ \tilde L_H^2 \tilde  \mu^2} \left(1- \frac{1}{d} - \frac{2\cdot 72^2 \cdot n}{\log 2 \cdot d\tilde L_H^2} \right),
	\end{align*}
	as claimed. The remaining results then follow from a straightforward calculation.
\end{proof}

\begin{lemma}\label{lem:far_noise_far_sol}
	Let $\{\theta_i\}_{i=1}^{M}$ be a set of points such that $\frac{12}{\tilde L_H} \le \|\theta_i\|_2 \le \frac{72}{\tilde L_H} $ and 
	\begin{align}\label{eqn:delta_separation}
		\|\theta_i - \theta_j\|_2\ge \frac{24}{\tilde L_H}, \qquad \text{for all } 1 \le i < j \le M.
	\end{align}
	Suppose that $0< \tilde \mu \le 1$.  Then we have $\|x^\star(F_{\tilde f_0, \tilde P_{\theta_i}}) \|_2 \le \frac{72}{\tilde L_H \tilde \mu}$  for any $1 \le i \le M$ and 
	\begin{align*}
		\| x^\star(F_{\tilde f_0, \tilde P_{\theta_i}}) - x^\star(F_{\tilde f_0, \tilde P_{\theta_j}})\|_2 \ge \frac{1}{2 \tilde \mu } \|\theta_i - \theta_j\|_2 \ge \frac{12}{\tilde L_H \tilde \mu}, \qquad \text{for all } 1 \le i < j \le M.
	\end{align*}
\end{lemma}
\begin{proof}
	Simple calculation shows that $F_{\tilde f_0, \tilde P_\theta}(x) =\tilde  F(x) - \dotp{\theta, x}$. For notational simplicity, let $F_i : = F_{\tilde f_0, \tilde P_{\theta_i}}$.
	On the one hand, by strong convexity and the fact that  zero is the minimizer of $\tilde F$, we have
	\begin{align*}
		\tilde \mu \|x^\star(F_i)\|_2  \le  \| \nabla \tilde F(x^\star(F_i))\|_2 =	\|\theta_i\|_2, 
	\end{align*}
	so we have $\|x^\star(F_i)\|_2 \le  \frac{72}{\tilde L_H \tilde \mu}$ for any $1\le i \le M$. 
	On the other hand, since $\nabla \tilde F$ is $3$-Lipschitz, 
	\begin{align*}
		\|\theta_i\|_2 = \| \nabla \tilde F(x^\star(F_i))\|_2 \le 3\|x^\star(F_i)\|_2.
	\end{align*}
	Since $\|\theta_i\|_2 \ge \frac{12}{\tilde L_H}$, we then have $\|x^\star(F_i)\|_2 \ge \frac{4}{\tilde L_H}$ for any $ 1 \le i \le M$. By the definition of $\tilde F$~\eqref{eqn:function_justify_liph} and $x^\star(F_i)$, for any $1\le i \le M$, we have
	\begin{align*}
		\theta_i  = \tilde\mu x^\star(F_i) + \frac{32}{15\tilde L_H} \frac{x^\star(F_i)}{\|x^\star(F_i)\|_2}.
	\end{align*}
	Applying the triangle inequality and~\eqref{eqn:delta_separation}, for any $i \neq j$, we have
	\begin{align*}
		\|x^\star(F_i) - x^\star(F_j) \|_2 &\ge \frac{1}{\tilde \mu} \left(\|\theta_i - \theta_j\|_2 - \frac{64}{15 \tilde L_H}\right)\\
		&\ge \frac{1}{2\tilde \mu} \|\theta_i - \theta_j\|_2\\
		&\ge \frac{12}{\tilde L_H \tilde\mu},
	\end{align*}
	as desired.
\end{proof}
\subsubsection{Proof of  Proposition~\ref{prop:necessity_liph_intro}}
We are now ready to prove Proposition~\ref{prop:necessity_liph_intro}. For any parameters $L \ge  3\mu >0$, $\liph > 0$, let us consider the function $\tilde F$ defined by~\eqref{eqn:function_justify_liph} using parameters $\tilde \mu = \frac{2\mu}{L-\mu}$ and $\tilde L_H =\liph$. By Lemma~\ref{lem:function_justify_liph}, we know that $\tilde F$ satisfies Assumption~\ref{assum: non-quadratic} with parameters $\left(\frac{2\mu}{L-\mu},   \frac{2 L}{L-\mu}, \liph\right)$. Therefore, the function $F := \frac{L-\mu}{2} \cdot \tilde F$ satisfies Assumption~\ref{assum: non-quadratic} with parameters $(\mu, L,\liph)$. In addition, consider the sample objective function
\begin{align*}
	f(x,z) &=   \frac{L-\mu}{2} \cdot \tilde F(x) - \dotp{z,x} \\
	&= F(x) - \dotp{z,x}.
\end{align*}
Let us define the sample distribution  $P_\theta = N(\theta, \frac{(L-\mu)^2}{4} I)$ and  $\Theta = \{\frac{L-\mu}{2} \theta \colon \theta \in \tilde \Theta\}$, where $\tilde \Theta$ is the packing of the annulus from Proposition~\ref{prop:Fano}. For $\theta = \frac{L-\mu}{2}\tilde \theta$, we have
\[
F_{f,P_\theta}(x)
= \frac{L-\mu}{2}\bigl(\tilde F(x) - \langle \tilde \theta,x\rangle \bigr),
\]
so $F_{f,P_\theta}$ and $F_{\tilde f_0,\tilde P_{\tilde \theta}}$ have the same minimizer. Therefore, Proposition~\ref{prop:Fano}  implies that  if the number of samples $n \le   \frac{\log 2 \cdot d  L_H^2 }{2\cdot 144^2} $, then
\begin{align}\label{eqn:FANO}
	\inf_{\widehat x_n} \sup_{\theta \in  \Theta} \EE_{z_i\overset{\mathrm{iid}}{\sim}  P_\theta}[\|\widehat x_n(\{z_i\}_{i=1}^{n}) - x^\star(F_{ f,  P_\theta})\|_2^2] \ge \frac{36}{ \tilde L_H^2 \tilde \mu^2} = \frac{9 (L-\mu)^2}{ \liph^2  \mu^2}.
\end{align}
To complete the proof, we show that for any $\theta \in \Theta$, the instance $(f, P_\theta) \in \tilde \nclass\left(n,F, \frac{(L-\mu)^2}{4} I, \frac{36(L-\mu)}{\liph \mu}\right)$. To this end, we first note that $F_{ f,  P_\theta}$ has exact the same Hessian as $F$, and the gradient noise is always $N(0, \frac{(L-\mu)^2}{4} I)$. Moreover,  by Lemma~\ref{lem:far_noise_far_sol},
\begin{align*}
\|x^\star(F_{ f,  P_\theta}) - x^\star(F)\|_2 &= \|x^\star(F_{ f,  P_\theta})\|_2\\
&\le \frac{72}{\tilde L_H \tilde \mu} \\
&= \frac{36(L-\mu)}{\liph \mu}.
\end{align*}
So, $(f, P_\theta) \in \tilde \nclass\left(n,F, \frac{(L-\mu)^2}{4} I, \frac{36(L-\mu)}{\liph \mu}\right)$.
Consequently, the lower bound~\eqref{eqn:FANO} can be restated as 
\begin{align*}
	\inf_{\widehat x_n \in \widehat \cX_n} \sup_{(f,P)\in \tilde \nclass\left(n,F, \frac{(L-\mu)^2}{4} I, \frac{36(L-\mu)}{\liph \mu}\right)} \EE_{z_i\overset{\mathrm{iid}}{\sim} P}[\|\widehat x_n(\{z_i\}_{i=1}^{n}, f) - x^\star(F_{f,P})\|_2^2] \ge \frac{9 (L-\mu)^2}{ \liph^2  \mu^2}.
\end{align*}
\qed

\section{Auxiliary lemmas}

This appendix collects several auxiliary technical lemmas that are used in multiple proofs.
\begin{lemma}\label{lem:general_lambda_min}
	Suppose $F$ is $\mu$-strongly convex with respect to $\|\cdot\|$ and twice differentiable at a point $x \in \RR^d$,
	define
	$
	\lmin(x) := \inf_{\|w\|=1}\|\nabla^2 F(x) w\|_* .
	$
	Then $\lmin(x) \ge \mu$. Moreover, for any $v \in \RR^d$, we have 
	\begin{align*}
		\|\nabla^2 F(x)^{-1} v\| \le \frac{1}{\lmin(x)} \|v\|_* \le \frac{1}{\mu} \|v\|_*.
	\end{align*}
\end{lemma}

\begin{proof}
	Fix any $w\in\mathbb{R}^d$ and consider the univariate function $\phi(t):=F(x+t w)$.
	By $\mu$-strong convexity of $F$ (with respect to $\|\cdot\|$), $\phi$ is $\mu\|w\|^2$-strongly convex on $\mathbb{R}$, hence $\phi''(0)\ge \mu\|w\|^2$.
	Since $\phi''(0)=\langle w,  \nabla^2 F(x) w\rangle$, we have
	\[
	\langle w, \nabla^2 F(x) w\rangle \ge \mu\|w\|^2.
	\]
	By duality, $\langle w,  \nabla^2 F(x) w\rangle \le \|w\|\cdot \| \nabla^2 F(x) w\|_*$, so
	\[
	\| \nabla^2 F(x) w\|_* \ge \mu\|w\|.
	\]
	Taking an infimum over $\|w\|=1$ yields $\lmin(x) \ge \mu$. Since $ \nabla^2 F(x)$ is invertible, for any $v \in \RR^d$, we have 
		\begin{align*}
			\|\nabla^2 F(x)^{-1} v\| \le \frac{1}{\lmin(x)} \|v\|_* \le \frac{1}{\mu} \|v\|_*.
	\end{align*} 
\end{proof}

\begin{lemma}\label{lem: opnormineq}
	If $\norm{A}_2 \le \frac{1}{2}$, then we have
	$$
	\norm{(I+A)^{-1} - I}_2\le 2\norm{A}_2.
	$$
\end{lemma}
\begin{proof}
	Note that 
	\begin{align*}
		\norm{(I+A)^{-1} - I}_2 & =  \norm{-(I+A)^{-1} A}_2\\
		&\le \norm{(I+A)^{-1}} _2\norm{A}_2\\
		&\le 2 \norm{A}_2,
	\end{align*}
	where the last inequality follows from the fact that $\norm{(I+A)^{-1} }_2 \le \frac{1}{1-\norm{A}_2}$.
\end{proof}

The next lemma is a basic fact on matrix spectra, so we do not include a proof.    \begin{lemma}\label{lem:trace_bound_opnorm}
	Suppose that $A \in \RR^{d\times d}$ is symmetric positive definite and $B \in \RR^{d\times d}$. Then we have
	\begin{align*}
		\trace{AB} = \trace{BA} \le \norm{B}_2\trace{A}.
	\end{align*}
\end{lemma}

\begin{lemma}\label{lem:optimal_trace_when_close}
	Suppose that Assumption~\ref{assum: non-quadratic} holds.  For any $x \in \RR^d$, define 
	$$
	B(x) := \int_{0}^{1} \nabla^2 F(x^\star + t( x - x^\star)) dt.
	$$ 
	Then, for any $x$ such that $\|x -x^\star\| \le  \frac{1}{2\liph}$, we have 
	\begin{align*}
		\norm{B(x)^{-1} \nabla^2 F(x^\star)} \le 2.
	\end{align*}
\end{lemma}
\begin{proof}
	For any unit vector $v$, we have
	\begin{align*}
		\norm{\nabla^2 F(x^\star)^{-1} B(x) v} &\ge \norm{v} - \liph\|x -x^\star\| \|v\|\\
		&\ge \frac{1}{2} \|v\|,
	\end{align*}
	where the first inequality follows from Assumption~\ref{assum: non-quadratic}.
	Hence, $\norm{B(x)^{-1} \nabla^2 F(x^\star)} \le 2,$ as desired.
\end{proof}

\begin{lemma}\label{lem:small_noise_close_solution}
	Suppose that Assumption~\ref{assum: non-quadratic} holds.  For any $a \in \RR^d$, let $x(a)$ be the unique solution to the equation $\nabla F(x) = a$. For any $a$ satisfying $\|\nabla^2 F(x^\star)^{-1} a \| \le \frac{1}{4\liph}$, we have 
	$$
	\|x(a) - x^\star\| \le 2\|\nabla^2 F(x^\star)^{-1} a\| \le \frac{1}{2\liph}.
	$$
\end{lemma}

\begin{proof}
	Define the map $H$ via $H(x) = \nabla^2 F(x^\star)^{-1} \nabla F(x)$. Note that for any $x$ with $\|x -x^\star\|\le \frac{1}{2\liph}$, we have
	\begin{align}
		\|H(x)\|&= \norm{\int_{0}^{1}\nabla H(x^\star + t(x-x^\star))(x-x^\star) dt } \notag\\
		&\ge  \norm{x- x^\star}- \frac{\liph}{2}\norm{x-x^\star}^2\label{eqn: H_lb_eq1}\\
		&\ge \frac{1}{2} \norm{x-x^\star} \label{eqn: H_lb_eq2},
	\end{align}
	where the estimate~\eqref{eqn: H_lb_eq1} follows since $\nabla H(x^\star) = I$ and Assumption~\ref{assum: non-quadratic} , and the estimate~\eqref{eqn: H_lb_eq2} follows from $\|x-x^\star\| \le \frac{1}{2\liph}$. As a result,
	$$
	\inf_{x \colon \|x-x^\star\| = \frac{1}{2\liph}} \|H(x)\| \ge \frac{1}{4\liph}.
	$$
	Since $H$ is a $C^1$-diffeomorphism, for any vector $\|y\| \le \frac{1}{4\liph}$, we must have $\|H^{-1}(y) - x^\star\| \le \frac{1}{2\liph}$. Note that $x(a) = H^{-1}(\nabla^2 F(x^\star)^{-1} a)$, for any $a$ with $\|\nabla^2 F(x^\star)^{-1} a\| \le \frac{1}{4\liph}$, we have the bound $\|x(a) - x^\star\|\le \frac{1}{2\liph}$. By estimate~\eqref{eqn: H_lb_eq2} again,  we have
	\begin{align*}
		\norm{x(a) - x^\star} \le 2 \|H(x(a))\| = 2\| \nabla^2 F(x^\star)^{-1} a\| \le \frac{1}{2\liph}.
	\end{align*}
	The proof is thus complete.
\end{proof}

Our next lemma is a basic fact about sub-exponential random variables.
\begin{lemma}[{\cite[Proposition 2.9]{wainwright2019high}}]\label{lem:subexp_tail}
	Let $X$ be a sub-exponential random variable with parameters $(\nu^2, \alpha)$. Then
	\begin{align*}
		P(|X - \EE[X]| \ge t) \le 2e^{-\frac{1}{2} \min\{\frac{t^2}{\nu^2}, \frac{t}{\alpha}\}}.
	\end{align*}
\end{lemma}

\begin{lemma}\label{lem:second_moment_sub_exponential}
	Let $X$ be a zero-mean sub-exponential random variable with parameters $(\nu^2, \alpha)$. Then
	\begin{align*}
		\EE[X^2] \le \nu^2.
	\end{align*}
\end{lemma}
\begin{proof}
	A result similar to this lemma has appeared in ~\cite[Exercise 2.5]{wainwright2019high}. We provide a proof here for completeness. Note that $e^{tX} -1 - tX \ge 0$ for any $t \in \RR$ and 
	\begin{align*}
		X^2 = \lim_{t \rightarrow 0} \frac{e^{tX} - 1 - tX}{\frac{1}{2}t^2}.
	\end{align*}
	By Fatou's lemma,
	\begin{align*}
		\EE[X^2] &\le \liminf_{t \rightarrow 0} \EE\left[\frac{e^{tX} - 1 - tX}{\frac{1}{2}t^2}\right] \\
		&\le  \liminf_{t \rightarrow 0} \frac{e^{\frac{t^2\nu^2}{2}} -1}{\frac{1}{2}t^2}\\
		&= \nu^2.
	\end{align*}
\end{proof}

\begin{lemma}\label{lem:norm_subGaussian}
	Let $X \in \RR^d$ be a zero-mean sub-exponential random vector with parameter $(\nu^2, \alpha)$. Then, for any $t \ge 0$, we have
	\begin{align*}
		P[\|X\|_*\ge t] \le 2 e^{-\frac{t^2}{8\nu^2} + 2d} + 2 e^{-\frac{t}{4\alpha} + 2d}
	\end{align*}
\end{lemma}
\begin{proof}
	The proof largely follows~\cite[Lemma 1]{jin2019short}, and we provide it here for completeness. Let $\{w_i\}_{i \in I}$ be a $\frac{1}{2}$-net of the unit sphere $\mathbb{S}^{d-1}$ under the standard Euclidean metric. Let $\| \cdot\|$ be induced by the symmetric positive definite matrix $Q$, i.e., $\|x\| = \sqrt{\dotp{x,Qx}}$. Define $v_i := Q^{-1/2} w_i$. It is straightforward to verify that \(\{v_i\}_{i \in I}\) is a \(\frac{1}{2}\)-net of the unit sphere
	\[
	S := \{x : \|x\| = 1\}
	\]
	with respect to the metric \(d(x,y) = \|x-y\|\). By the definition of sub-exponential random vector and Lemma~\ref{lem:subexp_tail}, for each $v_i$, we have
	\begin{align}\label{eqn:tail_bound_one_direction}
		P(\dotp{ v_i,  X} \ge t) \le 2e^{-\frac{1}{2} \min\{\frac{t^2}{\nu^2}, \frac{t}{\alpha}\}}.
	\end{align}
	Let $v(X) = \argmax_{\|v\| =1} \dotp{v,X}$. By definition of $\frac{1}{2}$-net, there exists $i(X)$  that $\|v(X) - v_{i(X)}\| \le \frac{1}{2}.$ As a consequence, we have 
	\begin{align*}
		\|X\|_*& =  \dotp{v_{i(X)}, X} + \dotp{v(X) - v_{i(X)}, X} \\
		& \le \dotp{v_{i(X)}, X}  + \frac{\|X\|_*}{2}.
	\end{align*} 
	Therefore,  $\|X\|_* \le 2 \dotp{v_{i(X)}, X}$. Note also that by~\cite[Lemma 5.2]{vershynin2010introduction}, the cardinality of $\{w_i\}_{i \in I}$  is upper bounded by $e^{2d}$, and so is $\{v_i\}_{i \in I}$. As a result, we have 
	\begin{align*}
		P(\|X\|_* \ge t) &\le P( \dotp{v_{i(X)}, X} \ge t/2)\\
		&\le e^{2d} P(\dotp{v_1, X} \ge t/2)\\
		&\le 2e^{ - \min\{\frac{t^2}{8\nu^2}, \frac{t}{4\alpha}\} + 2d} \\
		&\le 2e^{ - \frac{t^2}{8\nu^2}+ 2d} + 2e^{ - \frac{t}{4\alpha}+ 2d},
	\end{align*}
	where the second inequality follows from the union bound and the third inequality follows from~\eqref{eqn:tail_bound_one_direction}.
\end{proof}

The next lemma is a basic fact about expectations.
\begin{lemma}[Moment and tails]\label{lem: moment_and_tail}
	Let $X$ be a nonnegative random variable with finite $p$-th moment. Then
	\begin{align*}
		\EE[X^p] = \int_{0}^{\infty} p \cdot t^{p-1} P(X \ge t) dt.
	\end{align*}
\end{lemma}

\begin{lemma}\label{lem:truncted_second_moment_light_tail}
	Let $X \in \RR^d$ be a zero-mean sub-exponential random vector with parameter $(\nu^2,\alpha)$. Let $c > 0$ be a constant. We have
	\begin{align*}
		\EE[\|X\|_*^2 1_{\|X\|_* \ge c}] \le (2c^2+ 16 \nu^2) e^{-\frac{c^2}{8\nu^2} + 2d} + (4c^2+96\alpha^2)  e^{-\frac{c}{4\alpha} + 2d}.
	\end{align*}
\end{lemma}
\begin{proof}
	By Lemma~\ref{lem:norm_subGaussian}, for any $ t \ge 0$,  we have
	\begin{align}\label{eqn:norm_tail_bound}
		P[\|X\|_* \ge t] \le 2 e^{-\frac{t^2}{8\nu^2} + 2d} + 2 e^{-\frac{t}{4\alpha} + 2d}
	\end{align}
	By Lemma~\ref{lem: moment_and_tail}, the bound~\eqref{eqn:norm_tail_bound},  and some calculus, we have
	\begin{align*}
		\EE[\|X\|_*^2 1_{\|X\|_* \ge c}] &= \int_{0}^{c} 2t P(\|X\|_* \ge c) dt + \int_{c}^\infty 2tP(\|X\|_* \ge t) dt \\
		&\le  2c^2 e^{-\frac{c^2}{8\nu^2} + 2d} +  2c^2 e^{-\frac{c}{4\alpha} + 2d} + \int_{c}^{\infty}  4t (e^{-\frac{t^2}{8\nu^2} + 2d} +  e^{-\frac{t}{4\alpha} + 2d}) dt \\ 
		&= (2c^2 + 16\nu^2)  e^{-\frac{c^2}{8\nu^2} + 2d} + (2c^2+ 16 \alpha c+ 64\alpha^2)  e^{-\frac{c}{4\alpha} + 2d}.\\
		&\le (2c^2 + 16\nu^2)  e^{-\frac{c^2}{8\nu^2} + 2d} + (4c^2+ 96\alpha^2)  e^{-\frac{c}{4\alpha} + 2d}.
	\end{align*}
\end{proof}

\end{document}